\documentclass[11pt, reqno]{amsart}
\usepackage[T1]{fontenc}
\usepackage[english]{babel}
\usepackage{amssymb}
\usepackage{amsmath,esint}
\usepackage[shortlabels]{enumitem}
\usepackage{amsthm}
\usepackage{amscd,mdwlist}
\usepackage[margin=1in]{geometry}
\usepackage{cite}
\usepackage{bbm, mathrsfs,pgf,tikz}
\usepackage{thmtools}
\usepackage[backgroundcolor=white, bordercolor=blue,
linecolor=blue]{todonotes}
\usepackage{url,comment}
\usepackage{dsfont}
\usepackage{cases}
\usetikzlibrary{arrows}
\usepackage[pagebackref, pdfpagelabels, plainpages=false]{hyperref}
\usepackage[normalem]{ulem}
\usepackage{soul}

\hypersetup{colorlinks=true,linkcolor=teal,filecolor=teal,urlcolor=teal,citecolor=blue}

\usepackage{color}
  
\theoremstyle{plain}
\declaretheorem[title=Theorem, parent=section]{theorem}
\declaretheorem[title=Lemma,sibling=theorem]{lemma}
\declaretheorem[title=Proposition,sibling=theorem]{proposition}
\declaretheorem[title=Corollary,sibling=theorem]{corollary}

\theoremstyle{definition}
\declaretheorem[title=Definition,sibling=theorem]{definition}
\declaretheorem[title=Remark,sibling=theorem]{remark}
\declaretheorem[title=Remark, numbered=no]{remark*}
\declaretheorem[title=Example, sibling=theorem]{example}

\numberwithin{equation}{section}

\DeclareMathOperator*{\argmin}{argmin}

\DeclareMathOperator{\supp}{supp}

\DeclareMathOperator{\loc}{loc}

\DeclareMathOperator{\pv}{\operatorname{p.\!v.}}

\newcommand{\cF}{\mathcal{F}}
\newcommand{\cG}{\mathcal{G}}
\newcommand{\cE}{\mathcal{E}}

\DeclareMathOperator{\cP}{\mathcal{P}}
\DeclareMathOperator{\cV}{\mathcal{V}}
\newcommand{\cH}{\mathcal{H}}

\DeclareMathOperator{\cJ}{\mathcal{J}}

\DeclareMathOperator{\fD}{\mathscr{D}}

\newcommand{\iil}{\iint\limits}
\DeclareMathOperator{\R}{\mathbb{R}}
\DeclareMathOperator{\N}{\mathbb{N}}
\renewcommand{\d}{\mathrm{d}}
\renewcommand{\div}{\operatorname{div}}

\newcommand{\eps}{\varepsilon}

\newcommand{\Hnu}{H_{\nu}(\mathbb{R}^d)}
\newcommand{\Hnud}{\dot{H}_{\nu}(\mathbb{R}^d)}

\newcommand{\Hnup}{H^{\psi}(\mathbb{R}^d)}
\newcommand{\Hnupd}{\dot{H}^{\psi}(\mathbb{R}^d)}

\newcommand{\Hnuw}{H^{\widetilde{\psi}}(\mathbb{R}^d)}
\newcommand{\Hnuwd}{\dot{H}^{\widetilde{\psi}}(\mathbb{R}^d)}

\newcommand{\Hnus}{H^{\psi^*}(\mathbb{R}^d)}
\newcommand{\Hnusd}{\dot{H}^{\psi^*}(\mathbb{R}^d)}

\newcommand{\Hnuid}{\dot{H}^{\psi^{-1}}(\mathbb{R}^d)}
\newcommand{\Hnui}{H^{\psi^{-1}}(\mathbb{R}^d)}

\newcommand{\vertiii}[1]{{\left\vert\kern-0.25ex\left\vert\kern-0.25ex\left\vert #1 \right\vert\kern-0.25ex\right\vert\kern-0.25ex\right\vert}}

\addto\extrasenglish{

}

\begin{document}

\title{Gradient Flow Solutions For Porous Medium Equations with Nonlocal L\'{e}vy-type Pressure\\
\url{https://doi.org/10.1007/s00526-025-02942-6}}

\author[Foghem]{Guy Foghem}
\author[Padilla-Garza]{David Padilla-Garza}
\author[Schmidtchen]{Markus Schmidtchen}

\address[Foghem]{{\small Fakult\"{a}t f\"{u}r Mathematik Institut f\"{u}r wissenschaftliches Rechnen, TU Dresden Zellescher Weg 23/25, 01217, Dresden, Germany. Email: guy.foghem@tu-dresden.de}}
\address[Padilla-Garza]{{\small Fakult\"{a}t f\"{u}r Mathematik Institut f\"{u}r wissenschaftliches Rechnen, TU Dresden Zellescher Weg 23/25, 01217, Dresden, Germany. Email: david.padilla-garza@tu-dresden.de}}
\address[Schmidtchen]{{\small Fakult\"{a}t f\"{u}r Mathematik Institut f\"{u}r wissenschaftliches Rechnen, TU Dresden Zellescher Weg 23/25, 01217, Dresden, Germany. Email: markus.schmidtchen@tu-dresden.de}}

\begin{abstract}
We study a porous medium-type equation whose pressure is given by a  nonlocal L\'{e}vy operator associated to a symmetric jump L\'{e}vy kernel. The class of nonlocal operators under consideration appears as a generalization of the classical fractional Laplace operator. For the class of L\'evy operators, we construct weak solutions using a variational minimizing movement scheme. The lack of interpolation techniques is ensued by technical challenges that render our setting more challenging than the one known for fractional operators.
\end{abstract}


\maketitle

\tableofcontents

\section{Introduction}\label{sec:intro}
This paper is dedicated to studying the nonlocal continuity equation
\begin{align}
    \label{eq:main_eqn}
    \begin{cases}
    \partial_t u - \div(u\nabla v)=0 &\quad\text{in }\,\, \R^d\times (0, \infty),\\
    L v = u &\quad\text{in }\,\, \R^d\times (0, \infty),\\
    u(0)=u_0&\quad\text{in }\,\, \R^d\times \{0\},  
    \end{cases}
\end{align}
for some initial data $u_0\in \mathcal{P}_2(\R^d)$, where $\mathcal{P}_2(\R^d)$ is space of probability measures with finite second moment where $u=u(x,t)$ denotes the density at a point $x\in \R^d$ at time $t>0$. Here, the pressure, $v=v(x,t)$, is coupled to the density via a  linear, nonlocal, pseudo-differential operator
\begin{align}
    \label{eq:intro-pressure-relation}
    L v = u.
\end{align}
Yet, in all generality, little attention has been paid to studying more general pseudo-differential operators. In this paper, we want to address this dearth in the literature and assume the operator $L$ be a symmetric \emph{integrodifferential operator of L\'{e}vy type}, \emph{i.e.},
\begin{align}
    \label{eq:levy-operator}
    Lu (x):=2 \pv \int_{\R^d} (u(x)-u(y))\nu(x-y)\d y, 
\end{align}
where the kernel $\nu \geq 0$ is assumed to be the Lebesgue density of a \emph{symmetric L\'{e}vy measure}, that is $\nu$ satisfies
\begin{align}
    \tag{$L$}
    \label{eq:levy-cond}
    \nu(h)=\nu(-h), \qquad\text{and}\qquad \int_{\R^d} (1\land|h|^2)\nu(h)\d h<\infty.
\end{align}
Here, we use the notation $a\land b$ for $\min(a,b),$ $a,b\in \R$. This also includes the fractional Laplace operator, $L=(-\Delta)^{s}$, which has received much attention in the last decade. L\'{e}vy operators as in Eq. \eqref{eq:levy-operator} arise naturally in probability theory as the generator of L\'{e}vy processes with pure jumps, whose jumps interactions are regulated by the measure $\nu(h)\d h$, \emph{cf.} \cite{Sat13, App09, Ber96} for more details on L\'{e}vy processes.  The study of nonlocal problems driven by L\'{e}vy operators has recently gained increasing popularity. For some recent results in the study of problems involving pseudodifferential operators of L{\'e}vy type, we refer to \cite{guy-thesis} for the study of nonlocal function spaces, \cite{DFK22,Rut18} for the study of Dirichlet type problem, \cite{FK22,Fog23s}  for the study of (non)linear Neumann type problems and their connection with the corresponding local problems and \cite{BoEn23} for the study of generalized porous medium equations; see also the references therein. 

Allowing for L{\'e}vy operators in Eq.~\eqref{eq:intro-pressure-relation} comes at the price of losing interpolation inequalities, which used to play a crucial role in known results. In this work, we remedy the lack of interpolation techniques by obtaining surrogate estimates from the energy and its associated energy-dissipation functional as well as the study of fine properties of the symbol of the nonlocal operator $L$. Then, in order to construct solutions to Eq. \eqref{eq:main_eqn}, we employ a minimizing movement scheme \emph{\`a la} De Giorgi \cite{DeG1993}. 

In fact, in the last two decades, this topic has experienced a renaissance due to the intimate link between the continuity equation, 
\begin{subequations}
\label{eq:intro-gradflow-W2}
\begin{align}
    \partial_t u+ \div(u \mathbf{v}) =0, 
\end{align}
governing the evolution of the density, $u$, and absolutely continuous curves in the set of probability measures, \cite{AGS08}. 
At least formally, when
\begin{align}
    \mathbf{v} = - \nabla \frac{\delta \cE}{\delta u}.
\end{align}
\end{subequations}
solutions to the continuity equation can be understood as gradient flows in the set of probability measures, \cite{AGS08, JKO98, Ott2001}. In the subsequent discussion, three choices of functionals, $\cE \in \{\cF, \cJ, \cH\}$, will play a prominent role:
\begin{align*}
    \cJ_m(u) := \frac{1}{m-1} \int_{\R^d} u^m \d x
\end{align*}
which gives rise to the porous medium equation \cite{Vaz07, Ott2001}, as well as the entropy
\begin{align*}
    \cH(u) := \int_{\R^d} u \log u \d x
\end{align*}
which gives rise to the linear diffusion equation \cite{Ott2001, JKO98}, and the nonlocal interaction energy
\begin{align}
    \label{eq:intro-NL-energy}
    \cF(u) := \frac12 \int u K * u \d x,
\end{align}
for some kernel $K$, \cite{CDFS2011}. In our work, formally, the kernel $K$ can be related to a symmetric L{\'e}vy measure, $\nu$, mentioned above, denoted by $K_\nu$, see for example \cite{EFJ21} where this type of kernels are considered. In the case of the inverse fractional Laplacian, $K_\nu$, coincides with the Riesz kernel (for instance, see \cite{Ste70}) and the resulting equation reads
\begin{align*}
    \partial_t u = \nabla \cdot (u \nabla p),
\end{align*}
with $p = K_\nu * u = (-\Delta)^{-s} u$, and is referred to as porous medium equation \emph{with fractional pressure} in the literature. 
It is a special instance of the equation
\begin{align*}
    \partial_t u = \nabla \cdot (u \nabla^{\alpha} u^{m-1}),
\end{align*}
which was derived as a model for the dynamics of dislocations in a crystal \cite{BKM2010}. For this class of equations, the existence of self-similar profiles was shown in \cite{BIK2011, BIK2015}. The special case $m=2$ was studied in \cite{CV2011} as a generalization of the classical porous medium equation, where the pressure-density closure relation is local and given by $p = u^\gamma$, for some $\gamma \geq 1$. More precisely, the existence of weak solutions in the sense
\begin{align*}
    \int_0^T\int_\Omega u (\partial_t \varphi - \nabla \varphi \cdot \nabla K_s* u) \d x \d t + \int_\Omega \varphi(x,0) u_0(x) \d x = 0,
\end{align*}
where $K_s$ is the Riesz kernel. The argument hinges on a sophisticated approximation argument that consists of adding small viscosity, mollifying the Riesz kernel, and removing the degeneracy in the mobility. Subsequently, the authors show H\"older-regularity and boundedness of weak solutions in \cite{CSV2013} and extended the existence result to a wider class of initial data \cite{SV2014}. The limit $s\to 1$, where $K$ becomes the Newtonian potential, has been considered in \cite{LZ2000, SV2014}. The other limit case, $s \to 0$, corresponding to the local porous medium equation as the limit of a porous medium equation with fractional pressure has been established in \cite{LMS2018}. We also refer to \cite{LM2001,BE2022, HDPP2023} for nonlocal approximation of the porous medium equation with exponent two and \cite{CEW2023} for arbitrary exponent. In these works, however, the nonlocal approximation is smooth and integrable, unlike the Riesz kernel. Similarly, \cite{DDMS2023, ES2023, ambrosio2008gradient, ambrosio2011gradient}, consider a system of porous medium type with the pseudodifferential operator is $K_s = (\mathrm{id} - s \Delta )^{-1}$. In this vein, it is worth mentioning that the works \cite{Fog23s, FK22, guy-thesis,FGKV20,Voi17} discuss the nonlocal-to-local limit for `elliptic' problems in the context of L{\'e}vy kernels. 

Returning to Eq. \eqref{eq:main_eqn}, let us recall that the operator $L$ is a pseudo-differential operator, and we can associate a symbol to it, which is given by
$$
    \psi(\xi)= 2\int_{\R^d} (1-\cos(\xi\cdot h))\nu(h)\d h,
$$ 
that is, in other words, 
$$
    \widehat{Lu}(\xi)=\psi(\xi)\widehat{u}(\xi),
$$
for any $\xi\in \R^d$, and any Schwartz function, $u\in \mathcal{S }(\R^d)$. Here, we have used the common notation,  $\widehat{u}$, to denote the Fourier transform of $u$, \emph{\emph{i.e.}},
\begin{align*}
    \widehat{u}(\xi) = \frac{1}{(2\pi)^{d/2}} \int_{\mathbb{R}^d}e^{-i\xi \cdot x} u(x)\,\d x, 
\end{align*}
for any $\xi\in \mathbb{R}^d$. Therefore, formally, we may choose $v= L^{-1}u= K_\nu*u$, so that $\widehat{L^{-1}u}(\xi)= \psi^{-1}(\xi) \widehat{u}(\xi)$, where $K_\nu$ is a Borel measure potential (defining a tempered distribution) whose Fourier transform is given by $\widehat{K_\nu} = \psi^{-1}$,  motivating the interest in the nonlocal energy functional, Eq. \eqref{eq:intro-NL-energy}. Throughout this work, we shall write, by an abuse of notation, $v = L^{-1}u$ to refer to $\widehat v(\xi) = \psi^{-1}(\xi) \widehat u(\xi)$. In light of this formal link, we can cast Eq. \eqref{eq:main_eqn} into the form of a formal 2-Wasserstein gradient flow, \emph{cf.} Eq. \eqref{eq:intro-gradflow-W2}, for the nonlocal energy given in Eq. \eqref{eq:intro-NL-energy}, \emph{\emph{i.e.}},
\begin{align*}
    \partial_t u =  \div\left(u \nabla \frac{\delta \mathcal F}{\delta u}\right).
\end{align*} 
Moreover, taking into account \eqref{eq:intro-NL-energy} one obtains the natural quadratic energy 
\begin{align*}
\cF(u)= \frac12\int_{\R^d} |\widehat{u}(\xi)|^2\psi^{-1}(\xi) \d \xi. 
\end{align*}
\subsection{Relation to Fractional Laplacian}
The epitome of a L{\'e}vy operator is obtained by setting
$$
    \nu(h) =\frac{C_{d,s}}{2}|h|^{-d-2s },
$$
for $s \in (0,1)$, and the operator resulting from this choice is the well-known fractional Laplace operator,
\begin{align}
    \label{eq:frac-laplace-operator}
    (-\Delta)^{s}u(x)= C_{d,s} \pv \int_{\R^d}\frac{(u(x)-u(y))}{|x-y|^{d+2s}} \d y,
\end{align}
for any $x\in \R^d$, one of the most widely studied integrodifferential operators, see, for instance, \cite{BV16, CS07, Kw17, guy-thesis, WYP20,Ga19}.
Here, the constant $C_{d,s}$ guaranties the relation 
$$
    \widehat{(-\Delta)^{s} u}(\xi)= |\xi|^{2s}  \widehat{u}(\xi),
$$
for all $u$ in $C^\infty_c(\R^d)$, and it can be shown that it is given by 
\begin{align}\label{eq:normalizing-cst}
C_{d,s }= 
    \frac{2^{2s}\Gamma\big(\frac{d+2s }{2}\big)}{\pi^{d/2}\big|\Gamma(-s)\big|},
\end{align}
see \cite{guy-thesis,Buc16,SiTo10,ArSm61}.
Interestingly, the constant $C_{d,-s}$ guarantees a similar relation for the inverse of the fractional Laplacian $(-\Delta)^{-s}$, also known for the general range $s\in (0, d/2)$ as the Riesz potential. Namely, for $s \in (0, d/2)$, we have 
\begin{align}
    \label{eq:fourier-riesz-potential}
    \widehat{ (-\Delta)^{-s}u\,} (\xi)= |\xi|^{-2s}\widehat{u}(\xi),
\end{align}
in the distributional sense for all $u\in \mathcal{S}(\R^d)$ if and only if 
\begin{align}
    \label{eq:riesz-operator}
    (-\Delta)^{-s}u(x)= K_s*u(x)= C_{d,-s}\int_{\R^d}\frac{u(y)}{|x-y|^{d-2s}} \d y,
\end{align}
for any $x\in\R^d$, see, for instance, \cite{Ries38,Ste70, Put04,ArSm61}. Here $K_s(x)= C_{d,-s} |x|^{2s-d}$ is the Riesz kernel, and the constant $C_{d,-s}$ is exactly given by the formula in Eq. \eqref{eq:normalizing-cst}. This corresponds directly to the class of fractional pressure relations considered in \cite{LMS2018}, and we refer to recent mean-field approximations of these nonlocal systems \cite{RS2023, NRS2022}. Let us foreshadow that, below, we shall provide further, non-trivial, examples of L\'{e}vy operators to which our existence results apply.

\subsection{Our contribution}
In this article, we focus on problems of the form \eqref{eq:main_eqn}, for radial L{\'e}vy kernels, $\nu$. The recent work sparked our interest in this equation \cite{LMS2018}, in which the special case of Eq. \eqref{eq:main_eqn}, with $L = (-\Delta)^{s}$, referred to as `\emph{porous medium equation with fractional pressure}'
was investigated. Upon casting the problem into a gradient-flow setting, the authors prove an existence result of weak solutions and furnish decay estimates for $L^p$-norms of solutions to the evolution problem, Eq.  \eqref{eq:main_eqn}, which was first studied in \cite{SDV19} for the general porous medium equation with fractional pressure. Following the strategy of \cite{LMS2018}, we construct weak solutions to Eq. \eqref{eq:main_eqn} employing the minimizing movement scheme
\begin{align*}
    u_{\tau}^{k} = \argmin_{u \in \mathcal{P}_2(\R^d)} \left\{ \frac{1}{2 \tau} W^{2}\left( u,  u_{\tau}^{k-1} \right) + \cF(u)
    \right\},
\end{align*}
see Definition \ref{def:discrete}, below. In our work, we prove existence of solutions for a much wider class of L\'evy kernels which comes at the price of losing the homogeneity of the Riesz kernel associated to the fractional Laplacian. Mainly, there are two main difficulties. The first one lies in the derivation of the Euler-Lagrange equations, which is significantly more challenging for more general kernels. The second challenge is to obtain the appropriate compactness of the sequence obtained from the minimizing movement scheme, which is needed to identify its limit as a weak solution. Specifically, interpolation inequalities break down in the nonlocal energy spaces that act as surrogates of the usual fractional Sobolev spaces. A drawback of our new methodology is that we are currently unable to establish decay estimates similar to those in \cite{LMS2018}. Indeed, \cite{LMS2018} makes heavy use of interpolation inequalities and Sobolev inequalities. We point out that a Sobolev-type inequality associated to L\'evy kernels was established in \cite{Fog21b}. Therein, the Sobolev exponent is obtained in terms of a Young function, giving rise to embeddings into Orlicz-type spaces. However, the resulting Young function lacks sufficient regularity to use it as a generalized entropy. Such inequalities for L\'evy kernels are interesting in their own right, would have several other applications, and are the object of ongoing investigations. In particular, for special (not necessarily Riesz-type) kernels, the resulting Young function may already have sufficient regularity. This is kept as another avenue to explore in the future.

\subsection{Main result and additional comment}
Before stating this paper's main result, let us briefly introduce some necessary notation. Throughout, we consider spaces 
$$
    \dot{H}^\phi(\R^d) =\{u\in \mathcal{S}'(\R^d)\,:\, \widehat{u} \phi^{1/2}\in L^2(\R^d)\},
$$ and 
$$
    H^\phi(\R^d)= \dot{H}^\phi(\R^d)\cap L^2(\R^d),
$$ 
which act as generalizations of the usual (homogeneous) fractional Sobolev spaces. Here, $\phi: \R^d\setminus \{0\}\to (0,\infty)$ is a symmetric and continuous function which we refer to as `symbol', and we point the reader to Section \ref{sec:function-spaces}, for further details on their construction. Specifically, the \emph{special symbol $($or regularizing symbol$)$} 
$$\widetilde{\psi} (\xi) = |\xi|^2\psi^{-1}(\xi),$$ 
will play a prominent role as it determines the regularity of the weak solution to Eq. \eqref{eq:main_eqn}. We shall assume on the symbol that there exists a constant $c_\nu>0$ such that, for all $\xi\in \R^d$, there holds the lower-bound condition 
\begin{align}
    \tag{$C_\nu$}
    \label{eq:lower-bound-cond}
   \psi(\xi)\geq c_\nu(1\land|\xi|^2),
\end{align}
which is required to establish the convergence of the velocity, point $(iii)$, in the main theorem, Theorem \ref{thm:main-theorem}. We show in Theorem \ref{thm:symbol-bound} that Condition \eqref{eq:lower-bound-cond} holds whenever $\nu$ is unimodal, radial, and nontrivial. Here, we call a L\'{e}vy kernel $\nu$ \emph{unimodal} if there is constant $c>0$ such that 
\begin{align}
    \label{eq:nu-unimodal-bis}
     \nu(x)\leq c\nu(y),\qquad\text{whenever $|x|\geq |y|$.}
\end{align}
 Having introduced all necessary notations, we can now present the main result of this paper.

\begin{theorem}
\label{thm:main-theorem}
Assume $u_{0} \in \Hnuid \cap \mathcal{P}_2(\R^d)$.  Consider the special symbol $\widetilde{\psi} (\xi) = |\xi|^2\psi^{-1}(\xi)$. The following hold.
\begin{enumerate}[$(i)$]
    \item \emph{\textbf{Existence of discrete solutions}}. There is a unique sequence $ (u_{\tau}^{k})_k$ with $u_{\tau}^{k}\in \Hnuwd\cap \Hnuid\cap \cP_2(\R^d)$ satisfying the minimization problem \eqref{eq:discrete}.
    
    \item\emph{\textbf{Convergence and regularity}}. Define the discrete curve $u_{\tau} (t) =  u_{\tau}^{\lceil t/\tau\rceil}$
     with $u_{\tau}(0) = u_{\tau}^{0}$. There exists a curve 
     $u \in AC^{2} ([0, \infty), (\mathcal{P}_2(\R^d), W)) $ and a subsequence $(\tau_n)_n$ tending to $0$, such that 
    \begin{align*}
        u_{\tau_{n}}({t}) \to  u(t) \text{ narrowly as } n \to \infty \text{ for each } t \in \R.
    \end{align*}
    Let $\cH$  be the standard Boltzmann entropy. If $u_0\in D(\cH)$ then
    \begin{align*}
        u_{\tau_{n}} \rightharpoonup  u \text{ weakly in } 
        L^{2}( 0, T, \Hnuw ) \text{ as } n \to \infty.
    \end{align*}
    In addition, if we assume that 
    \begin{align}
        \label{eq:compactness-ass}
        \sup_{\xi\in \R^d}\frac{1}{\widetilde{\psi}(\xi)}
        |e^{i\xi\cdot h}-1|^2= \sup_{\xi\in \R^d}\frac{\psi(\xi)}{|\xi|^2}
        |e^{i\xi\cdot h}-1|^2\xrightarrow{|h|\to0}0, 
    \end{align}
    then we have 
    \begin{align*}
        u_{\tau_{n}} \to  u\,  \text{ strongly in }\,  L^{2}( 0, T, L^{2}_{\rm loc} (\R^{d}) ) \text{ as } n \to \infty.
    \end{align*}

    \item \emph{\textbf{Convergence of  the velocity}}  Define $v_{\tau}(t)= L^{-1}u_{\tau}(t) $ and  $v(t) = L^{-1}u(t)$. Assume that condition \eqref{eq:lower-bound-cond} holds and that $u_0\in D(\cH)$. Then we have $\nabla v \in  L^{2}( 0, T, L^{2} (\R^{d}) ),$ and
    \begin{align*}
        \nabla v_{\tau_{n}} \rightharpoonup  \nabla v 
        \, \text{ weakly in } L^{2}( 0, T, L^{2} (\R^{d}) ) \text{ as } n \to \infty.
    \end{align*}
            
    \item \label{item}\emph{\textbf{Solution of the limiting equation.}} Assume that the conditions \eqref{eq:compactness-ass} \eqref{eq:lower-bound-cond} are satisfied. Assume also that $u_0\in D(\cH)$. Moreover, assume the symbol $\widetilde{\psi}$ is associated with a unimodal L{\'e}vy kernel $\widetilde{\nu}$ satisfying the following condition 
    \begin{align}
        \label{eq:double-cond-origin-intro}
        \text{For any $0<\lambda<1$ there is $c_\lambda>0$ such that $\widetilde{\nu}(\lambda h)\leq c_\lambda\widetilde{\nu} (h)$, whenever $|h|\leq 1$}.
    \end{align} 
    Then the limiting curve $u$ is a weak solution of the Eq. \eqref{eq:main_eqn}, viz., the functions $u,v$ defined by items $(ii)$ and $(iii)$ satisfy 
    \begin{align}
        \label{eq:def:weak-solutions}
        \int_{0}^{\infty} \int_{\R^{d}} (\partial_{t} \varphi - \nabla \varphi \cdot \nabla v) u \, \d x \, \d t =0, \, \quad \forall \varphi \in C^{\infty}_{c} (\R^{d}\times (0, \infty)).
    \end{align}
    \item \emph{\textbf{Energy dissipation inequality.}}
   If  $u$ satisfies the weak formulation from point $(iv)$ then
    \begin{align*}
        \cF(u(T)) + \int_{0}^{T} \int_{\R^{d}} u(t)  \big| \nabla v (t) \big|^{2} \, \d x \, \d t \leq \cF(u_{0}),
    \end{align*}
    \item \emph{\textbf{ Entropy and $L^p$ Boundedness.}} 
    If $u_0\in D(\cH)$ then $$\cH(u(t))\leq \cH(u_0).$$ 
    If in addition, $\widetilde{\psi}$ is the symbol associated with a  L\'{e}vy-integrable kernel $\widetilde{\nu}$ then for $1< p<\infty$ and $u_0\in L^p(\R^d)$ we have 
    \begin{align*}
        \|u(t)\|_{L^p(\R^d)}\leq \|u_0\|_{L^p(\R^d)}.
    \end{align*}
\end{enumerate}
\end{theorem}
For the readers' convenience, we give a short overview of where to find the individual statements. Item $(i)$ follows from Theorem \ref{thm:lower-semicontinuous}, which establishes the lower semi-continuity of the Moreau-Yosida penalization
\begin{align*}
u \mapsto \frac{1}{2\tau} W^2(u,u_*)+\mathcal{F}(u), \quad u_*\in \cP_2(\R^d). 
\end{align*}
 The narrow convergence in $(ii)$ is a consequence of Theorem \ref{thm:mini-mvt-scheme}, and the weak and strong convergence is proven in Theorem \ref{thm:convergence-discrete}. In Theorem \ref{thm:weak-solution}, we identify the limit curve as a weak solution, thereby establishing point $(iv)$. The dissipation of the energy, $(v)$, is proven in Theorem \ref{thm:energy-disspation-est}, and finally, the control of convex entropies, $(vi)$, is established in Theorem \ref{thm:entropy-boundedness} and Theorem \ref{thm:lp-boundedness}. \\

The rest of this paper is organized as follows. In Section \ref{sec:function-spaces}, we introduce fundamental results on L\'{e}vy operators, different formulations thereof, and their link to stochastic jump processes. Furthermore, we introduce energy spaces associated to this class of kernels and discuss compactness criteria in these spaces. We conclude Section \ref{sec:function-spaces}, with a look at these kernels through the Fourier lens, which will play a crucial ingredient in the subsequent analysis. Section \ref{sect:exist} is dedicated to introducing the variational framework in the set of probability measures seminally introduced by Jordan, Kinderlehrer \& Otto \cite{JKO98}. Using this minimizing movement scheme, we will construct a sequence of probability measures and show its narrow compactness. The limit curve is a candidate of a weak solution to our main Eq. \eqref{eq:main_eqn}. In order to identify the limit curve as a weak solution, we employ the flow-interchange technique \`a la Matthes, McCann \& Savar{\'e}\cite{MMS09} in Section \ref{sec:flow-interchange}. Using the Boltzmann-Shannon entropy and $L^p$-norms as auxiliary functionals, we obtain additional regularity, which we exploit in Section \ref{sec:convergence-discrete} to obtain convergence in better spaces. Indeed, this is sufficient to pass to the limit in the Euler-Lagrange equations derived in Section \ref{sect:dissipationsolution}, which concludes the existence result. In Section \ref{sec:examples}, we conclude our work with a comprehensive selection of standard nonlocal operators covered by our result. Interestingly, as we shall in Section \ref{sec:Misc-example}, our result also applies to non-standard nonlocal operators such as the anisotropic fractional Laplacian 
\begin{align*}
Lu(x) =(-\partial^2_{x_1x_1})^s u(x)+  \cdots+ (-\partial^2_{x_dx_d})^su(x)= 2\pv \int_{\R^d}(u(x)-u(x+h))\d\nu(h)
\end{align*}
where here $\d\nu$ is singular the L{\'e}vy measure given by 
\begin{align*}
\d \nu(h) = \sum_{j=1}^{d}C_{1,s} |h|^{-1-2s}\d h\prod_{i\neq j}\delta_{0_i}(\d h). 
\end{align*}
It is important to note that our setting does not, in general, extend to singular measures. Indeed, the following counter-example of the discrete operator
\begin{align*}
Lu(x) = 2u(x) - u(x + a) -u(x -a)= 	 2\pv \int_{\R^d}(u(x)-u(x+h))\d\nu_a(h)
\end{align*}
with $a\in \R^d$ and $\nu_a= \frac12(\delta_a+\delta_{-a})$. Here the corresponding symbol  $\psi(\xi)= 2(1-\cos(\xi\cdot a))$ is degenerate and fails to satisfy the lower bound condition $\psi(\xi)\geq c(1\land|\xi|^2)$.  Thus, this simple case does not enter the scope of our study.

\vspace{2mm}

\emph{{Acknowledgment:}
GF and DP were supported by the Deutsche Forschungsgemeinschaft / German Research Foundation (DFG) via the Research Group 3013: ``Vector-and Tensor-Valued Surface PDEs.'' Moreover, the authors would like to thank Juan Luis V\'azquez for his helpful comments on the proof of Theorem \ref{thm:general-flow-interchange}. All authors thank the anonymous referees for their valuable comments and suggestions, which helped improve the quality of this paper. }

\section{Symmetric L\'{e}vy operators and nonlocal function spaces}\label{sec:function-spaces}
This section is dedicated to acquainting the reader with certain fundamental properties of symmetric L\'{e}vy operators, $L$, and the \emph{nonlocal Sobolev-type space} associated to them.
We refer the reader to \cite{guy-thesis,Fog23, Fog23s}, which contain a comprehensive summary of recent findings concerning these spaces. Moreover, Gagliardo-Nirenberg-Sobolev-type inequalities were recently obtained in \cite{Fog21b}, and the notion of nonlocal trace spaces is discussed in \cite{FK22,Fog23s}. 

\subsection{L\'{e}vy operator}
\label{subsec:Levy}
There are several possible ways to characterize a L\'{e}vy operator. Here, we point out the most common ones and refer to \cite[Chapter 2]{guy-thesis} where more than ten characterizations of $L$ are listed. 

\subsubsection*{\textbf{Second order difference}} 
First of all, the change of variables, $y=x\pm h$, in \eqref{eq:levy-operator}, yields
\begin{align*}
    Lu(x) =  2\pv \int_{\mathbb{R}^d} \left(u(x)-u(x\pm h)\right)\nu(h)\,\d h,
\end{align*}
which, upon summing up the two expressions, yields
\begin{align}
    \label{eq:second-difference}
    Lu(x) = \int_{\mathbb{R}^d} (2u(x)-u(x+h)-u(x-h))\nu(h)\,\d h.
\end{align}
It is worth noting that the expression on the right-hand side of \eqref{eq:second-difference} is well-defined if $u\in L^\infty(\R^d)\cap C^2(B_\delta(x))$ for some $\delta>0$. In this case, the principal value may be dropped in the definition.

\subsubsection*{\textbf{Pseudo-differential operator}} Next, we show that the integrodifferential operator $L$ can be realized as a pseudo-differential operator, as foreshadowed in the introduction. Indeed, its definition using the Fourier symbol can be justified rigorously.
\begin{theorem}
    \label{thm:fourier-symbol}
    For $u \in \mathcal{S}(\mathbb{R}^d)$ and $\xi\in\R^d$ the following relation holds:
    \begin{align*}
    \widehat{Lu} (\xi) = \psi(\xi)\widehat{u}(\xi).
    \end{align*}
    Here $\psi$ is the Fourier symbol of $L$, which is given by
    \begin{align*}
        \psi(\xi) = 2 \int_{\mathbb{R}^d}(1- \cos{(\xi\cdot h)})\nu(h)\,\d h.
    \end{align*}
\end{theorem}

\begin{proof}
Observe that for each $h\in \R^d$, we have
\begin{align*}
    \int_{\mathbb{R}^d}|u(x+h)+u(x-h)-2u(x)| \,\d x &= \int_{\R^d} \left|\int_{0}^{1} \int_{0}^{1} 2t \big[D^2 u(x-th + 2sth) \cdot h\big]\cdot h\,\d s\d t\right| \d x\\
    &\leq |h|^2 \int_{\mathbb{R}^d} \big|D^2 u(x)\big|\,\d x,  
\end{align*}
proving that the integral is finite. On the other hand, we have 
\begin{align*}
\int_{\mathbb{R}^d}|u(x+h)+u(x-h)-2u(x)| \,\d x\, 
&\leq 4\int_{\mathbb{R}^d}|u(x)| \,\d x<\infty.
\end{align*}
Combining the two preceding estimates, we readily find that 
\begin{align*}
    \int_{\mathbb{R}^d}|u(x+h)+u(x-h)-2u(x)| \,\d x\, \leq C(1\land |h|^2), 
\end{align*}
with $C= 4\|u\|_{L^1(\R^d)} + \||D^2 u|\|_{L^1(\R^d)}$. Setting
$$ 
    \Lambda(x,h)= |u(x+h)+ u(x-h) -2u(x)|\nu(h),
$$
we note that $\Lambda \in L^1(\mathbb{R}^d\times \mathbb{R}^d)$ since
\begin{align*}
\iil_{\mathbb{R}^d \mathbb{R}^d} \Lambda(x,h) \,\d h,\d x
&\leq C\int_{\mathbb{R}^d} (1\land |h|^2)\nu(h)\d h.
\end{align*}
Therefore, using the identity $\widehat{u(\cdot+h )}(\xi) = \widehat{u}(\xi) e^{i\xi \cdot h}$, along with Fubini's Theorem, we get the desired result as follows
\begin{align*}
\widehat{Lu}(\xi) 
&= - \int_{\mathbb{R}^d} e^{-i\xi \cdot x} \int_{\mathbb{R}^d}(u(x+h)+u(x-h)-2u(x)) \nu(h)\,\,\d h\, \d x \\
&= - \int_{\mathbb{R}^d} \nu(h)\int_{\mathbb{R}^d}e^{-i\xi \cdot x} (u(x+h)+u(x-h)-2u(x)) \,\,\d x\, \d h\\
&=- \widehat{u}(\xi)\int_{\mathbb{R}^d} (e^{i\xi \cdot h}+ e^{-i\xi \cdot h}-2) \nu(h)\,\d h\\
&= 2\widehat{u}(\xi)\int_{\mathbb{R}^d} (1-\cos{(\xi \cdot h)}) \nu(h)\,\d h= \widehat{u}(\xi)\psi(\xi),
\end{align*}
which concludes the proof.
\end{proof}

\begin{proposition}[Upper bound on the symbol]
    \label{prop:bound-levy-symbol}
    There exists a constant, $C>0$ such that
    \begin{align}
        \label{eq:symbol-upper-bound}
        \psi(\xi)\leq 2\int_{\R^d} (1\land |\xi|^2|h|^2)\nu(h)\d h\leq C(1+|\xi|^2 ),
    \end{align}
    for any  $\xi\in \R^d$. The constant can be chosen as $C = \kappa_\nu$k where 
    $$
        \kappa_\nu:= 2\|\nu \|_{L^1(\R^d, 1\land |h|^2)}.
    $$
\end{proposition}
\begin{proof}
    From the elementary inequality $|\sin t|\leq 1 \land |t|$,  for all $t\in \mathbb{R}$, we readily obtain
    \begin{align*}
        |1- \cos{(\xi\cdot h)}|= \big|2\sin^2{\frac{\xi\cdot h}{2}}\big|\leq 2\land \frac12|\xi|^2|h|^2\leq 2(1\land |\xi|^2|h|^2).
    \end{align*}
    The desired estimates follow from the fact that $(1\land|\xi|^2|h|^2)\leq (1+|\xi|^2)(1\land|h|^2)$. 
\end{proof}
In the case of a radial L{\'e}vy kernel, a lower bound counterpart for the symbol can be obtained.

\begin{theorem}[Lower bound on the symbol]
    \label{thm:symbol-bound}
    Assume $\nu$ is radial. Then there exists a constant, $c>0$, such that 
    \begin{align*}
        c\int_{\R^d}(1\land|h|^2|\xi|^2)\nu(h)\d h\leq \psi(\xi)\leq 2\int_{\R^d}(1\land|h|^2|\xi|^2) \nu(h)\d h,
    \end{align*}
    for all $\xi\in\R^d$. Moreover, using $ \kappa_\nu = 2\|\nu\|_{L^1(\R^d,1\land|h|^2)}$,  we have 
    \begin{align}
        \label{eq:symbol-bound}
        \frac{c\kappa_\nu}{2}(1\land|\xi|^2)\leq \psi(\xi)\leq
        \kappa_\nu (1+|\xi|^2),
    \end{align}
    for all $\xi\in\R^d$.
\end{theorem}
\begin{proof} The upper bounds follow from Proposition \ref{prop:bound-levy-symbol}. We only prove the lower bound. The Riemann-Lebesgue lemma implies 
\begin{align*}
\int_{\mathbb{S}^{d-1}} (1-\cos(\xi\cdot w))  \d \sigma_{d-1}(w)=
|\mathbb{S}^{d-1}|-\int_{\mathbb{S}^{d-1}} \cos(|\xi|w_1) \d \sigma_{d-1}(w)\xrightarrow{|\xi|\to\infty}|\mathbb{S}^{d-1}|. 
\end{align*}
Using this fact in conjunction with the estimate 
$$
    1 - \cos(t) = 2 \sin^2 \left( \frac{t}{2} \right) \geq \frac{2t^2}{\pi^2},
$$
for any $0 \leq t \leq \frac{\pi}{2}$, we can find a constant $c>0$ such that 
\begin{align*}
    \int_{\mathbb{S}^{d-1}} (1-\cos(\xi\cdot w)) \d \sigma_{d-1}(w)\geq c|\mathbb{S}^{d-1}|(1\land |\xi|^2), 
\end{align*}
for all $\xi\in \R^d$. Switching to polar coordinates and using the above estimate we get 
\begin{align*}
    \psi(\xi)
    &=\int_{\R^d} (1-\cos(\xi\cdot h))\nu(h)\d h\\
    &= \int_0^\infty r^{d-1} \nu(r) \int_{\mathbb{S}^{d-1}}
    (1-\cos(\xi\cdot rw))\d \sigma_{d-1}(w) \d r\\
    &\geq c|\mathbb{S}^{d-1}| 
    \int_0^\infty (1\land|\xi|^2r^2) r^{d-1} \nu(r) \d r\\
    &=c\int_{\R^d}(1\land|\xi|^2|h|^2)\nu(h)\d h . 
\end{align*}
Furthermore, since $(1\land a)(1\land b)\leq 1\land ab$, for any $a, b>0$, we get 
\begin{align*}
    \psi(\xi)\geq c\int_{\R^d}(1\land|\xi|^2|h|^2)\nu(h)\d h \geq\frac{c\kappa_\nu}{2}(1\land|\xi|^2),
\end{align*}
with $\kappa_\nu$ defined as above.
\end{proof}

\begin{remark}[Comparable growth]
\label{rem:comparable-growth}
The symbols $\psi$ and $\widetilde{\psi}(\xi) := |\xi|^2\psi^{-1}(\xi)$ have a similar growth. Indeed it is not difficult to check that
there exists $c_1, c_2>0$ such that, for all $\xi\in \R^d$, 
\begin{align*}
    c_1(1\land|\xi|^2)\leq \psi(\xi)\leq  c_2(1+|\xi|^2),
\end{align*}
 if and only if there exists $c_3, c_4>0$ such that, for all $\xi\in \R^d$,
\begin{align*}
    c_3(1\land|\xi|^2)\leq \widetilde{\psi}(\xi)\leq c_4(1+|\xi|^2). 
\end{align*}
\end{remark}

\begin{remark}
Assume that $\nu$ is radial, in which case  we write  $\nu(h)= \nu(|h|) $, then $\psi$ is also a radial function. 
Indeed, by the rotation invariance of the Lebesgue measure we get that
\begin{align*}
    \psi(\xi)& = 2\int_{\mathbb{R}^d}(1- \cos{(\xi\cdot h)})\nu(|h|)\,\d h\notag \\
    &=2 \int_{\mathbb{R}^d}(1- \cos{(|\xi|e_1 \cdot h' )})\nu(|h'|)\,\d h'\notag\\
    &= 2\int_{\mathbb{R}^d}(1- \cos{(h_1})\nu(h/|\xi|)\frac{\,\d h}{|\xi|^{d}}\quad(\xi\neq 0 )
    \\&= \psi(|\xi|e_1)\notag.
\end{align*}
\end{remark}
In particular, if $\nu(h)= \frac12 C_{d, s}|h|^{-d-2s}$, $s\in (0, 1)$ then $\psi(\xi)=|\xi|^{2s}$.
\vspace{2mm}

\subsubsection*{\textbf{Generator of a symmetric L\'{e}vy process and of a semigroup}} 
According to Bochner's Theorem for the Fourier transform (see \cite{BeFo75}), for each $t\geq0$, there exists a function $p_t\geq0$ continuous on $\R^d\setminus \{0\}$  such that 
\begin{align*}
    \widehat{p_t}(\xi)= \frac{1}{(2\pi)^{d/2}}
    \int_{\R^d}e^{-i\xi\cdot x}p_t(x)\d x = e^{-t\psi(\xi)},
\end{align*}
for any $\xi \in \R^d$. The convolution rule implies that $\widehat{p_{t+s}}= e^{-(t+s)\psi}=\widehat{p_t}\widehat{p_s}= \widehat{p_t*p_s}$, whence, we have 
$p_{t+s} = p_t* p_s=p_s*p_t$, for all $t,s\geq0$. Therefore, the family of operators $(P_t)_t$ defined by
\begin{align*}
    P_tu(x)= u*p_t(x)= \int_{\R^d}u(y)p_t(x-y)\d y,
\end{align*}
with $x \in \R^d$, is a strongly continuous semigroup on $L^2(\R^d)$, \emph{\emph{i.e.}}, $P_{t+s} = P_t\circ P_s =P_s\circ P_t$, for all $s,t\geq 0$,  and $\|P_tu-u\|_{L^2(\R^d)} \xrightarrow{t\to 0^+}0.$ As we shall see next, it turns out that the generator of semigroup $(P_t)_t$ is  the operator $-L$.  To this end, let $u\in \mathcal{S}(\R^d)$, whence  $\widehat{Lu} (\xi) = \widehat{u}(\xi) \psi(\xi)\in L^2(\R^d)$, by \eqref{eq:symbol-upper-bound}. The Plancherel Theorem implies,
\begin{align*}
    \Big\|\frac{P_t u-u}{t} -(-Lu) \Big\|_{L^2(\R^d)} = \Big\|\frac{\widehat{p_t} \widehat{u}-\widehat{u}}{t} +\widehat{u}\psi \Big\|_{L^2(\R^d)} = \Big\|\widehat{u} \psi\frac{e^{-t\psi} -1+t\psi}{t\psi } \Big\|_{L^2(\R^d)},
\end{align*}
having used the definition of $p_t$. Finally, we observe that the rightmost term goes to zero, as $t\to 0$, due to the fact that the function $\zeta: s\mapsto \frac{e^{-s}-1+s}{s}$ with $\zeta(0)=0$ is continuous and bounded on $[0,\infty)$. Indeed, an application of the Dominated Convergence Theorem suffices to establish
\begin{align*}
    \Big\|\widehat{u} \psi\frac{e^{-t\psi} -1+t\psi}{t\psi } \Big\|_{L^2(\R^d)}\xrightarrow{t\to 0}0.
\end{align*}

The Kolmogorov Extension Theorem (see \cite{Sat13}) implies the existence of a stochastic process $(X_t)_t$ with the transition density is $p_t(x,y) = p_t(x-y)$, namely $\mathbb{P}^x(X_t\in A)=\mathbb{E}^x[\mathds{1}_A(X_t)]$. More generally 
\begin{align*}
    \mathbb{E}^x[u(X_t)]= \int_{\R^d} u(y)p_t(x,y)\d y. 
\end{align*}
Here $\mathbb{P}^x$ (resp. $\mathbb{E}^x$) is the probability (resp. the expectation) corresponding to a process $(X_t)_t$ starting from the position $x$, \emph{\emph{i.e.}} $\mathbb{P}^x(X_0=x) =1$. 
The generator of such a stochastic process turns out to be $-L$. Indeed for a smooth function $u$,
\begin{align*}
    \lim_{t\to 0}\frac{\mathbb{E}^x[u(X_t)]-u(x)}{t} =\lim_{t\to 0}\frac{P_t u(x)-u(x)}{t} =- Lu(x). 
\end{align*}
In fact, $(X_t)_t $ is a pure-jump isotropic unimodal L\'{e}vy process in $\R^d$, \emph{i.e.},  a stochastic process with stationary and independent increments and c\`{a}dl\`{a}g paths whose transition function $p_t(x)$ is isotropic and unimodal. We refer to \cite{Sat13} for a more extensive study on L\'{e}vy processes. 
 
\begin{remark}[Particular cases]
The above applies to the case where $\psi(\xi) = |\xi|^2$, in which case, $L=-\Delta$ and the process $(X_t)_t$ is the well-known Brownian motion. If $\psi(\xi)=|\xi|$, the corresponding process is a Cauchy process whose generator is $L=(-\Delta)^{1/2}$. More generally  if  $\psi(\xi)=|\xi|^{2s}$, the corresponding process is a $2s$-stable process  whose generator is $L=(-\Delta)^{s}.$
\end{remark}

\subsubsection*{\textbf{ Energy form}} We now show that the integrodifferential operator $L$ is intimately related to a Hilbert space of great interest in its own right. Let $H_\nu(\mathbb{R}^d)$ be the space of functions $u\in L^2(\mathbb{R}^d)$ such that $\mathcal{E}_\nu(u,u)<\infty$ where  the bilinear form  $\mathcal{E}_\nu$ is defined as 
\begin{align}
    \label{def:EOm}
    \cE_\nu(u,v) = \iil\limits_{\mathbb{R}^d\mathbb{R}^d} (u(x)-u(y))(v(x)-v(y))\nu(x-y)\,\d y\,\d x. 
\end{align}
As we shall see below, $H_\nu(\mathbb{R}^d)$ is a Hilbert space.
\begin{theorem}
    \label{thm:energy-fourier}
    If $\cE_\nu(u,u)<\infty$ then 
    \begin{align*}
    \cE_\nu(u,u)= \int_{\R^d} |\widehat{u}(\xi)|^2\psi(\xi)\, \d\xi. 
    \end{align*} 
    Moreover, if $\cE_\nu(u,u)<\infty$ and $\cE_\nu(v,v)<\infty$ then $(u , v)_{\psi}= \mathcal{E}_\nu(u,v)$ which can be characterized as
    \begin{align*}
    \big(u , v\big)_{\psi}: = \int_{\R^d} \widehat{u}(\xi) \overline{\widehat{v}}(\xi)\psi(\xi)\, \d\xi.
    \end{align*}
\end{theorem}

\begin{proof} 
The second claim follows by applying the first one on $u+v$ and $u-v$. Therefore, we only prove the first claim. Note that  $|1-e^{-it}|^2=2(1-\cos{t})$, $t\in \mathbb{R}$. Plancherel's Theorem yields,
\begin{alignat*}{2}
    \cE_\nu(u,u) &=  \iil_{\mathbb{R}^d \mathbb{R}^d}(u(x)-u(y))^2\nu(x-y)\,\d y\,\d x
    & &=  \int_{\mathbb{R}^d} \nu(h) \int_{\mathbb{R}^d}(u(x)-u(x+h))^2\,\d x\,\d h\\
    &= \int_{\mathbb{R}^d} \nu(h) \int_{\mathbb{R}^d} |\widehat{u}(\xi)|^2|1-e^{-i\xi \cdot h}|^2\,\d \xi\, \d h
    &&= 2\int_{\mathbb{R}^d} |\widehat{u}(\xi)|^2\int_{\mathbb{R}^d} (1-\cos{(\xi \cdot h)})\nu(h) \,\d h \,\d \xi\\
    &=\int_{\mathbb{R}^d} |\widehat{u}(\xi)|^2 \psi(\xi) \,\d \xi,
\end{alignat*}
which concludes the proof.
\end{proof}
Based on this, let us next argue that $L$ can be extended to a continuous linear operator, To this end, let $u ,v \in \mathcal{S}(\mathbb{R}^d)$ and observe that 
\begin{align*}
    \cE_\nu(u,u) &=  \int_{\R^d} |\widehat{u}(\xi)|^2\psi(\xi)\, \d\xi=  \int_{\mathbb{R}^d} \widehat{Lu}(\xi) \overline{\widehat{u}(\xi)} \,\d \xi,
\end{align*}
since $\widehat{Lu}(\xi)= \psi(\xi)\widehat{u}(\xi)$. Another application of Plancherel's Theorem to the last expression gives the relation 
\begin{align*}
    \mathcal{E}_\nu(u,u)= \int_{\mathbb{R}^d} u(x) Lu(x) \,\d x.
\end{align*}
Replacing $u$ by $u+v$ leads to the relation
\begin{align*}
    \cE_\nu(u,v)= \int_{\mathbb{R}^d} v (x) Lu(x) \,\d x= \int_{\mathbb{R}^d} u(x) L v(x) \,\d x.
\end{align*}
\noindent Therefore, due to the density of $C_c^\infty(\R^d)$ and hence of $\mathcal{S}(\mathbb{R}^d)$ in $\Hnu $, see \cite[ Chapter 3]{guy-thesis}, $Lu$ can extended to a continuous linear form on $H_\nu(\mathbb{R}^d)$. Moreover, through the dual pairing we have 
\begin{align*}
    (Lu, v)= \cE_\nu(u,v), 
\end{align*}
for all $v \in H_\nu(\mathbb{R}^d)$. The integrodifferential operator $L$ can be extended to functions $u$ in $ H_\nu(\mathbb{R}^d)$. Thereupon, $L$  can legitimately be regarded as a linear bounded operator from $H_\nu(\mathbb{R}^d)$ into its dual, \emph{i.e.} $L: H_\nu(\mathbb{R}^d)\to \big(H_\nu(\mathbb{R}^d)\big)'$ where $\big(H_\nu(\mathbb{R}^d)\big)'$ is the dual of $H_\nu(\mathbb{R}^d)$. 
Treating the operator $L$ this way, \emph{i.e.}, derived from an associated energy form, we observe that $H_\nu(\mathbb{R}^d)$ is a fairly large domain for $L$ compared to the definition second-order differences or as pseudo-differential operators. Of course, it is worthwhile stressing that $L$ may not always be evaluated in the classical sense if defined through the correspondence $L: H_\nu(\mathbb{R}^d)\to \big(H_\nu(\mathbb{R}^d)\big)'$.

\subsection{Sobolev-Slobodeckij-like spaces}
In the last subsection, we have tied the operator $L$ to an associated nonlocal energy form. In doing so, we already got a glimpse at a bilinear form that can be derived from the quadratic from $\mathcal E_\nu$. Motivated by our treatise of the operator $L$ derived from the energy form, we shall now pursue a closer investigation of the associated spaces, $\Hnud$ and $\Hnu$, given by
\begin{align}
    \label{eq:def-Hnud}
    \Hnud&= \big\{ u \in L_{\loc}^1(\R^d)\,:\,\, \cE_\nu(u,u) <\infty \big \},
\end{align}
and
\begin{align}
    \label{eq:def-Hnu}
    \Hnu&= \big\{u \in L^2(\R^d)\,\,:\,\, \cE_\nu(u,u) <\infty \big \}.
\end{align}

We equip the space $\Hnud$ with the seminorm 
\begin{align*}
|u|^2_{\Hnu}:= \cE_\nu(u,u)= \|\widehat{u}\psi^{1/2}\|^2_{L^2(\R^d)},
\end{align*}
and, respectively, the space $\Hnu$ with the norm 
\begin{align*}
    \|u\|^2_{H_{\nu} (\R^d)}= \|u\|^2_{L^2 (\R^d)}+|u|^2_{\Hnu}= \|u\|^2_{L^2 (\R^d)}+ \cE_\nu(u,u). 
\end{align*}

\begin{remark}[Comparison with $L^2$]
    \label{rem:comparison-with-L2}
    We claim that, if $\nu \in L^1(\R^d)$  then  $L^2(\R^d)\subset \Hnud$  and $\Hnu=L^2(\R^d)$.  To prove this claim, we write 
    \begin{align*}
        \iil_{{\R^d\R^d}} &\big(u(x)-u(y) \big)^2 \,\nu(x-y)\d y\,\d x
        \leq 4 \iil_{{\R^d\R^d}}  |u(x)|^2 \nu(x-y)\d y\,\d x
    =4\|\nu\|_{L^1(\R^d)}\|u\|_{L^2(\R^d)}^{2}\,.
    \end{align*}
\end{remark}
Now, we finally show that the space $\Hnu$ is a Hilbert space.
\begin{theorem}[$H_\nu$ is a Hilbert space] The couple
$\big(\Hnu, \|\cdot\|_{\Hnu}\big)$ is a separable Hilbert space with the scalar product
\begin{align*}
    (u,v)_{\Hnu} = (u,v)_{L^2(\R^d)}+ \cE_\nu(u,v). 
\end{align*}
\end{theorem}
\begin{proof}
Clearly, $(\cdot, \cdot)_{\Hnu} $ is a scalar product on $\Hnu$ associated with the norm $ \|\cdot\|_{\Hnu}$.  Let $(u_n)_n$ be a Cauchy sequence in $\Hnu$, then a subsequence $(u_{n_k})_k$ converges to some $u$ in  $L^2(\mathbb{R}^d)$ and a.e.  in $\mathbb{R}^d$. Fix $k$ large enough, the Fatou's lemma implies 

\begin{align*}
|u_{n_k}-u|^2_{\Hnu } \leq  \liminf_{\ell\to \infty}  \iil_{{\R^d\R^d}} \big|[u_{n_k}-u_{n_\ell}](x)-[u_{n_k}-u_{n_\ell}](y) \big|^2 \, \nu (x-y) \d y\,\d x \,.
\end{align*}

\noindent
Since $(u_{n_k})_k$ is a Cauchy sequence, the right-hand side is finite for any $k$  and tends to $0$ as $k\to \infty$. This implies $u\in \Hnu $ and $|u_{n_k}-u|^2_{\Hnu } \xrightarrow{k\to \infty} 0$. Finally, $u_n\to u$ in $\Hnu $ and hence $\Hnu$ is a Hilbert space. The map $\mathcal{I}: \Hnu \to L^2(\R^d) \times L^2(\R^d\times \R^d)$ with 
\begin{align*}
\mathcal{I}(u) = \Big(u(x), (u(x)-u(y))\nu^{1/2}(x-y)\Big)
\end{align*}
is an isometry. Hence, identifying $\Hnu$ with the closed subspace $\mathcal{I}\big(\Hnu \big)$ of $L^2(\R^d) \times L^2(\R^d\times \R^d)$ implies  that $\Hnu$ is separable. 
\end{proof}

\medskip 
Albeit not pertinent to the arguments in the proof of this paper's main result, the following properties highlight the importance of the L\'{e}vy condition \eqref{eq:levy-cond} by providing its analytic interpretation. This characterization is also true in the nonlinear setting; see \cite{Fog23,Fog23s}. 
\begin{theorem} 
\label{teo:eq}
Let $\nu:\R^d\setminus\{0\}\to [0, \infty)$.  Consider the following assertions. 
\begin{enumerate}[$(i)$]
\item The L\'{e}vy condition \eqref{eq:levy-cond} holds.
\item The embedding $H^1(\R^d)\hookrightarrow\Hnu$ is continuous. 
\item $\cE_\nu(u,u)<\infty$ for all $u\in H^1(\R^d)$. 
\item $\cE_\nu(u,u)<\infty$ for all $u\in C_c^\infty(\R^d)$. 
\item  $\Hnu \neq\{0\}$. 
\end{enumerate}
The assertions $(i)-(iv)$ are equivalent. If in addition $\nu$ is radial then the assertions $(i)-(v)$ are equivalent.
\end{theorem}
For a proof we refer to \cite[Section 4]{Fog23s} and \cite[Section 2.1]{FK22}.  We will also need the following result inferring the stability of the space $\Hnu$ under a push-forward. This is a centrepiece in the establishment of the Euler-Lagrange equation which ultimately leads to the identification of the limit obtained from the minimizing movement scheme as a weak solution of \eqref{eq:main_eqn}.

\begin{theorem}
    \label{thm:push-bi-Lipschtz}
    Assume $\nu$ is a unimodal L\'{e}vy kernel, \emph{i.e.}, there exists $c\in (0,1)$ such that 
    \begin{align}
        \label{eq:nu-unimodal}
        \nu(x)\leq c\nu(y),
    \end{align}
    whenever $|x|\geq |y|$, and  
    such that the following scaling condition near the origin holds:
    \begin{align}
        \label{eq:double-cond-origin}
        \text{For every $\lambda>0$ there is $c_\lambda>0$ s.t. $\nu(\lambda h)\leq c_\lambda\nu(h)$, whenever $|h|\leq 1$}.
    \end{align}
    Finally, let $\zeta: \R^d\to\R^d$ be a bi-Lipschitz diffeomorphism on $\R^d$, \emph{i.e.}, there exists $\sigma >0$ such that
    $$
        \sigma|x-y|\leq |\zeta(x)-\zeta(y)|\leq \sigma^{-1}|x-y|.
    $$
    Then, for $u\in \Hnu$, we have $u\circ \zeta \in \Hnu$ and the following estimate holds
    \begin{align*}
        \|u\circ \zeta\|^2_{\Hnu} \leq 
        A_{\sigma,\nu} \big(1+\| \det D\zeta^{-1}  \|_{L^\infty(\R^d)}\big)^2\|u\|^2_{\Hnu},
    \end{align*}
    for some constant $A_{\sigma,\nu}>0$ only depending on $\sigma$ and $\nu$. 
\end{theorem}

\begin{proof}
Note that $|\zeta^{-1}(x)-\zeta^{-1}(y)|\geq \sigma |x-y|$ so that unimodality of $\nu$, see \eqref{eq:nu-unimodal}, implies 
\begin{align*}
    \nu(\zeta^{-1}(x)-\zeta^{-1}(y))\leq c\nu(\sigma(x-y)).
\end{align*}
By a change of variables and the unimodality of $\nu$ we get 
\begin{align*}
    \cE_\nu (u\circ \zeta,u\circ \zeta)
    &=\iil_{\mathbb{R}^d \mathbb{R}^d} |u\circ \zeta(x)-u\circ \zeta(y)|^2\nu(x-y)\d y \d x\\
    &=\iil_{\mathbb{R}^d \mathbb{R}^d} |u(x)-u(y)|^2|\det D\zeta^{-1}(x)| |\det D\zeta^{-1}(y)| \nu(\zeta^{-1}(x)-\zeta^{-1}(y)) \d y \d x\\
    &\leq c\||\det D\zeta^{-1}|\|^2_{L^\infty(\R^d)} \iil_{\mathbb{R}^d \mathbb{R}^d} |u(x)-u(y)|^2 \nu(\sigma(x-y)) \d y \d x.
\end{align*}
Furthermore, using the scaling condition near the origin, \eqref{eq:double-cond-origin}, we find that 
\begin{align*}
    \iil_{\R^d\R^d} &
    |u(x)-u(y)|^2 \nu(\sigma(x-y)) \d y \d x\\
    &= \iil_{|x-y|\geq 1} |u(x)-u(y)|^2 \nu(\sigma(x-y)) \d y \d x + \iil_{|x-y|\leq 1}\! |u(x)-u(y)|^2 \nu(\sigma(x-y)) \d y \d x\\
    & \leq 4\|u\|^2_{L^2(\R^d)} \int_{|h|>1} \nu(\sigma\,h) \d h+ c_\sigma\int_{\R^d}\int_{|h|\leq 1} |u(x)-u(x+h)|^2 \nu(h) \d h \d x\\
    &\leq 4\sigma^{-d}\|u\|^2_{L^2(\R^d)} \int_{|h|>\sigma} \nu(h) \d h+ c_\sigma \cE_\nu(u,u)\\
    &\leq \left(4\sigma^{-d}\int_{|h|>\sigma} \nu(h) \d h+ c_\sigma \right) \|u\|^2_{\Hnu}.
\end{align*}
On the other hand, we have 
\begin{align*}
    \int_{\R^d} |u\circ \zeta(x)|^2 \d x = \int_{\R^d} |u(x)|^2 |\det D\zeta^{-1}(x)| \d x\leq \| \det D\zeta^{-1} \|_{L^\infty(\R^d)} \|u\|^2_{L^2(\R^d)}. 
\end{align*}
Hence, combining the last two estimates, we obtain the desired estimate 
\begin{align*}
    \|u\circ \zeta\|^2_{\Hnu} \leq 
    A_{\sigma,\nu} \big(1 + \| \det D\zeta^{-1} \|_{L^\infty(\R^d)}\big)^2\|u\|^2_{\Hnu}. 
\end{align*}
\end{proof}

\subsection{Compact embedding}\label{subsec:assum-levy}
In this section, we establish the compact embedding of $\Hnu$ into $L^2_{\loc}(\R^d)$. We recall that $L^2_{\loc}(\R^d)$ is equipped with the topology of $L^2$-convergence on compact sets. Namely, a sequence $(u_n)_n$ converges in  $L^2_{\loc}(\R^d)$ if for every compact set $K\subset \R^d$ there is $u_K\in L^2(K)$ such that  $\|u_n-u_K\|_{L^2(K)}\to 0$ as $n\to \infty$. In this case, we say that $(u_n)_n$ converges in  $L^2_{\loc}(\R^d)$  to the function $u$ defined by $u|_K=u_K$ for every compact set.
\begin{remark}[$L^1$-L{\'e}vy kernels]
    \label{rem:L1-levy-kernels}
As noticed in Remark \ref{rem:comparison-with-L2}, if $\nu\in L^1(\R^d)$, then $\Hnu =L^2(\R^d)$ which is not locally compactly embedded into $L^2_{\loc}(\R^d)$. Thus, imposing $\nu$ to satisfy the L{\'e}vy condition, \eqref{eq:levy-cond}, as well as the non-integrability condition, $\nu\not\in L^1(\R^d)$, is paramount. 
\end{remark}

Let us start with the following lemma on the compactness of convolution operators.
\begin{lemma} 
    \label{lem:compactness-convolution} 
    Let $w \in L^1(\mathbb{R}^d)$ and $K\subset \R^d$ be compact. The convolution operator $T_w : L^2(\mathbb{R}^d )\to L^2(K)$, with $u \mapsto T_w=w*u$, is compact.
\end{lemma}
\begin{proof}
In virtue of Young's inequality, for very $u\in L^2(\mathbb{R}^d)$
\begin{align}
    \label{eq:wyoung-inequality}
    \|w*u\|_{L^2(\mathbb{R}^d)}\leq \|w\|_{L^1(\mathbb{R}^d)}\|u\|_{L^2(\mathbb{R}^d)}. 
\end{align}

Let $\mathrm{B}$ be a bounded subset of $L^2(\mathbb{R}^d)$ and set $M= \sup\limits_{u\in \mathrm{B}} \|u\|_{L^2(\mathbb{R}^d)}$. Then, \eqref{eq:wyoung-inequality} implies that $T_w(\mathrm{B})$ is a bounded subset of $L^2(\mathbb{R}^d)$, too, and we can control the shifts
\begin{align*}
\sup_{u\in \mathrm{B} }\int_{\mathbb{R}^d}\big|T_wu(x+h)-T_wu(x)\big|^2\d x 
&\leq M^2\|w(\cdot+h)-w(\cdot)\|^2_{L^1(\mathbb{R}^d)}\xrightarrow[]{|h|\to0 }0.
\end{align*}
The Riesz-Fr\'echet-Kolmogorov theorem implies that $T_w(\mathrm{B})|_{K}$ is relatively compact in $L^2(K)$ and the desired result follows. 
\end{proof}

\begin{theorem}[Characterization of local compact embeddings]
    \label{thm:local-compactness}
The embedding $H_\nu(\R^d) \hookrightarrow L^2_{\loc}(\R^d)$ is compact if and only if  $\nu\not\in L^1(\R^d)$.
\end{theorem}

\begin{proof} Assume $\nu\in L^1(\R^d)$. Then, by Remarks \ref{rem:comparison-with-L2} and \ref{rem:L1-levy-kernels}, $\Hnu= L^2(\R^d)$. Of course, $L^2(\R^d) \hookrightarrow L^2_{\loc}(\R^d)$ is not compact, which proves the first direction.

Next, we prove the converse, namely that any L{\'e}vy kernel lacking integrability at the origin gives rise to an energy space that compactly embeds into $L^2$ -- which is perhaps more surprising. To prove the claim, we assume $\nu\not \in L^1(\R^d)$ and show the embedding is compact, indeed. To this end, let $\delta>0$ be sufficiently  small such that, upon removing the singularity close to the origin by introducing $\nu_\delta :=\nu\mathds{1}_{B^c_\delta(0)}$, the resulting measure has finite mass, \emph{i.e.}, $0<\|\nu_\delta\|_{L^1(\mathbb{R}^d)}<\infty$. Rescaling its mass to unity, we introduce 
$$
    w_\delta(h) := \nu_\delta(h) \|\nu_\delta\|^{-1}_{L^1(\mathbb{R}^d)},
$$
which we shall use as a convolution kernel.

For fixed $u \in L^2(\mathbb{R}^d)$, by evenness of $\nu $ for all $x\in \mathbb{R}^d$ we have 
\begin{align*}
    T_{w_\delta} u(x)= \int_{\mathbb{R}^d} w_\delta(y) u(x-y)\d y = \int_{\mathbb{R}^d} w_\delta(y) u(x+y)\d y\,. 
\end{align*} 
Thus, by Jensen's inequality
\begin{align*}
    \|u- T_{w_\delta} u\|^2_{L^2(\mathbb{R}^d)} &= \int_{\mathbb{R}^d} \big| \int_{\mathbb{R}^d} [u(x)-u(x+h) ]w_\delta (h)\d h\big|^2\d x\\
    & \leq \|\nu_\delta\|^{-1}_{L^1(\mathbb{R}^d)} \iint\limits_{\mathbb{R}^d\mathbb{R}^d} |u(x)-u(x+h) |^2 \nu (h)\d h \d x \\
    &\leq \|\nu_\delta\|^{-1}_{L^1(\mathbb{R}^d)} \|u\|^2_{\Hnu} \,.
\end{align*}

\noindent For a compact set $K\subset\mathbb{R}^d$, we put $R_K u= u|_{K}$. Since  $\nu\not\in L^1(\R^d)$ it follows that 
\begin{align*}
    \|R_K -R_K T_{w_\delta} \|_{\mathcal{L}\big(\Hnu,\, L^2(K)\big)}\leq \|\nu_\delta\|^{-1}_{L^1(\mathbb{R}^d)}\xrightarrow[]{\delta\to 0}0\,.
\end{align*}
\noindent Thus the operator $R_K: \Hnu\to L^2(K)$, is compact since by Lemma \ref{lem:compactness-convolution}, the operator $R_K \circ T_{w_\delta}$ is also compact for every $\delta$.
\end{proof}

\noindent As a straightforward consequence of Theorem \ref{thm:local-compactness} we have the following. 

\begin{corollary}\label{cor:local-compatcness}
Assume  $\nu\not\in L^1(\R^d)$ and  $(u_n)_n$ be a bounded sequence in $\Hnu$. Then, there exist $u\in \Hnu$ and a subsequence $(u_{n_j})_j$ converging to $u$ in $L^2_{\operatorname{loc}}(\R^d)$. 
\end{corollary}

\medskip
It is worth mentioning that Theorem \ref{thm:local-compactness} was first proved in \cite[Theorem 1.1]{JW20}.  However,  earlier results using similar techniques also appeared in \cite[Proposition 6]{PZ17} for periodic functions on the  torus. The technique of removing the singularity is also used in \cite[Lemma 3.1]{BJ13} and \cite[Proposition 1]{BJ16}.  The proof of the global compactness on bounded domains can be found in \cite{FK22} or in \cite[Chapter 3]{guy-thesis} for the general case.

\subsection{Switching to Fourier notations}
In the remainder of this section, we shall switch to the Fourier side setting up similar nonlocal spaces as above defined for a general Fourier symbol,  $\phi: \R^d\setminus\{0\}\to (0, \infty)$, which is assumed to be continuous and symmetric, \emph{i.e.}, $\phi(\xi)=\phi(-\xi)$. 
\begin{definition}[Symbolic Sobolev Spaces] 
    We define the symbolic Sobolev spaces
    \begin{align*}
        \dot{H}^{\phi} (\R^d)= \big\{u\in\mathcal{S}'(\R^d)\,:\,  \widehat{u}\phi^{1/2}\in L^2(\R^d)\big\}, 
    \end{align*}
    and $H^{\phi} (\R^d) := L^2(\R^d)\cap \dot{H}^{\phi}(\R^d)$  as 
    \begin{align*}
        H^{\phi} (\R^d)= \big\{u\in L^2(\R^d)\,:\,  \widehat{u}\phi^{1/2}\in L^2(\R^d)\big\}. 
    \end{align*}
    The two spaces $\dot{H}^{\phi}(\R^d)$ and  $H^{\phi} (\R^d)$ are respectively equipped with the following (semi)norms 
    \begin{align*}
        \|u\|^2_{\dot{H}^{\phi} (\R^d)} 
        &= \int_{\R^d} |\widehat{u}(\xi)|^2\phi(\xi)\d \xi,
    \end{align*}
    as well as
    \begin{align*}
        \|u\|^2_{H^{\phi} (\R^d)} 
        &=  \int_{\R^d} |\widehat{u}(\xi)|^2\d \xi+ \int_{\R^d} |\widehat{u}(\xi)|^2\phi(\xi)\d \xi.
    \end{align*}
    It is an easy observation that, $\|\cdot\|_{\dot{H}^{\phi} (\R^d)}$ is the seminorm associated with the symmetric bilinear form $\big(\cdot,\cdot \big)_{\phi}$ given by
    \begin{align*}
        \big( u , v\big)_{\phi}: = \int_{\R^d} \widehat{u}(\xi) \overline{\widehat{v}}(\xi)  \phi(\xi)\, \d\xi.
    \end{align*}
    In the same vein, we define the space 
    $\dot{H}^{\phi^{-1}}(\R^d)$ as the dual space $(\dot{H}^{\phi}(\R^d))'$ endowed with the norm
    \begin{align*}
    \|u\|_{\dot{H}^{\phi^{-1}}(\R^d)} =\sup_{v\in \dot{H}^{\phi}(\R^d), \,v\neq0} \frac{\langle u, v\rangle}{\|u\|_{\dot{H}^{\phi}(\R^d)}} =  \Big(\int_{\R^d} |\widehat{u}(\xi)|^2\phi^{-1}(\xi)\d \xi\Big)^{1/2},
    \end{align*}
\end{definition}
where  $\langle u,v \rangle := \big( \widehat{u}\phi^{-1/2}, \widehat{v}\phi^{1/2}\big)_{L^2(\R^d)}$ for any  $u\in \dot{H}^{\phi^{-1}}(\R^d)$ and  $v\in \dot{H}^{\phi}(\R^d)$. It is worth mentioning that a class of Symbolic Sobolev spaces also called anisotropic Sobolev spaces associated with symbol of general L\'evy operators are studied, for instance in \cite{Jac94,Hoh94,Hoh98,EFJ21}, in the context of the study of the martingale problem for Markov processes.

Let us recall that
\begin{align*}
    \psi(\xi)= 2\int_{\R^d} (1-\cos(\xi\cdot h))\nu(h)\d h,
\end{align*}
for a L\'{e}vy kernel $\nu\in L^1(\R^d, 1\land|h|^2)$. It is apparent that the symbol $\psi$, respectively $\psi^{-1}$, will play a prominent role throughout the paper. Moreover, we stress the importance of the associated  symbols $ \widetilde{\psi}$, and $\psi^*$ given by
\begin{align*}
    \widetilde{\psi}(\xi) = |\xi|^2\psi^{-1}(\xi), \qquad \text{and} \qquad
\psi^* (\xi) = |\xi|^2\psi^{-2}(\xi)= \psi^{-1}(\xi)\widetilde{\psi}(\xi).
\end{align*}
In particular, note that we have $\Hnu= H^{\psi} (\R^d)$. 
\begin{remark}[Relation to fractional Sobolev spaces]\label{rem:symbol}
    For the standard fractional case $ \nu(h) = \frac{C_{d,s}}{2}|h|^{-d-2s}$, $s\in (0,1)$, the aforementioned symbols read
    \begin{align*}
        \psi(\xi) =|\xi|^{2s}, \qquad \text{and} \qquad \psi^{-1}(\xi)=|\xi|^{-2s},
    \end{align*}
    \emph{i.e.}, the symbol of the fractional Laplacian and the associated Riesz kernel, respectively. Moreover,
    \begin{align*}
        \widetilde{\psi}(\xi)=|\xi|^2\psi^{-1}(\xi)=|\xi|^{2(1-s)}, \qquad \text{and} \qquad \psi^*(\xi)= |\xi|^2\psi^{-2}(\xi)= |\xi|^{2(1-2s)}.
    \end{align*}
    The nonlocal spaces associated to these symbols simply become
    \begin{align*}
        \Hnu =\Hnup= H^s(\R^d), \quad \text{and} \quad \Hnui= (\Hnup)' = H^{-s}(\R^d).
    \end{align*}
    as well as
    \begin{align*}
        \Hnuw = H^{1-s}(\R^d), \qquad \text{and}\qquad  \Hnus = H^{1-2s}(\R^d).
    \end{align*}
\end{remark}

Let us now mention the compactness result with respect to the symbolic Sobolev $H^\phi(\R^d)$.
\begin{theorem}\label{thm:local-comp-symb}
The embedding $H^\phi(\R^d)\hookrightarrow L^2_{\loc}(\R^d)$ is compact, provided that $\phi$ satisfies 
\begin{align}
    \label{eq:symb-asump-compact}
    \sup_{\xi\in \R^d} \frac{1}{\phi(\xi)}|e^{i\xi\cdot h}-1|^2\xrightarrow{|h|\to0}0.
\end{align}
\end{theorem}

\begin{proof}
Let $\mathbf{B}\subset H^\phi(\R^d)$ be a bounded subset, and let $M:=\sup_{u\in \mathbf{B}}\|u\|_{ H^\phi(\R^d)}<\infty$. For $u\in\mathbf{B}$, using Plancherel we have 
\begin{align*}
\|u(\cdot+h)-u\|^2_{L^2(\R^d)}
&= \int_{\R^d}|\widehat{u}(\xi)|^2|e^{i\xi\cdot h}-1|^2\d \xi \\
&\leq  \sup_{\xi\in \R^d} \frac{1}{\phi(\xi)}|e^{i\xi\cdot h}-1|^2\int_{\R^d}|\widehat{u}(\xi)|^2\phi(\xi)\d \xi\\
&\leq M \sup_{\xi\in \R^d} \frac{1}{\phi(\xi)}|e^{i\xi\cdot h}-1|^2. 
\end{align*} 
Hence $\mathbf{B}$ is a bounded subset $L^2(\R^d)$ and we have 
\begin{align*}
\sup_{u\in \mathbf{B}}
\|u(\cdot+h)-u\|^2_{L^2(\R^d)}\leq 
M \sup_{\xi\in \R^d} \frac{1}{\phi(\xi)}|e^{i\xi\cdot h}-1|^2\xrightarrow{|h|\to0}0. 
\end{align*} 
The sought compactness thus follows from the Riesz-Fr\'echet-Kolmogorov theorem. 
\end{proof}

\begin{remark}
Taking $\phi=\psi$, the equivalence in Theorem \ref{thm:local-compactness} the condition in \eqref{eq:symb-asump-compact}, i.e., 
\begin{align*}
\sup_{\xi\in \R^d} \frac{1}{\psi(\xi)}|e^{i\xi\cdot h}-1|^2\xrightarrow{|h|\to0}0
\end{align*}
immediately implies that $\nu\not\in L^1(\R^d)$. It would be  interesting to know if the converse holds true. Namely does the fact that $\nu\in L^1(\R^d, 1\land|h|^2) \setminus L^1(\R^d)$ imply that the above condition in \eqref{eq:symb-asump-compact}. In other words, in view of Theorem \ref{thm:local-compactness} and Theorem \ref{thm:local-comp-symb}, it is natural to ask the question whether the condition in \eqref{eq:symb-asump-compact} is a necessary and sufficient condition on $\phi$ so that the embedding $H^\phi(\R^d)\hookrightarrow L^2_{\loc}(\R^d)$ is compact?
\end{remark}

\section{Existence of solution to the minimizing movement scheme}\label{sect:exist}
As foreshadowed in the introduction, the construction of solutions to \eqref{eq:main_eqn} is achieved by employing a minimizing movement scheme in the set of probability measures equipped with the 2-Wasserstein distance; we refer the reader to \cite{Vil09,ABS21} for more details on the Wasserstein distance.  Throughout, let us denote by $\mathcal{P}(\R^d)$ the set of Borel probability measures on $\R^d$. We say that a sequence $(u_n)_n\subset \mathcal{P}(\R^d)$ converges narrowly to $u\in \mathcal{P}(\R^d)$, if 
\begin{align*}
\lim_{n\to \infty}  \int_{\R^d}  \phi(x)\d u_n (x) = \int_{\R^d} \phi(x)\d u(x),
\end{align*}
for all $\phi\in C_b(\R^d)$, where $C_b(\R^d)$ is the space of continuous and bounded function on $\R^d$.  Additionally, the space of Borel probability measures with finite second moment is defined as 
\begin{align*}
    \mathcal{P}_2(\R^d) := \left\{u \in \mathcal{P}(\R^d)\,:\, \int_{\R^d} |x|^2\d u(x) < \infty\right\}.
\end{align*}
It is well-known that this set is a complete, separable metric space when equipped with the 2-Wasserstein distance, $W$, defined by
\begin{align*}
    W(u_0, u_1) = \min_{\gamma \in \Gamma(u_0, u_1)} \left\{ \iint_{\R^d\times \R^d} |x-y|^2\d \gamma(x,y)\right\}^{1/2},
\end{align*}
where $\Gamma(u_0, u_1)$ denotes the set of all transport plans between $u_0$ and $u_1$, that is, the set of probability measures on the product space $\R^d \times \R^d$ with marginals $u_0$ and $u_1$, \emph{i.e.}, $(\pi_x)_\# \gamma=u_0$ and $(\pi_y)_\# \gamma=u_1$. Moreover \cite[Theorem 5.2]{ABS21}, if $u_0$ is absolutely continuous with respect to the Lebesgue measure, then $\gamma= (\operatorname{I}\times T_{u_0}^{u_1})_\# u$  is the minimizer for $W(u_0,u_1)$  where  $T_{u_0}^{u_1}$ is the  transport map such that $(T_{u_0}^{u_1})_\# u_0=u_1$, \emph{i.e.},  $u_1$ is the push forward of $u_0$ through $T_{u_0}^{u_1}$. Moreover we have 
\begin{align*}
W(u_0, u_1) = \min_{S_\# u_0=u_1} \Big\{ \int_{\R^d} |S(x)-x|^2\d u_0(x)\Big\}^{1/2}
: = \Big\{ \int_{\R^d} |T^{u_0}_{u_1}(x)-x|^2\d u_0(x)\Big\}^{1/2},
\end{align*}
Now, if both $u_0$ and $u_1$ have densities then 
\begin{align*}
T^{u_0}_{u_1}\circ T_{u_0}^{u_1}= \operatorname{I}\quad\text{$u_0$-a.e.\quad  and }\quad  T_{u_0}^{u_1} \circ T^{u_0}_{u_1}= \operatorname{I}\quad \text{$u_1$-a.e.}. 
\end{align*}

\begin{remark}
    Let us point out that the space $\Hnuid$ is naturally associated to the operator
    $\mathcal{F}:\mathcal{P}_2(\R^d)\to \R\cup \{\infty\}$,
    \begin{align*}
    u \mapsto \mathcal{F}(u)=\frac{1}{2} \|\psi^{-1/2} \widehat{u}\|^2_{L^2(\R^d)}= \frac{1}{2} \int_{\R^d} |\widehat{u}(\xi)|^2\psi^{-1}(\xi) \d \xi=\frac12\|u\|^2_{\Hnuid}. 
    \end{align*}
    Clearly, $\mathcal{F}(u)<\infty$ if and only if  $u\in \Hnuid\cap \mathcal{P}_2(\R^d),$ \emph{i.e.}, 
    $D(\cF)= \Hnuid\cap \mathcal{P}_2(\R^d)$.
\end{remark}
\begin{theorem}\label{thm:lower-semicontinuous}
Let $u_*\in \mathcal{P}_2(\R^d)$ and $\tau > 0$ be given. Then, the  Moreau-Yosida penalization mapping 
$$
u \mapsto \frac{1}{2\tau} W^2(u,u_*)+\mathcal{F}(u),
$$ 
is lower semi-continuous with respect to the narrow convergence in $\mathcal{P}_2(\R^d)$. Moreover, there exists a unique minimizer, $u \in \mathcal{P}_2(\R^d)$.
\end{theorem}

\begin{proof} We begin by establishing the lower semi-continuity and show the existence and uniqueness of a minimizer later.
By the lower semi-continuity of the Wasserstein distance, \emph{cf.} \cite[Proposition 7.1.3]{AGS08}, it is sufficient to prove the lower semi-continuity of $u\mapsto \mathcal{F}(u)$, as the sum of two lower semi-continuous functions is lower semi-continuous. 
Now, let  $(u_n)_n \subset \mathcal{P}_2(\R^d)$ be narrowly convergent to $u \in \mathcal{P}_2(\R^d)$. 
Without loss of generality, we assume that $\mathcal{F}(u_n) < \infty$, uniformly in $n\in \N$, as, otherwise, the liminf-inequality is trivially satisfied. 
Next, let us set $U_n := \psi^{-1/2} \widehat{u}_n$ and observe, that the bound on $(\mathcal{F}(u_n))_n$ implies the uniform $L^2$-bound $\sup_n\|U_n\|_{L^2(\R^d)}<\infty$. 
Consequently, up to a subsequence,  $U_n \rightharpoonup U$, weakly converges in $L^2(\R^d)$. Next we show that $U(\xi) =\psi^{-1/2}(\xi)\widehat{u}(\xi)$. By the Banach-Saks Theorem, see\cite{Fog23BS} or \cite[Appendix A]{MR12}, there exists a further subsequence still denoted $(U_n)_n$ whose Ces{\'a}ro mean converges strongly in $L^2(\R^d)$, \emph{i.e.},
\begin{align*}
    V_n := \frac1n \sum_{i=1}^n U_i  \longrightarrow U.
\end{align*}
Passing to another subsequence, we have $V_n \to U$, almost everywhere in $\R^d$.  Simultaneously, this narrow convergence implies the pointwise convergence of $\widehat{u}_n$: $\widehat{u}_n(\xi)\to \widehat{u}(\xi)$ for all $\xi\in \R^d$, as $n\to \infty$. Therefore, we get 
\begin{align*}
U(\xi)= \lim\limits_{n\to \infty}V_n(\xi)= \frac1n \sum_{i=1}^n \psi^{-1/2} \widehat u_i = \lim\limits_{n\to \infty}\psi^{-1/2}(\xi)\widehat{u}_n(\xi) = \psi^{-1/2}(\xi)\widehat{u}(\xi),
\end{align*}
for almost every $\xi\in \R^d$. 
The weak lower semi-continuity of $\|\cdot\|_{L^2(\R^d)}$ with respect to pointwise a.e. convergence implies that 
\begin{align*}
 \mathcal{F}(u)= \|U\|^2_{L^2(\R^d)}\leq \liminf_{n\to \infty}  \|U_n\|^2_{L^2(\R^d)}= \liminf_{n\to \infty}  \mathcal{F}(u_n),
\end{align*}
which proves the lower semi-continuity of the Moreau-Yosida penalization, as claimed.
Now, the existence of a unique minimizer is a straightforward consequence of the direct method.
Indeed, $u\mapsto \frac{1}{2\tau} W^2(u,u_*)+\mathcal{F}(u)$ is lower semicontinuous. Moreover, it is coercive on $\mathcal{P}_2(\R^d)$ with respect to the narrow convergence -- each sublevel set $\{u\in \mathcal{P}_2(\R^d)\,:\, \frac{1}{2\tau}W^2(u,u_*)+\mathcal{F}(u)\leq M\},M\in\R$ is either empty or bounded in $\mathcal{P}_2(\R^d)$. It is well-known that any bounded set in $\mathcal{P}_2(\R^d)$ is relatively compact with respect to the narrow convergence. The uniqueness of the minimizer follows from the strict convexity of the functional $u\mapsto \frac{1}{2\tau} W^2(u,u_*)+\mathcal{F}(u)$. 
\end{proof}

\subsection{Minimizing movement scheme}
The minimizing movements scheme in the set of probability measures, originally introduced in \cite{JKO98} (see also \cite[Definition 2.0.2]{AGS08}), is defined as follows. 
\begin{definition}
\label{def:discrete}
Given $\tau > 0$ and $u_{0} \in D(\cF)$, we consider the sequence of discrete approximations $(u_{\tau}^{k})_{k \in \mathbb{N}}$ uniquely defined through the recursive scheme: $u_{\tau}^{0} := G_{\omega(\tau)}*u_{0}$, where $ G_t(x) = \frac{1}{(4t\pi)^{d/2}} e^{-\frac{|x|^2}{4t} }$ is the standard Gaussian heat kernel and $\omega(\tau)= -1/\log(\tau)$ if $\tau\in (0,1/2)$ and  $\omega(\tau)= 1/\log(2)$ if $\tau\in [1/2,\infty)$ and for $k\geq1$, 
\begin{align}
    \label{eq:discrete} 
    u_{\tau}^{k} \in \argmin_{u \in \mathcal{P}_2(\R^d)} \Big\{ \frac{1}{2 \tau} W^{2}\Big( u,  u_{\tau}^{k-1} \Big) + \mathcal{F}(u)\Big\}.
\end{align}
In other words $u_{\tau}^{k}$ is the unique element such that 
\begin{align}
    \label{eq:discrete-bis}
     \frac{1}{2 \tau} W^{2}\Big( u_{\tau}^{k},  u_{\tau}^{k-1} \Big) +\cF(u_{\tau}^{k}) 
    = \min_{u \in \mathcal{P}_2(\R^d)} \Big\{\frac{1}{2 \tau} W^{2}\Big(u,  u_{\tau}^{k-1} \Big) + \mathcal{F}(u)\Big\}.
\end{align}
We introduce the piecewise constant interpolation associated with the  sequence of minimizers defined as follows $ u_{\tau}: [0,\infty)\to \cP_2(\R^d)$, 
\begin{align*}
    u_{\tau} (0) :=  u_{\tau}^{0} \quad\text{and}\quad u_{\tau} (t) :=  u_{\tau}^{\lceil t/\tau\rceil} \quad\text{for $t>0$}. 
\end{align*}
Recall that the existence and uniqueness of each $u_{\tau}^{k}$ are guaranteed by  Theorem \ref{thm:lower-semicontinuous}.  Keep in mind that  given $a\in\R$ we denote $\lfloor a\rfloor = \max\{m\in \mathbb{Z}: m\leq a\}$  and  $\lceil a\rceil= \min\{m\in \mathbb{Z}: a\leq m\}$ so that $\lceil a\rceil= \lfloor a\rfloor +1$, $\lfloor a\rfloor \leq a<\lfloor a\rfloor +1$ and $\lceil a\rceil-1<a\leq \lceil a\rceil$.
\end{definition}

\noindent The next result follows the idea from \cite[Theorem 3.3]{LMS2018}. 
\begin{proposition}\label{prop:discrete} Let $u_0\in \cP_2(\R^d)$, $\tau>0$ and $(u_{\tau}^{k})_k$ be the sequence defined as in Eq. \eqref{eq:discrete}. 
\begin{enumerate}[$(i)$]
\item For all $k\geq 1$ we have
    \begin{align*}
        \frac{1}{2\tau}W^2(u_{\tau}^{k-1}, u_{\tau}^{k})\leq  \cF(u_{\tau}^{k-1})-\cF(u_{\tau}^{k}),
    \end{align*}
    and therefore $\cF(u_{\tau}^{k})\leq \cF(u_{\tau}^{k-1})\leq\cdots \leq \cF(u_0)$.
\item For all $N\geq 1$  we have 
\begin{align}
\label{eq:dist-uk-delta}
\begin{split}
\int_{\R^d} |x|^2u_{\tau}^{N}(x)\d x
\leq \frac{8d}{\log(2)}\left(1+\tau N\cF(u_0)+ \int_{\R^d} |x|^2\d u_0(x)\right)\,. 
\end{split}
\end{align}
\end{enumerate}
\end{proposition}

\begin{proof}
$(i)$  Testing Eq. \eqref{eq:discrete-bis} with $u=u_{\tau}^{k-1}$ implies $W^2(u_{\tau}^{k-1}, u_{\tau}^{k})\leq  2\tau(\cF(u_{\tau}^{k-1}))-\cF(u_{\tau}^{k})$, whereas the estimate $\cF(u_{\tau}^{0})\leq \cF(u_0)$ is an obvious consequence of the fact that $0< \widehat{G}_{\omega(\tau)}\leq 1$.

$(ii)$ The triangle inequality and the estimates $\frac{1}{2\tau}W^2(u_{\tau}^{k-1}, u_{\tau}^{k})\leq  \cF(u_{\tau}^{k-1})-\cF(u_{\tau}^{k})$ imply 
\begin{align*}
    \int_{\R^d} |x|^2u_{\tau}^{N}(x)\d x
    = W^2(u_{\tau}^{N}, \delta_0)&
    \leq 2W^2(u_{\tau}^{0}, \delta_0) + 2\Big(\sum_{k=1}^N W(u_{\tau}^{k-1}, u_{\tau}^{k}) \Big)^2\\
    &\leq 2W^2(u_{\tau}^{0}, \delta_0) + 4\tau N\sum_{k=1}^N \frac{1}{2\tau}W^2(u_{\tau}^{k-1}, u_{\tau}^{k})\\
    &\leq2W^2(u_{\tau}^{0}, \delta_0) + 4\tau N \sum_{k=1}^N  \cF(u_{\tau}^{k-1})-\cF(u_{\tau}^{k})\\
    &= 2W^2(u_{\tau}^{0}, \delta_0) + 4\tau N\left(\cF(u_{\tau}^{0})-\cF(u_{\tau}^{N})\right)\\
    &\leq 2W^2(G_{\omega(\tau)}*u_0, \delta_0) + 4\tau N\cF(u_{\tau}^{0}).
\end{align*}
A standard computation reveals  that
\begin{align}
    \label{eq:dist-u0-delta}
    W^2(G_\tau*u, \delta_0)\leq 2W^2(u, \delta_0)+ 2W^2(G_\tau, \delta_0)= 2\int_{\R^d} |x|^2\d u(x)+ 4\tau d. 
\end{align}
We deduce from this, and the fact that $\omega(\tau)\leq \frac{1}{\log(2)}$, that 
\begin{align*}
\int_{\R^d} |x|^2u_{\tau}^{N}(x)\d x
= W^2(u_{\tau}^{N}, \delta_0)
&\leq 4\tau N\cF(u_0)+  4\int_{\R^d} |x|^2\d u_0(x)+ 8d\omega(\tau)\\
&\leq \frac{8d}{\log(2)}\Big(1+\tau N\cF(u_0)+ \int_{\R^d} |x|^2\d u_0(x)\Big).
\end{align*}
\end{proof}

Next, we give the definition of an absolutely continuous curve. 
\begin{definition}[{\hspace{-0.1ex}\cite[Definition 1.1.1]{AGS08}, \cite[Definition 9.1]{ABS21}}]
We say that a  curve $u : [0, \infty) \to \mathcal{P}_2(\R^d)$ belongs to $AC^{2} ([0, \infty),(\mathcal{P}_2(\R^d), W) )$ if  there exist $m\in L^2((0,\infty))$ such that 
\begin{align}
    \label{eq:absolutely-cont-esti}
    W(u(t_{1}), u(t_{2})) \leq \int_{t_{1}}^{t_{2}} m(t) \, \d t,\qquad\text{for every $0 \leq t_{1} < t_{2}<\infty$.}
\end{align}
It is worth mentioning that, if $u\in AC^{2} ([0, \infty),(\mathcal{P}_2(\R^d), W))$ then (see \cite[Theorem 1.1.2]{AGS08} or \cite[Theorem 9.2]{ABS21}) $u$ has a metric derivative $|u'|$ almost everywhere, that is, the following limit exists 
\begin{align*}
|u'|(t):=\lim_{s\to 0}\frac{W(u(t+s),u(t))}{|s|},
\end{align*}
for almost all $t>0$. Accordingly, the Lebesgue Differentiation Theorem  implies $|u'|(t)\leq m(t)$ for almost all $t>0$. Furthermore, $t\mapsto |u'|(t)$ turns out to be the minimal (smallest) among functions $m$ in $L^2((0,\infty))$ satisfying Eq. \eqref{eq:absolutely-cont-esti}. 
\end{definition}
\noindent The next result establishes the existence of an absolute continuous curve $u:[0, \infty)\to (\cP_2(\R^d),W)$, limit of the piecewise constant curves $(u_\tau)_\tau$ as $\tau\to 0$ and representing the minimizing movement (see \cite[Definition 2.0.6]{AGS08}) for the main Eq. \eqref{eq:main_eqn}. 
\begin{theorem}[Convergence of the Minimizing Movement Scheme]\label{thm:mini-mvt-scheme}
Let $u_0\in D(\cF)=\Hnuid\cap\cP_2(\R^d)$.  Let $u_{\tau}(t)= u_{\tau}^{\lceil t/\tau \rceil}$ defined as in Definition \ref{def:discrete}. Then, there holds
\begin{align*}
    &\lim_{\tau\to 0^+} \cF( u_{\tau}^{0})= \cF( u_{0}), \quad\text{and}\quad \limsup_{\tau\to 0^+} W( u_{\tau}^{0}, u_0)<\infty,
\end{align*}
as well as
\begin{align*}
    & u_{\tau}^{0}\to u_0, \quad \text{narrowly  as $\tau \to 0^+$}. 
\end{align*}

Moreover, there is a curve $u\in AC^2([0, \infty),(\cP_2(\R^d), W))$ and a subsequence $\tau_n\to0^+$, as $n\to\infty$, such that $u(0^+)= u_0$  and 
\begin{align*}
u_{\tau_n}(t)\to u(t) \quad \text{narrowly  as $n\to \infty$},
\end{align*}
for all $t>0$.
\end{theorem}

\begin{proof}
First of all, since $u_{\tau}^{0}= G_{\omega(\tau)}*u_0$ (recall that we denote by $ G_t(x)$ the heat kernel at time $t$, $G_t(x) = \frac{1}{(4t\pi)^{d/2}} e^{-\frac{|x|^2}{4t} }$) by definition, we have that $W^2(u_{\tau}^{0}, u_0)\leq 2d\omega(\tau)\to 0$ as $\tau\to0^+$. This implies on the one hand that $\limsup_{\tau\to 0^+} W( u_{\tau}^{0}, u_0)=0$, and a standard computation shows that $u_{\tau}^{0}\to u_0$ narrowly as $\tau\to 0^+$. Secondly, given that $\widehat{G}_{\omega(\tau)}\to 1$ and $0\leq \widehat{G}_{\omega(\tau)}\leq 1$ the Dominated Convergence Theorem implies 
\begin{align*}
\int_{\R^d} |\widehat{u}_{0}(\xi)|^2 |\widehat{G}_{\omega(\tau)}(\xi)|^2\psi^{-1}(\xi)\d \xi\to  
\int_{\R^d} |\widehat{u}_0(\xi)|^2\psi^{-1}(\xi)\d \xi,\quad\text{as $\tau\to0^+$}.
\end{align*}
Equivalently we get that $\cF( u_{\tau}^{0})\to\cF( u_0)$ as $\tau\to0^+$. Lastly, we want to show the existence of $u\in AC^2([0,\infty), (\cP_2(\R^d), W))$ which is a narrow limit of a subsequence $(u_{\tau_n})_n$. Accordingly, for $t>0$ we define $m_\tau(t)$ by 
\begin{align*}
    m_\tau(t):=\begin{cases}
        \dfrac{W(u_\tau(t), u_\tau(t-\tau))}{\tau},& t-\tau>0\\[1em]
        \dfrac{W(u_\tau(t), u_\tau(0))}{\tau},& t-\tau\leq 0. 
    \end{cases}
\end{align*}
The estimate $W^2(u_{\tau}^{k-1}, u_{\tau}^{k})\leq  2\tau(\cF(u_{\tau}^{k-1})-\cF(u_{\tau}^{k}))$ (see Proposition \ref{prop:discrete}) implies 
\begin{align*}
    \int_0^{\infty}m^2_\tau(t)\d t
    &= \sum_{k=0}^\infty \int_{\tau k}^{\tau(k+1)} \frac{W^2(u_\tau(t), u_\tau(t-\tau))}{\tau^2}\d t = 2\sum_{k=0}^\infty  \frac{W^2(u_{\tau}^{k+1}, u_{\tau}^{k})}{2\tau}\\
    &\leq 2\sum_{k=0}^\infty  \cF(u_{\tau}^{k})-\cF(u_{\tau}^{k+1}) = 2\lim_{N\to \infty} \cF(u_{\tau}^{0})-\cF(u_{\tau}^{N+1})\\
    &\leq 2\cF(u_{\tau}^{0}). 
\end{align*}
Due to the uniform boundedness of the family $(m_\tau)_\tau$ in $L^2((0,\infty))$, it converges weakly, up to a subsequence, to some $m\in L^2([0,\infty))$, as $\tau\to0^+$. Moreover, the weak lower semicontinuity of the $L^2$-norm implies 
\begin{align*}
    \int_{0}^\infty m^2(t)\d t\leq \liminf_{\tau\to0^+}\int_{0}^\infty m^2_\tau(t)\d t\leq 2\cF(u_0). 
\end{align*}
Now, for $0\leq t_1<t_2$ we set $N_i(\tau) := \lfloor t_i/\tau\rfloor$ so that $\tau N_i(\tau)\leq t_i<\tau (N_i(\tau)+1)$. Taking into account the fact that, $u_{\tau}(t_i)= u_{\tau}^{N_i}$ or $u_{\tau}(t_i)= u_{\tau}^{N_i+1}$ depending whether $t_i/\tau$ is an integer or not, the triangle inequality yields  
\begin{align*}
    W(u_{\tau}(t_1), u_{\tau}(t_2))
    &\leq 	W( u_{\tau}(t_1), u_{\tau}^{N_1(\tau)+1})+ \sum_{k=N_1(\tau)+1}^{N_2(\tau)}W(u_{\tau}^{k-1},u_{\tau}^{k})
    + W(u_{\tau}^{N_2(\tau)}, u_{\tau}(t_2))\\
    &\leq 	 \sum_{k=N_1(\tau)+1}^{N_2(\tau)+1}W(u_{\tau}^{k-1},u_{\tau}^{k})
    =\sum_{k=N_1(\tau)+1}^{N_2(\tau)+1} \int_{\tau k}^{\tau(k+1)} \frac{W(u_{\tau}(t),u_{\tau}(t-\tau))}{\tau}\d t\\
    &= \int_{\tau N_1(\tau)+\tau}^{\tau N_2(\tau)+\tau} m_\tau(t)\d t.
\end{align*}
Since $\tau N_i(\tau)\to  t_i$, the weak convergence of $(m_\tau)_\tau$ entails 
\begin{align}
    \label{eq:limsup-abs-cont}
    \limsup_{\tau\to 0^+}W(u_{\tau}(t_1), u_{\tau}(t_2))
    &\leq \lim_{\tau\to 0^+}\int_{\tau N_1(\tau)}^{\tau N_2(\tau)} m_\tau(t)\d t
    =\int_{t_1}^{t_2} m(t)\d t, 
\end{align}
Combining this  with the Cauchy-Schwartz inequality yields 
\begin{align}
    \label{eq:compactness-assumption1}
    \limsup_{\tau\to 0^+}W(u_{\tau}(t_1), u_{\tau}(t_2))
    &\leq \|m\|_{L^2((0, \infty))} |t_2-t_1|^{1/2}.
\end{align}
Next, let us  fix $T>0$. Estimate \eqref{eq:dist-uk-delta} (see Proposition \ref{prop:discrete}) and the fact that $\tau\lfloor t/\tau\rfloor\leq t$  yield
\begin{align*}
    \begin{split}
    W^2(u_{\tau}^{\lfloor t/\tau\rfloor}, \delta_0)&=\int_{\R^d} |x|^2u_{\tau}^{\lfloor t/\tau\rfloor}(x)\d x
    \leq R_T,
    \end{split}
\end{align*}
for all $0\leq t\leq T$, where 
\begin{align*}
R_T=\frac{8d}{\log(2)}\Big(1+T \cF(u_0)+ \int_{\R^d} |x|^2\d u_0(x)\Big).
\end{align*}
Consider  $\mathcal{K}_T=\{u\in \cP_2(\R^d): W^2(u, \delta_0)\leq R_T\}$ which is a bounded closed subset of $\cP_2(\R^d)$, and recall that $u_{\tau}(t)= u_{\tau}^{\lfloor t/\tau\rfloor}$. Note that
\begin{align}
    \label{eq:compactness-assumption2}
    \text{$\mathcal{K}_T$ is narrowly compact in $\cP_2(\R^d)$}, \quad\text{and}\quad u_{\tau}(t)\in \mathcal{K}_T,\,\,\,\text{for all $0\leq t<T$}.
\end{align}
In light of Eq. \eqref{eq:compactness-assumption1} and Eq. \eqref{eq:compactness-assumption2} and the refined version of the Arzel\'a-Ascoli Theorem (\emph{cf.} \cite[Proposition 3.3.1]{AGS08}) there is subsequence $\tau_n\to 0^+$ and  a curve $u: [0,\infty)\to (\cP_2(\R^d), W)$ such that 
$u_{\tau_n}(t)\to u(t)$ for almost all $t\geq0$. Estimate \eqref{eq:limsup-abs-cont} in conjunction with the narrow convergence of $(u_{\tau}(t))$ and the lower semicontinuity of $W$, see \cite[Proposition 7.1.3]{AGS08},  yield 
\begin{align*}
    W(u(t_1), u(t_2))\leq \liminf_{n\to\infty}W(u_{\tau_n}(t_1), u_{\tau_n}(t_2))\leq \int_{t_1}^{t_2} m(t)\d t.
\end{align*}
It follows that  $u\in AC^2([0, \infty), (\cP_2(\R^d), W))$ since $m\in L^2((0,\infty))$. In particular, since  $u_{\tau}^{0}=u_{\tau}(0)\to u_0$ narrowly as $\tau\to 0^+$ up to a modification of $u$ at $t=0$ we get $u_0= u(0^+)$.
\end{proof}

\section{Improved regularity of discrete solutions}\label{sec:flow-interchange}
In this section, we aim to establish better regularity properties of the solutions to minimization problem, Eq.
\eqref{eq:discrete}, using the flow interchange technique introduced by Matthes, McCann, and Savar{\'e} \cite{MMS09}. Let us first recall some basics of this technique. 
\begin{definition}[Displacement convex entropy]
Let $V\in C^1((0, \infty))\cap C([0,\infty))$ with $V(0)=0$ be a convex function such that 
$$
    \lim_{x\to 0^+}\frac{V(x)}{x^{\alpha}}>-\infty, \quad \text{for some}\   \alpha>\frac{d}{d+2},
$$
and the McCann  displacement convexity condition holds \cite{Mc97}, \emph{i.e.}, 
\begin{align*}
r\mapsto r^dV(r^{-d})\qquad\text{is convex and decreasing on $(0, \infty)$}. 
\end{align*}

In this case, the displacement convex entropy associated with $V$ is  the functional $\cV:\cP_2(\R^d) \to \R\cup\{\infty \}$ defined as follows 
\begin{align*}
\cV(u)=
\begin{cases}
\int_{\R^d} V(u(x))\d x, \quad &\text{ if $\d u(x)= u(x)\d x$}\\
\infty, \quad & \text{ otherwise}. 
\end{cases}
\end{align*}   
The effective domain of $\cV$ is denoted by $D(\cV)=\{u\in \cP_2(\R^d)\,:\, \cV(u)<\infty\}$.  It is worth mentioning that $\cV$ is lower semicontinuous with respect to the narrow convergence.  
\end{definition}
According to \cite[Theorem 11.1.4]{AGS08} see also \cite{MuSa20}, the displacement convex entropy $\cV$ generates a continuous semigroup $S_t:D(\cV)\to D(\cV)$
satisfying the following evolution variational inequality (EVI): 
\begin{align}
    \label{eq:evo-var-ineq}
    \frac12 W^2(S_tu, v)-\frac12W^2(u,v)\leq t(\cV(v)-\cV(S_tu)),\qquad\text{for all $u,v\in D(\cV),\, t>0$,}
\end{align}
where, by definition, $S_tu$ is the unique distributional (with respect to the narrow topology) solution of the Cauchy problem
\begin{align}
    \label{eq:cauhy-problem}
    \partial_t w= \Delta L_V(w), \qquad w(0)=u, 
\end{align}
with $L_V(u): = uV'(u)-V(u) $. The flow associated to $\mathcal V$, $(S_t)_t$, is a semigroup of contractions with respect to the Wasserstein distance $W$ and extends to $\overline{D(\cV)}=\cP_2(\R^d)$.
\begin{remark}\label{rem:regularizing effect}
It is worth emphasizing that 
the regularizing effect of the semigroup implies that $S_tu\in D(\cV)$ for all $u\in \cP_2(\R^d)$.
Analogously, if $u_0\in \cP_2(\R^d)$ then  the boundedness of $u^0_{\tau}=G_{\omega(\tau)}*u_0$  implies that we also have  $u^0_{\tau}\in D(\cV)$ for any displacement convex entropy. 
\end{remark}

\begin{example}\label{Ex:standard-entropy}
The standard example of a displacement convex entropy is obtained by considering the convex function $H(x)= x\log x-x$, $x> 0$ and $H(0)=0$, giving rise to the usual Boltzmann-Shannon entropy
\begin{align*}
    \cH(u)=\begin{cases}
    \int_{\R^d}u(x)\log u(x)- u(x)\d x, \quad &\text{ if $\d u(x)= u(x)\d x$}\\
\infty, \quad &\text{ otherwise}. 
\end{cases}	\end{align*}   
Indeed it is clear that $H\in C^1((0, \infty))\cap C([0,\infty))$, and it can be shown that $\lim_{x\to 0^+}\frac{H(x)}{x^{\alpha}}=0$, for all  $\alpha\in (\frac{d}{d+2}, 1).$ 
We also find that, $r^dH(r^{-d})= -d\log(r)- 1$  is convex and decreasing on $(0,\infty)$.
For completeness, let us point out the well-known fact that $L_H(u)= uH'(u)-uH(u)=u$, which,  according to Eq. \eqref{eq:cauhy-problem}, gives rise to heat semigroup since, here, $S_t u$ is the distributional solution to the heat equation $	\partial_t w= \Delta w $ and $ w(0)=u$.
\end{example}

\begin{definition}[Dissipation along the flow] The dissipation of $\cF$ along the flow $S_t$ associated to $\cV$ at the point $u\in D(\cF)$ is defined as 
\begin{align*}
\fD_{\cV}\big[\cF\big](u):= \limsup_{t\to0^+}\frac{ \cF(u) - \cF(S_t u)}{t}. 
\end{align*}	
\end{definition}
\begin{theorem}[Flow interchange]\label{thm:dissp-ineq-V}
Let $(u_{\tau}^{k})_{k}$ be a sequence solving Eq. \eqref{eq:discrete} and $\cV$ be a displacement convex entropy. If, for all $k\in \N$, $\fD_{\cV}[\cF](u_{\tau}^{k})>-\infty$, then $u_{\tau}^{k} \in D(\cV)$ and 
\begin{align*}
\fD_{\cV}\big[\cF\big](u_{\tau}^{k})\leq  \frac{\cV(u_{\tau}^{k-1} ) - \cV(u_{\tau}^{k})}{\tau}\qquad\text{for all $k\geq1$}. 
\end{align*}
\end{theorem}
The statement and proof can be found, for example, in \cite{MMS09}. The next result infers that $ \fD_{\cH}( u_{\tau}^{k})>-\infty$, for any $k\in\N$, that is, we consider the flow interchange for the particular case $\cV=\cH$.

\begin{lemma}\label{lem:disp-entropy-H}
Let $u_0\in D(\cF)$ and $(u_{\tau}^{k})$ be defined as in Eq. \eqref{eq:discrete}. 
For any $k\geq0$ we have $u_{\tau}^{k}\in \Hnuw \cap D(\cH)$, (recalling that $\Hnuw=\Hnuwd\cap L^2(\R^d)$). Moreover we have 
\begin{align*}
    \|u_{\tau}^{k}\|^2_{\Hnuwd}\leq \fD_{\cH}[\cF](u_{\tau}^{k}),\quad\text{ for any $k\geq0$}.  
\end{align*}
\end{lemma}

\begin{proof}
First, we need to show that  $t\mapsto \cF(w_t)$ belongs to $C^1((0,\infty))\cap C([0,\infty))$. 
By definition $u_{\tau}^{0}= G_{\omega(\tau)}*u_0$, and therefore $u_{\tau}^{0}\in D(\cF)$.
Recall that the flow associated to $\cH$ is the heat  semigroup, viz., $w_t:=S_t u_{\tau}^{k}= G_{t} *u_{\tau}^{k}$. In the Fourier variables, we have $\widehat{w}_t(\xi)= \widehat{G}_t(\xi) \widehat{u}_{\tau}^{k}(\xi)=\widehat{u}_{\tau}^{k}(\xi)e^{-t|\xi|^2}.$  
In particular, we have $\widehat{w}_t \in C^1((0,\infty))\cap C([0, \infty))$ and $\partial_t\widehat{w}_t(\xi)=-|\xi|^2\widehat{w}_t(\xi)$. It is not difficult to see that
\begin{align*}
    \left| \partial_t |\widehat{w}_t(\xi)|^2 \right| \ \psi^{-1}(\xi)= 
2|\widehat{w}_t(\xi)|^2|\xi|^2\psi^{-1}(\xi) =2|\widehat{w}_t(\xi)|^2\widetilde{\psi}(\xi). 
\end{align*}
For fixed  $\eps>0$ and all $t>\eps$, we can estimate 
\begin{align*}
    |\widehat{w}_t(\xi)|^2 \widetilde{\psi} (\xi)
    &\leq\max_{r\geq0} re^{-2rt}|\widehat{u}_{\tau}^{k}(\xi)|^2 \psi^{-1}(\xi)\\
    &=\frac{1}{2te}|\widehat{u}_{\tau}^{k}(\xi)|^2 \psi^{-1}(\xi)\\
    &\leq \frac{1}{2\eps e}|\widehat{u}_{\tau}^{k}(\xi)|^2 \psi^{-1}(\xi). 
\end{align*}
We know that  $\cF(u_{\tau}^{k})\leq \cF(u_{\tau}^{k-1})\leq \cdots \leq \cF(u_{\tau}^{0}) <\infty$, \emph{i.e.}, 
$ \||\widehat{u}_{\tau}^{k}|^2 \psi^{-1} \|_{L^1(\R^d)}= 2\cF(u_{\tau}^{k})<\infty$. In particular, we get $w_t= S_tu_{\tau}^{k}\in \Hnuwd$, since the previous estimate implies 
\begin{align*}
\|w_t\|^2_{\Hnuwd}\leq \frac{1}{2te}\|u_{\tau}^{k}\|^2_{\Hnuid}= \frac{1}{te} \cF(u_{\tau}^{k})<\infty. 
\end{align*}
Next, again for  $\eps>0$ fixed, we find that $ |\widehat{u}_{\tau}^{k}|^2 \psi^{-1}\in L^1(\R^d)$
and, by combining the two preceding estimates, we have
\begin{align}
    \label{eq:dominated-bound}
    \left|  \partial_t|\widehat{w}_t(\xi)|^2\right| \psi^{-1}(\xi)= |\widehat{w}_t(\xi)|^2\widetilde{\psi}(\xi)\leq \frac{1}{2\eps e}|\widehat{u}_{\tau}^{k}(\xi)|^2 \psi^{-1}(\xi), 
\end{align}
for all $t>\eps, \,\,\xi\in \R^d$
On the one hand, Leibniz rule together with the estimate in Eq. \eqref{eq:dominated-bound} implies that $t\mapsto \cF(w_t)$ is differentiable on $(\eps, \infty)$ and we have 
\begin{align*}
\frac{\d}{\d t}\cF(w_t)&
= \frac12\frac{\d}{\d t}\int_{\R^d} |\widehat{w}_t(\xi)|^2\psi^{-1}(\xi)\d \xi=\int_{\R^d} \operatorname{\mathcal{R}e}(\overline{\widehat{w}_t}(\xi)\, \partial_t\widehat{w}_t(\xi)) \psi^{-1}(\xi)\d \xi\\
&=-\int_{\R^d}|\widehat{w}_t(\xi)|^2|\xi|^2\psi^{-1}(\xi)\d \xi=-\|w_t\|^2_{\Hnuwd}. 
\end{align*}
On the other hand, by the dominated convergence theorem, Eq. \eqref{eq:dominated-bound} also implies that $t\mapsto -\|w_t\|^2_{\Hnuwd}=\frac{\d}{\d t}\cF(w_t)$ is continuous on $(\eps,\infty)$. Since $\eps>0$ is arbitrarily chosen we deduce that $t\mapsto \cF(w_t)$ belongs to $C^1((0,\infty))$ with 
\begin{align*}
    \frac{\d}{\d t}\cF(w_t)= -\|w_t\|^2_{\Hnuwd}=- \|S_t u_{\tau}^{k}\|^2_{\Hnuwd}. 
\end{align*}  
On the other side, $|\widehat{w}_t(\xi)|^2= e^{-2t|\xi|^2}|\widehat{u}_{\tau}^{k}(\xi)|^2\to |\widehat{u}_{\tau}^{k}(\xi)|^2$ as $t\to0^+$ and
$|\widehat{w}_t(\xi)|^2\leq|\widehat{u}_{\tau}^{k}(\xi)|^2.$ The dominated convergence theorem implies $\cF(w_t)\to \cF(u_{\tau}^{k})$ as $t\to0^+$; which proves the continuity of $t\mapsto \cF(w_t)$ at $t=0$.  Therefore, $t\mapsto \cF(w_t)$ belongs to $C^1((0,\infty))\cap C([0,\infty))$. 

Next, the fundamental theorem of calculus implies 
\begin{align*}
\frac{\cF(u_{\tau}^{k})-\cF(S_tu_{\tau}^{k})}{t}
&=\frac{\cF(w_0)-\cF(w_t)}{t} 
= \frac{1}{t}\int_0^t\|w_r\|^2_{\Hnuwd}\d r.
\end{align*}
Additionally, by Fatou's Lemma there holds
\begin{align*}
    \|u_{\tau}^{k}\|^2_{\Hnuwd}
    \leq \liminf_{t\to 0^+}\| w_t\|^2_{\Hnuwd}. 
\end{align*} 
Combining the two estimates, it follows that 
\begin{align*}
\fD_{\cH}[\cF](u_{\tau}^{k})=\limsup_{t\to0^+} \frac{\cF(u_{\tau}^{k})-\cF(S_tu_{\tau}^{k})}{t}
\geq \|u_{\tau}^{k}\|^2_{\Hnuwd}. 
\end{align*}
We deduce from Theorem \ref{thm:dissp-ineq-V} that, for any $k\geq1$, we have $u_{\tau}^{k}\in D(\cH)$ and \begin{align*}
\|u_{\tau}^{k}\|^2_{\Hnuwd} \leq \fD_{\cH}[\cF](u_{\tau}^{k}) \leq 
\frac{\cH(u_{\tau}^{k-1})- \cH(u_{\tau}^{k})}{\tau}. 
\end{align*}
Therefore, we may deduce $u_{\tau}^{k}\in \Hnuwd\cap D(\cH)$. Finally, it remains to show that $u_{\tau}^{k}\in L^2(\R^d)$. For $|\xi|\geq 1$, it follows from Proposition \ref{prop:bound-levy-symbol} that $\psi(\xi)\leq \kappa_\nu(1+|\xi|^2) \leq 2\kappa_\nu|\xi|^2$. Then, using Plancherel, we find
\begin{align*}
\int_{\R^d} |\widehat{u}^{k}_{\tau}(\xi)|^2\d \xi 
&= 
\int_{|\xi|\leq 1}  |\widehat{u}^{k}_{\tau}(\xi)|^2\d \xi+ \int_{|\xi|>1} |\widehat{u}^{k}_{\tau}(\xi)|^2\psi(\xi)\psi^{-1}(\xi) \d \xi \\
&\leq |B_1(0)|\|\widehat{u}^{k}_{\tau}\|^2_{L^\infty(\R^d)}+ 2\kappa_\nu\int_{|\xi|>1}  |\widehat{u}^{k}_{\tau}(\xi)|^2|\xi|^2\psi^{-1}(\xi) \d \xi\\
&\leq |B_1(0)|\|\widehat{u}^{k}_{\tau}\|^2_{L^\infty(\R^d)}+ 2\kappa_\nu 
\|u_{\tau}^{k}\|^2_{\Hnuwd}
\end{align*} 
where we also used the fact that $u^{k}_{\tau}$ is a probability density. Hence, we deduce $u^{k}_{\tau}\in L^2(\R^d)$, and therefore $u^{k}_{\tau}\in \Hnuw= \Hnuwd\cap L^2(\R^d)$.
\end{proof}
\medskip 

The following result is an immediate consequence of Theorem \ref{thm:dissp-ineq-V} and Lemma \ref{lem:disp-entropy-H}. 
\begin{corollary}\label{cor:disp-entropy-H}
Let $u_0\in D(\cF)$ and $(u_{\tau}^{k})_k$ be defined as in Eq. \eqref{eq:discrete}. 
Then $u_{\tau}^{k}\in \Hnuw\cap D(\cH)$ for any $k\geq0$. Moreover, we have 
\begin{align*}
\|u_{\tau}^{k}\|^2_{\Hnuwd} \leq 
\frac{\cH(u_{\tau}^{k-1})- \cH(u_{\tau}^{k})}{\tau}, \quad\text{ for any $k\geq1$}.
\end{align*}
In particular $ \cH(u_{\tau}^{k})\leq \cH(u_{\tau}^{k-1})$. 
\end{corollary}
Next, we use the lifting (perturbation) of the entropy technique introduced in \cite{MMS09} to show that $u^k_\tau$ is in the domain of every displacement convex entropy. 
\begin{theorem}\label{thm:general-flow-interchange}
Assume that $\widetilde{\psi}(\xi)=|\xi|^2\psi^{-1}(\xi)$ is the symbol associated with a radial  L\'{e}vy kernel $\widetilde{\nu}$. 
Let $(u^k_\tau)_k$ be the sequence of  Definition \ref{def:discrete}. 
Let $\cG$ be a displacement convex entropy with density function $G$. Then, for any $k\geq0$,  $u^k_\tau\in D(\cG)$ and we have 
\begin{align*}
    0\leq\big(u^k_\tau, L_G(u^k_\tau)\big)_{\widetilde{\psi}}
    \leq \mathscr{D}_{\cG}[\cF](u^{k}_\tau)
    \leq \frac{\cG(u^{k-1}_\tau)- \cG(u^{k}_\tau)}{\tau}. 
\end{align*}
\end{theorem}

\begin{proof}
Note that $u^{0}_\tau= G_{\omega(\tau)}*u$ clearly belongs to $D(\cG)$. For fixed $\eps>0$, consider the perturbed displacement convex entropy 
\begin{align*}
\cV(u)= \cG(u)+ \eps\cH(u). 
\end{align*}
Let us denote by $S_t$ the flow associated to $\cV$. For fixed $k\geq1, \tau>0$, and $\eps>0$ we set $w_t = S_tu^k_\tau$ which is the unique solution to the generalized porous medium equation 
\begin{align*}
    \partial_t w_t = \Delta \Phi (w_t)= \div(\Phi'(w_t)\nabla w_t),
    \quad \text{and}\quad w_0= u^k_\tau, 
\end{align*}
where $\Phi(v)=L_G(v)+\eps v$, and   $L_G(v)= vG'(v)-G(v)$ as before. Since $G$ is convex we have $\Phi'(r)= rG''(r)+\eps\geq \eps$, for $r>0$, that is $\Phi$ is monotone increasing, and the above equation non-degenerate. Note that since $u^k_\tau\in L^1(\R^d)$ and $u^k_\tau\geq0$, each $w_t$ is strictly positive, see \cite[Chapter 3]{Vaz07} and bounded since $\Phi'(r)
\geq \eps>0$, see for instance \cite{Ben78,Ver70,BeBe85,Vaz06,Vaz05}.  The facts that  $w_t\geq 0$, $\Phi'(w_t)\geq0$,  and $\partial_t w_t= \Delta \Phi (w_t)$ imply that
\begin{align*}
\frac{\d}{\d t}\int_{\R^d}|w_t(x)|^2\d x
= 2\int_{\R^d}w_t(x)\Delta \Phi( w_t(x))\d x= -2\int_{\R^d}|\nabla w_t(x)|^2 \Phi'( w_t(x))\d x\leq 0. 
\end{align*}
Since $w_0=u^k_\tau\in L^2(\R^d)\cap \Hnuwd$, by Lemma \ref{lem:disp-entropy-H}, we  find that  $w_t\in L^2(\R^d)$  and 
\begin{align*}
\|w_t\|_{L^2(\R^d)}\leq \|u^k_\tau\|_{L^2(\R^d)}. 
\end{align*}
Note that the narrow convergence of ${w}_t$ to ${u}^{k}_{\tau}$ implies that $\widehat{w}_t(\xi)\to \widehat{u}^{k}_{\tau}(\xi)$ as $t\to 0$ for all $\xi$. Furthermore, Fatou's Lemma yields $\|w_t\|_{L^2(\R^d)}\to \|u^k_\tau\|_{L^2(\R^d)}$ as $t\to 0$. These two results together with the pointwise convergence $\widehat{w}_t\to \widehat{u}^{k}_{\tau}$, as $t\to 0$, yield
\begin{align*}
\lim_{t\to 0}\|w_t-u^k_\tau\|_{L^2(\R^d)}=0. 
\end{align*}
Therefore, a subsequence (not relabeled) satisfies that $w_t\to u^{k}_\tau$, a.e. as $t\to 0$. According to Theorem \ref{thm:energy-fourier}, for $u,v\in \Hnuw$ we have 
\begin{align*}
\big( u , v\big)_{\widetilde{\psi}}: = \int_{\R^d} \widehat{u}(\xi) \overline{\widehat{v}}(\xi)\widetilde{\psi}(\xi)\, \d\xi=  \iil_{\mathbb{R}^d \mathbb{R}^d}(u(x)-u(y))(v(x)-v(y))\widetilde{\nu}(x-y)\,\d y\,\d x.\end{align*}
Since  $L_G$ increases, $(u(x)-u(y))(L_G\circ u(x)-L_G\circ u(y))\geq0$, so that Fatou's Lemma  yields 
\begin{align*}
\liminf_{t\to 0}\big( L_G\circ w_t , w_t\big)_{\widetilde{\psi}} + \eps\big(w_t , w_t\big)_{\widetilde{\psi}}\geq\big( L_G\circ u^k_\tau,u^k_\tau\big)_{\widetilde{\psi}} + \eps\big(u^k_\tau,u^k_\tau\big)_{\widetilde{\psi}}. 
\end{align*} 
\noindent Recall that $L_G$ is locally Lipschitz and increasing. For every function $u\in  L^\infty(\R^d)$, 
\begin{align*}
0\leq (L_G\circ u(x)-L_G\circ u(y))^2\leq C_L(u(x)-u(y))(L_G\circ u(x)-L_G\circ u(y))\leq C_L^2 (u(x)-u(y))^2, 
\end{align*}
where $C_L>0$ is a Lipschitz constant of $L_G$ depending on $u$. This immediately implies 
\begin{align}\label{eq:local-lipscitz}
0\leq\big( L_G\circ u ,  L_G\circ u\big)_{\widetilde{\psi}} 
\leq C_L\big( L_G\circ u , u\big)_{\widetilde{\psi}}
\leq C_L^2\big(u , u\big)_{\widetilde{\psi}}. 
\end{align}
Next, we exploit these estimates with $u = w_t\in  L^\infty(\R^d)$. Differentiating $t\mapsto \cF(w_t)=\frac{1}{2}\|w_t\|^2_{\Hnuid}$ gives
\begin{align*}
\frac{\d}{\d t}\cF(w_t)&
=
\int_{\R^d}\operatorname{\mathcal{R}e}(\overline{\widehat{w}_t}(\xi)\, \partial_t\widehat{w}_t(\xi)) \psi^{-1}(\xi)\d \xi\\
&=-\int_{\R^d}
\big(\widehat{L_G\circ w_t}(\xi)\overline{\widehat{w}_t}(\xi)+\eps |\widehat{w}_t(\xi)|^2\big)|\xi|^2\psi^{-1}(\xi) \d \xi
\\
&=-\big( L_G\circ w_t , w_t\big)_{\widetilde{\psi}} -\eps\big(w_t , w_t\big)_{\widetilde{\psi}}\leq 0. 
\end{align*}
Accordingly, we have $\cF(w_t)\leq \cF(u^{k}_{\tau})$ and hence  $w_t\in\Hnuid$. Let us recall that the narrow convergence implies $\widehat{w}_t(\xi)\to \widehat{u}^{k}_{\tau}(\xi)$, as $t\to 0$ for all $\xi$, and that by Fatou's Lemma we deduce $\|w_t\|_{\Hnuid}\to \|u^k_\tau\|_{\Hnuid}$ as $t\to 0$, that is  $t\mapsto \cF(w_t)$
is continuous at $t=0$. Then, the fundamental theorem of calculus yields 
\begin{align*}
\frac{\cF(u^k_\tau)-\cF(w_t) }{t}
&=\frac{1}{t}\int_0^t\big( L_G\circ w_r , w_r\big)_{\widetilde{\psi}} +\eps\big(w_r , w_r\big)_{\widetilde{\psi}} \d r\geq \frac{\eps}{t}\int_0^t\big(w_r , w_r\big)_{\widetilde{\psi}} \d r.  
\end{align*}
Without loss of generality, we can assume that $w_t\in \Hnuwd$ and hence by the estimates given in Eq. \eqref{eq:local-lipscitz} that $L_G\circ w_t\in \Hnuwd$. By definition we have 
\begin{align*}
\mathscr{D}_{\cV}[\cF](u_{\tau}^{k})=\limsup_{t\to0} \frac{\cF(u_{\tau}^{k})-\cF(S_tu_{\tau}^{k})}{t}
\geq\big( L_G\circ u^k_\tau,u^k_\tau\big)_{\widetilde{\psi}} + \eps\big(u^k_\tau,u^k_\tau\big)_{\widetilde{\psi}}.  
\end{align*}
In particular, $\mathscr{D}_{\cG}[\cF](u_{\tau}^{k})\geq0$. Furthermore, if $S'_t$ is the flow associated to $\cG$ then  
\begin{align*}
\frac{\cF(u^k_\tau)-\cF(S'_t u^k_\tau) }{t}
&=\frac{1}{t}\int_0^t\big( L_G\circ S'_r u^k_\tau , S'_r u^k_\tau\big)_{\widetilde{\psi}} \d r\geq 0. 
\end{align*}
By Theorem \ref{thm:dissp-ineq-V}, we have that $u_{\tau}^{k}\in D(\cV)= D(\cH)\cap D(\cG)$ for any $k\geq1$ and 
\begin{align*}
    0 \leq \big(L_G\circ u^k_\tau,u^k_\tau\big)_{\widetilde{\psi}}
    &\leq  \mathscr{D}_{\cV}[\cF](u_{\tau}^{k}) \leq 
    \frac{\cV(u_{\tau}^{k-1})- \cV(u_{\tau}^{k})}{\tau}\\
    &=\frac{\cG(u_{\tau}^{k-1})-\cG(u_{\tau}^{k})}{\tau}
    +\eps\frac{\cH(u_{\tau}^{k-1})- \cH(u_{\tau}^{k})}{\tau}, 
\end{align*}
and also 
\begin{align*}
0\leq   \mathscr{D}_{\cG}[\cF](u_{\tau}^{k}) 
&\leq \frac{\cG(u_{\tau}^{k-1})-\cG(u_{\tau}^{k})}{\tau}.
\end{align*}
Thus, we get $0\leq\big( L_G\circ u^k_\tau,u^k_\tau\big)_{\widetilde{\psi}} <\infty$ and, upon letting $\eps\to0$, we obtain
\begin{align*}
0\leq\big( L_G\circ u^k_\tau,u^k_\tau\big)_{\widetilde{\psi}}
\leq  \mathscr{D}_{\cV}[\cF](u_{\tau}^{k}) \leq \frac{\cG(u_{\tau}^{k-1})- \cG(u_{\tau}^{k})}{\tau}.
\end{align*}
\end{proof} 

\noindent  In particular, for the specific class of functionals $G(u)=\frac{u^p}{p-1}$ we have $L_G(u)= u^p$, which we will use to get $L^p$-control.
\begin{corollary}\label{cor:lp-discrete-monotone}
Assume $u_0\in L^p(\R^d)$ and let $(u^k_\tau)_k$ be the sequence of  Definition \ref{def:discrete}. Assume that $\widetilde{\psi}(\xi)=|\xi|^2\psi^{-1}(\xi)$ is the symbol associated with a L\'{e}vy kernel $\widetilde{\nu}$. Then $u^k_\tau\in L^p(\R^d)$ for any $k\geq0$ and for all $k\geq1$ we have 
\begin{align*}
\|u^{k}_\tau\|_{L^p(\R^d)}\leq \|u^{k-1}_\tau\|_{L^p(\R^d)}. 
\end{align*}
\end{corollary}

As a consequence, we are able to prove Theorem \ref{thm:main-theorem}, item $(vi)$.
\begin{theorem}\label{thm:entropy-boundedness}
Assume $u_0\in D(\cH)$ and let $u\in AC^2([0,\infty), (\cP_2, W))$ be the limit curve obtained in 
Theorem \ref{thm:mini-mvt-scheme}.   Then, for all $t\geq0$, we have 
\begin{align*}
\cH(u(t)) \leq \cH(u_0). 
\end{align*}
\end{theorem}

\begin{proof}
Recalling that $u^{0}_\tau=G_{\omega(\tau)}*u_{0}$ and $u_\tau(t)= u^{\lceil t/\tau\rceil}_{\tau}$, by Corollary \ref{cor:disp-entropy-H} we have 
\begin{align*}
\cH(u_\tau(t))\leq \cH(u^{0}_\tau)\leq \cH(u_0). 
\end{align*}
Since $\cH$ is convex, the Banach-Saks theorem see \cite{Fog23BS} implies, the lower semicontinuity of $\cH$ with respect to the narrow convergence so that 
\begin{align*}
\cH(u(t))\leq \liminf_{\tau\to 0}\cH(u^{0}_\tau)\leq \cH(u_0). 
\end{align*}
\end{proof}

\begin{theorem}\label{thm:lp-boundedness}
Assume that $\widetilde{\psi}(\xi)=|\xi|^2\psi^{-1}(\xi)$ is the symbol associated with a L\'{e}vy kernel $\widetilde{\nu}$. Assume $u_0\in L^p(\R^d)$ and let $u\in AC^2([0,\infty), (\cP_2, W))$ be the limit curve obtained in 
Theorem \ref{thm:mini-mvt-scheme}.   Then, for all $t\geq0$, we have 
\begin{align*}
\|u(t)\|_{L^p(\R^d)}\leq \|u_0\|_{L^p(\R^d)}. 
\end{align*}
\end{theorem}

\begin{proof}
Recalling that $u^{0}_\tau=G_{\omega(\tau)}*u_{0}$ and $u_\tau(t)= u^{\lceil t/\tau\rceil}_{\tau}$, by Corollary \ref{cor:lp-discrete-monotone} we have 
\begin{align*}
\|u_\tau(t)\|_{L^p(\R^d)}\leq \|u^{0}_\tau\|_{L^p(\R^d)}\leq \|u_0\|_{L^p(\R^d)}. 
\end{align*}
Since $\|\cdot\|^p_{L^p(\R^d)}$ is convex, the Banach-Saks theorem see \cite{Fog23BS} implies the lower semicontinuity of $\|\cdot\|_{L^p(\R^d)}$ with respect to the narrow convergence so that
\begin{align*}
\|u(t)\|_{L^p(\R^d)}\leq \liminf_{\tau\to 0}\|u_\tau(t)\|_{L^p(\R^d)}\leq \|u_0\|_{L^p(\R^d)}.
\end{align*}
\end{proof}

\medskip 

\begin{corollary}\label{cor:bound-utau}
Let $u_0\in D(\cF)$ and $(u_{\tau}^{k})$ defined as in Eq. \eqref{eq:discrete} and consider the piecewise constant approximation $u_\tau(t)= u_{\tau}^{\lceil t/\tau\rceil}$. 
For every $t>0$, $u_{\tau}(t)\in \Hnuwd$ . Moreover, for $T> T_0\geq\tau> 0$, we have 
\begin{align*}
\int_{T_0}^T\|u_{\tau}(t)\|^2_{\Hnuwd}\d t\leq \cH(u_{\tau}^{N_0(\tau)})
+ \tilde \kappa \left(1 + \cF(u_0)+ \int_{\R^d} |x|^2\d u_0(x)\right).
\end{align*}
where $N_0(\tau)=\lfloor T_0/\tau\rfloor$ and $\tilde \kappa>0$ is a constant only depending on $T$ and $d$.  
\end{corollary}

\begin{proof}
Set $N=\lfloor T/\tau\rfloor$  and $N_0= N_0(\tau)=\lfloor T_0/\tau\rfloor$
so that $(T_0,T)\subset (\tau N_0, \tau( N+1))$. Thus, 
\begin{align*}
\int_{T_0}^T\|u_{\tau}(t)\|^2_{\Hnuwd}\d t&\leq \int_{\tau N_0}^{\tau(N+1)}\|u_{\tau}(t)\|^2_{\Hnuwd}\d t\\
&=  \sum_{k=N_0}^N \int_{\tau k}^{\tau(k+1)} \|u_{\tau}(t)\|^2_{\Hnuwd}\d t
= \sum_{k=N_0}^{N} \tau \|u_{\tau}^{k+1}\|^2_{\Hnuwd}\\
&\leq \sum_{k=N_0}^N  \big(\cH(u_{\tau}^{k})- \cH(u_{\tau}^{k+1})\big)
= \cH(u_{\tau}^{N_0})- \cH(u_{\tau}^{N+1}),
\end{align*}
having used Corollary \ref{cor:disp-entropy-H}.
Once we prove the inequality
\begin{align}
    \label{eq:carleman-inequality}
    -\cH(u_{\tau}^{N+1})\leq \kappa\Big(1+\int_{\R^d} |x|^2 u_{\tau}^{N+1}(x)\d x \Big),
\end{align}
then, together with the estimate in Eq. \eqref{eq:dist-uk-delta}, we have
\begin{align*}
    -\cH(u_{\tau}^{N+1})
    &\leq \frac{8d \kappa }{\log(2)}\Big(1+\tau (N+1)\cF(u_0)+ \int_{\R^d} |x|^2\d u_0(x) \Big)\\
    &\leq  \frac{8d\kappa}{\log(2)}\Big(1+2T\cF(u_0)+ \int_{\R^d} |x|^2\d u_0(x)\Big).
\end{align*}
Hence, it is enough to prove
Inequality \eqref{eq:carleman-inequality} to complete the proof. To this end, let us proceed as follows
\begin{align*}
    -\int_{\R^d} u_{\tau}^{N+1}  \log u_{\tau}^{N+1}   \d x 
    = &-\int_{\R^d}u_{\tau}^{N+1} \log u_{\tau}^{N+1} \mathds{1}_{\{ u_{\tau}^{N+1} 
    \geq \exp(-|x|^2) \}} \d x \\
    &- \int_{|x|\leq1}u_{\tau}^{N+1} \log u_{\tau}^{N+1} \mathds{1}_{\{ u_{\tau}^{N+1} \leq \exp(-|x|^2) \}} \d x \\
    &-\int_{|x|\geq1}u_{\tau}^{N+1} \log u_{\tau}^{N+1} \mathds{1}_{\{ u_{\tau}^{N+1} \leq \exp(-|x|^2) \}} \d x.
\end{align*}
First of all, we note that if $u_{\tau}^{N+1} \geq \exp(-|x|^2)$ then 
$-\log u_{\tau}^{N+1} \leq |x|^2$, so that 
\begin{align*}
-\int_{\R^d}u_{\tau}^{N+1} \log u_{\tau}^{N+1} \mathds{1}_{\{ u_{\tau}^{N+1} 
\geq \exp(-|x|^2) \}} \d x \leq \int_{\R^d}u_{\tau}^{N+1}|x|^2\d x. 
\end{align*}
For the second term, we have 
\begin{align*}
- \int_{|x|\leq1}u_{\tau}^{N+1} \log u_{\tau}^{N+1} \mathds{1}_{\{ u_{\tau}^{N+1} \leq \exp(-|x|^2) \}} \d x \leq |B_1(0)| \max_{t \in [0,1]} t| \log t | \leq C.
\end{align*}
Last, the monotonicity of $t\mapsto t\log t$ implies
\begin{align*}
-\int_{|x|\geq1}u_{\tau}^{N+1} \log u_{\tau}^{N+1} \mathds{1}_{\{ u_{\tau}^{N+1} \leq \exp(-|x|^2) \}} \d x 
&\leq  \int_{|x|\geq1} u_{\tau}^{N+1} |\log u_{\tau}^{N+1} |\mathds{1}_{\{ u_{\tau}^{N+1} \leq \exp(-|x|^2) \}} \d x \\
&\leq \int_{\mathbb{R}^d} |x|^2\exp(-|x|^2) \d x \leq C.
\end{align*}
Hence, we obtain the desired bound for $u_{\tau}^{N+1}$.
\end{proof}

\section{Convergence of discrete approximations}\label{sec:convergence-discrete}
In this section, we establish the convergence of the piecewise constant interpolations of the approximate solutions $(u_\tau)_\tau$  and the associated discrete pressures $ (v_\tau)_\tau$  
(recalling $v_\tau = L^{-1} u_\tau$, \emph{i.e.}, $v_\tau$ is defined so that $\widehat{v}_\tau(\xi)= \psi^{-1}(\xi)\widehat{u}_\tau(\xi)$) 
in appropriate spaces. In Section \ref{sect:dissipationsolution}, we prove that the limit curve obtained in Theorem \ref{thm:mini-mvt-scheme} satisfies Eq. \eqref{eq:main_eqn} in the sense of Eq. \eqref{eq:def:weak-solutions}.
Let us start with the following observation. 
\begin{theorem}
\label{thm:interpolating} Let $u\in \Hnuid\cap \Hnuwd $ then we have 
\begin{align*}
    \|u\|^2_{L^2(\R^d)} &\leq 2\kappa_\nu \big(\|u\|^2_{\Hnuwd} + \|u\|^2_{\Hnuid}\big), 
\end{align*} 
with $\kappa_\nu = 2\|\nu\|_{L^1(\R^d,1\land|h|^2)}$.  Now if the  condition \eqref{eq:lower-bound-cond} is satisfied, then we have
\begin{align*}
    \|u\|^2_{\Hnusd} &\leq c_\nu^{-1}\Big( \|u\|^2_{\Hnuwd} + \|u\|^2_{\Hnuid} \Big)
\end{align*} 
In particular, the embedding $ \Hnui\cap \Hnuw\hookrightarrow \Hnus$ is continuous.

\end{theorem}

\begin{proof}
Note that  $\psi(\xi)\leq 2\kappa_\nu|\xi|^2$ for $|\xi|>1$, and $\psi(\xi)\leq 2\kappa_\nu$ for $|\xi|\leq 1$. Using this observation, we find
\begin{align*}
    \|u\|^2_{L^2(\R^d)} 
    &=\int_{|\xi|>1} |\widehat{u}(\xi)|^2\psi(\xi)\psi^{-1}(\xi) \d \xi + \int_{|\xi|\leq 1} |\widehat{u} (\xi)|^2\psi(\xi)\psi^{-1}(\xi) \d \xi\\
    &\leq 2\kappa_\nu \int_{|\xi|>1} |\widehat{u} (\xi)|^2 \widetilde{\psi}(\xi) \d \xi + 2\kappa_\nu \int_{|\xi|\leq 1} |\widehat{u} (\xi)|^2 \, \psi^{-1}(\xi) \d \xi\\
    &\leq 2\kappa_\nu \left(\|u\|^2_{\Hnuwd}+\|u\|^2_{\Hnuid}\right),
\end{align*}
which proves the first statement. Concerning the second statement, let us use Condition \eqref{eq:lower-bound-cond}, \emph{i.e.}, $\psi^{-1}(\xi) \leq \frac{c_\nu^{-1}}{1\land |\xi|^2}$. Again, decomposing the domain of integration into high and low frequencies, we may estimate both regimes separately and obtain
\begin{align*}
    \|u\|^2_{\Hnusd} 
    &=\int_{|\xi|>1} |\widehat{u}(\xi)|^2\psi^*(\xi) \d \xi + \int_{|\xi|\leq 1} |\widehat{u} (\xi)|^2\psi^*(\xi) \d \xi\\
    &= \int_{|\xi|>1} |\widehat{u} (\xi)|^2\psi^{-1}(\xi)\widetilde{\psi}(\xi) \d \xi + \int_{|\xi|\leq 1} |\widehat{u} (\xi)|^2|\xi|^2\psi^{-1}(\xi)\, \psi^{-1}(\xi) \d \xi\\
    &\leq  c_\nu^{-1} \int_{|\xi|>1} |\widehat{u} (\xi)|^2\frac{1}{1\land |\xi|^2} \widetilde{\psi}(\xi) \d \xi + c_\nu^{-1} \int_{|\xi|\leq 1} |\widehat{u} (\xi)|^2\frac{|\xi|^2}{1\land |\xi|^2} \, \psi^{-1}(\xi) \d \xi\\
    &=c^{-1} \int_{|\xi|>1} |\widehat{u} (\xi)|^2\widetilde{\psi}(\xi) \d \xi 
    + c_\nu^{-1} \int_{|\xi|\leq 1} |\widehat{u} (\xi)|^2 \, \psi^{-1}(\xi) \d \xi \\
    &\leq c_\nu^{-1} \big(\|u\|^2_{\Hnuwd} + \|u\|^2_{\Hnuid} \big). 
\end{align*}
From the two estimates the continuity of the embedding $\Hnui\cap \Hnuw\hookrightarrow \Hnus$ follows.   
\end{proof}

\begin{theorem}
\label{thm:convergence-discrete}
Let $u_0\in \Hnuid\cap \cP_2(\R^d)$, $(u_\tau)_\tau$ be the piecewise constant approximations in Definition \ref{def:discrete} and $u$ its limit curve obtained in Theorem \ref{thm:mini-mvt-scheme}. Define $v_\tau= L^{-1} u_\tau $ and $v= L^{-1} u$. 
Then, there is a subsequence (not relabeled) $(\tau_n)_n$, such that following hold:  

\begin{enumerate}[$(i)$]
\item We have $u\in L^2(0,T; \Hnuw)$ and 
\begin{align*}
    &\text{$u_{\tau_n} \rightharpoonup u$, \qquad  weakly in $L^2(0,T; \Hnuw)$}.
\end{align*}
If, in addition, Condition \eqref{eq:compactness-ass} is met, \emph{i.e.}, there holds
\begin{align*}
    \sup_{\xi\in \R^d}\frac{1}{\widetilde{\psi}(\xi)}
    |e^{i\xi\cdot h}-1|^2= \sup_{\xi\in \R^d}\frac{\psi(\xi)}{|\xi|^2}
    |e^{i\xi\cdot h}-1|^2\xrightarrow{|h|\to0}0, 
\end{align*}
then, for any $0<T_0<T$, we have 
\begin{align*}
    &\text{$u_{\tau_n}\to u$, \qquad strongly in $L^2(T_0,T; L^2_{\loc}(\R^d))$}. 
\end{align*}

\item If Condition \eqref{eq:lower-bound-cond} is met, we have $\nabla v\in L^2(0,T; L^2(\R^d))$, and
\begin{align*}
    &\text{$\nabla v_{\tau_n} \rightharpoonup \nabla v$, \qquad weakly in $L^2(0,T; L^2(\R^d)).$ }
\end{align*}
\end{enumerate}
\end{theorem}

\begin{proof}
By Plancherel's Theorem, we deduce that 
\begin{align*}
    &\|\nabla v_\tau \|^2_{L^2(\R^d)} = \int_{\R^d} |\widehat{v}_\tau(\xi)|^2 |\xi|^2\d \xi
    = \int_{\R^d} |\widehat{u}_\tau(\xi)|^2 |\xi|^2\psi^{-2}(\xi)\d \xi
    = \|u_\tau \|^2_{\Hnusd}.
\end{align*}
Moreover, by Theorem \ref{thm:interpolating}, we know that 
\begin{align*}
    \|u_\tau \|^2_{L^2(\R^d)} \leq C\left(\|u_\tau \|^2_{\Hnuwd} +  \|u_\tau \|^2_{\Hnuid}\right), 
\end{align*}
and, using Condition \eqref{eq:lower-bound-cond}, we also have 
\begin{align*}
    \|u_\tau \|^2_{\Hnusd} \leq C\left(\|u_\tau \|^2_{\Hnuwd}+ \|u_\tau \|^2_{\Hnuid}\right). 
\end{align*}
Note that the narrow convergence  $u_{\tau_n}\to u$ implies the pointwise convergence of the Fourier transforms, \emph{i.e.}, $\widehat{u}_{\tau_n} \to \widehat{u}$. Similarly, $\widehat{v}_{\tau_n}(\xi)\to \widehat{v}(\xi)$, for all $\xi \in \R^d$. Therefore, all claimed weak convergences are true once the boundedness of $(u_\tau)_\tau$ in $\Hnuwd \cap \Hnuid$ is established. 
By Proposition \ref{prop:discrete}, it follows that  $\cF(u_\tau)\leq  \cF(u_0) $. That is, for any $T_0 \in (0,T)$
\begin{align*}
    \int_{T_0}^T \|u_\tau(t)\|^2_{\Hnuid} \d t \leq (T-T_0) \|u_0\|^2_{\Hnuid}.
\end{align*}
By Corollary \ref{cor:bound-utau}, we know that for $T> T_0\geq\tau> 0$,
\begin{align*}
    \int_{T_0}^T\|u_{\tau}(t)\|^2_{\Hnuwd}\d t\leq \cH(u_{\tau}^{N_0(\tau)})
    + \tilde \kappa \left(1+\cF(u_0)+ \int_{\R^d} |x|^2\d u_0(x)\right).
\end{align*}
where $N_0(\tau)=\lfloor T_0/\tau\rfloor$ and $\tilde \kappa$ is a constant only depending on $T$ and $d$.

Since $u_0\in D(\cH)$ we have $\cH(u_{\tau}^{N_0(\tau)}) \leq \cH(u_0) < \infty$. Moreover, under Condition \eqref{eq:compactness-ass}, the embedding $\Hnuw\hookrightarrow L^2_{\loc}(\R^d)$ is compact in virtue of Theorem \ref{thm:local-comp-symb}. Thence, the strong convergence of $(u_{\tau_n})_{n}$ in  $L^2(T_0,T; L^2_{\loc}(\R^d))$ follows. 
\end{proof}

\section{Weak solution of the equation and energy dissipation inequality}
\label{sect:dissipationsolution}
This section is dedicated to establishing the energy dissipation inequality, item $(v)$ of Theorem \ref{thm:main-theorem}. 
Moreover, we identify the limit curve as a weak solution of Eq. \eqref{eq:main_eqn}, Theorem \ref{thm:main-theorem} $(iv)$, in the sense of Eq. \eqref{eq:def:weak-solutions}. 
To this end, let us begin by deriving the associated Euler-Lagrange equations.

\begin{theorem}
\label{thm:euler-lagrange-discrete}
Assume that ${\nu} \notin L^1(\R^d)$ and satisfies Condition \eqref{eq:lower-bound-cond}. Moreover, assume the symbol $\widetilde{\psi}$ is associated with a unimodal L{\'e}vy kernel $\widetilde{\nu}$ satisfying the following condition: 
\begin{align}
    \label{eq:double-cond-origin-bis}
    \text{For any $0<\lambda<1$ there is $c_\lambda>0$ s.t. $\widetilde{\nu}(\lambda h)\leq c_\lambda\widetilde{\nu} (h)$ whenever $|h|\leq 1$}.
\end{align} 
Let $u_{0} \in \Hnuid \cap \mathcal{P}_2(\R^d)$ and let $(u^{k}_{\tau})_k$ be the associated solution to the minimizing movement scheme, Eq. \eqref{eq:discrete}. Define $(v^{k}_{\tau})_k$ by ${v}^{k}_{\tau} := L^{-1} {u}^{k}_{\tau}$. Then, for any $k \geq 0$, we have 
\begin{align}
    \label{eq:firsteq}
    \int_{\R^{d}} \nabla {v}^{k}_{\tau} \cdot \eta u^{k}_{\tau} \, \d x = \frac{1}{\tau}  \int_{\R^{d}}\big(T_{u^{k}_{\tau}}^{u^{k-1}_{\tau}} - \mathrm{I}\big) \cdot \eta u^{k}_{\tau} \, \d x, \qquad\text{for all \, $\eta \in C^{\infty}_{c} (\R^{d}, \R^{d})$}.
\end{align}
We recall that $T_u^v$ denotes the transport plan between $u$ and $v$, see Section \ref{sect:exist}. Moreover, we have 
\begin{align}
    \label{eq:secondeq}
      \int_{\R^{d}} \big| \nabla {v}^{k}_{\tau} \big|^{2} u^{k}_{\tau} \, \d x = 
    \frac{1}{\tau^{2}} W^{2}\big(u^{k}_{\tau}, u^{k-1}_{\tau} \big).
\end{align}
\end{theorem}

\begin{proof}
\textbf{Step 1. -- Perturbation of minimizers.}\\
For ease of notation, throughout the proof we shall simply write $u := {u}^{k}_{\tau}$ and $v := v^{k}_{\tau}$. Given $\delta>0$, we define $\phi_{\delta}: \R^{d} \to \R^{d}$ as $\phi_{\delta}(x) := x + \delta \eta(x)$. Clearly, for $\delta_0>0$ small enough, we have 
 \begin{align}\label{eq:det-bound}
     \frac{1}{2}\leq \det (D\phi_\delta(x))\leq\frac{3}{2}, 
 \end{align}
 for all $x\in \R^d$, $\delta\in [0,\delta_0]$. Define $u_{\delta} := {\phi_{\delta}}_\# u^{k}_{\tau}= \det(D \phi_\delta)^{-1} u^{k}_{\tau}\circ \phi_\delta^{-1}$ and $\widehat{v}_{\delta} := \psi^{-1} \widehat{u_\delta}$.  Since $u$ is optimal in Eq. \eqref{eq:discrete}, there holds
\begin{align*}
    0 \leq \frac{1}{\delta}
    \left[ 
    \mathcal{F}(u_{\delta}) - \mathcal{F}(u) + \frac{1}{2 \tau}
    \left( 
    W^{2}(u_{\delta}, u^{k-1}_{\tau}) -  W^{2}(u, u^{k-1}_{\tau}) 
    \right) 
    \right].
\end{align*}
The second term is classical which is known to satisfy
\begin{align*}
   \lim_{\delta \to 0} \frac{1}{\delta}
   \left[ 
   \frac{1}{2 \tau}
   \left( 
   W^{2}(u_{\delta}, u^{k-1}_{\tau}) - W^{2}(u, u^{k-1}_{\tau}) 
   \right) 
   \right]
   = \frac{1}{\tau}  \int_{\R^{d}}\big(T_{u^{k}_{\tau}}^{u^{k-1}_{\tau}} - \mathrm{I}\big) \cdot \eta u^{k}_{\tau} \, \d x,
\end{align*}
which can be adapted from \cite[Theorem 8.13] {Vil03}, see also \cite[Section 7.2.2]{San15}. The rest of the proof focuses on the treating the limit

\begin{align*}
   \lim_{\delta \to 0} \frac{1}{\delta}
   \big[ 
        \mathcal{F}(u_{\delta}) - \mathcal{F}(u) 
   \big].
\end{align*}
Switching to Fourier, let us rewrite the energy
\begin{align}
    \label{eq:polardecomp}
    \frac{1}{\delta}
    \big[ \mathcal{F}(u_{\delta}) - \mathcal{F}(u) 
    \big] 
    = \frac{1}{2\delta}
    \int_{\R^d} \big(|\widehat{u_\delta}(\xi)|^2- |\widehat{u}(\xi)|^2\big)\psi^{-1}(\xi)\d \xi.
\end{align}
Using the identity 
$$ 
    |\widehat{u}_\delta(\xi)|^2 - |\widehat{u}(\xi)|^2 = \big(\overline{ \widehat{u}_\delta}(\xi) - \overline{\widehat{u}}(\xi)\big) \big(\widehat{u}_\delta(\xi) +\widehat{u}(\xi)\big) + \overline{\widehat{u}_\delta}(\xi) \widehat{u}(\xi) - \widehat{u}_\delta(\xi) \overline{\widehat{u}}(\xi),
$$
in conjunction with $\widehat{v}(\xi) = \psi^{-1}(\xi)  \widehat{u}(\xi)$ and $\widehat{v}_{\delta}(\xi)=  \psi^{-1}(\xi) \widehat{u}_{\delta}(\xi)$, the variation of the energy can be simplified such that
\begin{align}
    \label{eq:factorization}
    \frac{1}{\delta}
    \big[ 
        \mathcal{F}(u_{\delta}) - \mathcal{F}(u) 
    \big]
    = \frac{1}{2} \int_{\R^{d}} |\xi| \big( \widehat{v}_{\delta}(-\xi) + \widehat{v}(-\xi) \big) \cdot \frac{1}{\delta} |\xi|^{-1}\big( \widehat{u}_{\delta}(\xi) - \widehat{u}(\xi) \big) \, \d \xi.
\end{align}
Next, let $R>1$ be sufficiently large such that $\supp\eta\subset B(0,R)$. By definition of $u_\delta$, we have 
\begin{align*}
    \widehat{u}_{\delta}(\xi)- \widehat{u}(\xi)
    &= \int_{\R^d} \exp(- i \xi \cdot (x + \delta \eta(x)) ) u(x) \, \d x- \int_{\R^d} \exp(- i \xi \cdot x) u(x) \, \d x\\
    &=\int_{B(0,R)} \exp(- i \xi \cdot x) \ \big(\exp(- i \xi \cdot \delta \eta(x)) -1\big)\ u(x) \, \d x.
\end{align*} 
Then, the dominated convergence theorem implies the pointwise convergence $\widehat{u}_{\delta}(\xi)\to \widehat{u}(\xi)$, as $\delta \to 0$. Therefore, we also have $|\xi| \widehat{v}_{\delta}(-\xi)\to |\xi|\widehat{v}(-\xi)$ as $\delta \to 0$, pointwise.

\medskip

\textbf{Step 2. -- Boundedness of $(u_{\delta})_\delta$ in $\Hnuw$.}\\
Note that by Corollary \ref{cor:disp-entropy-H} we have $u\in \Hnuw$. According to Theorem \ref{thm:energy-fourier}, the existence of $\widetilde{\nu}$ implies that $H_{\widetilde{\nu}}(\R^d) = \Hnuw$, and we have to estimate
\begin{align*}
    \|u_\delta\|^2_{\Hnuw} 
    &= \iil_{\R^d\R^d} |u_\delta(x)-u_\delta(y)|^2\widetilde{\nu}(x-y)\d y\d x\\
    &\leq \iil_{\R^d\R^d} \left[2|u\circ \phi_\delta^{-1}(x)-u\circ \phi_\delta^{-1}(y)|^2|h_\delta(x)|^2+ 2|h_\delta(x)-h_\delta(y)|^2|u_\delta(y)|^2\right] \widetilde{\nu}(x-y) \d y\d x,
\end{align*}
having used $u_\delta= h_\delta u\circ \phi_\delta^{-1}$, where $h_\delta:=\det(D \phi_\delta)^{-1} $. Since $h_\delta (x) = \det((I+\delta\, D\eta)^{-1} (x) )$ is at least $W^{1,\infty}(\R^d)$, there exists $A>0$ independent of $\delta >0$ such that $\|h_\delta\|_{W^{1,\infty}(\R^d)}\leq A$. Hence, for all $x,y\in \R^d$ we have 
\begin{align*}
    |h_\delta(x) |\leq A, \quad\text{as well as}\quad |h_\delta(x) -h_\delta(y)|\leq 2A(1\land|x-y|). 
\end{align*}
Therefore, 
\begin{align}
    \label{eq:u_delta_intermediate}
    \begin{split}
    \|u_\delta\|^2_{\Hnuw} 
    &\leq 2A^2 \iil_{\R^d\R^d} |u\circ \phi_\delta^{-1}(x)-u\circ \phi_\delta^{-1}(y)|^2 \widetilde{\nu}(x-y)\d y\d x \\
    &\quad + 4A^4 \iil_{\R^d\R^d} (1 \land |x-y|^2) |u\circ \phi_\delta^{-1}(y)|^2 \widetilde{\nu}(x-y)\d y\d x,
    \end{split}
\end{align}
Now, by Theorem \ref{thm:push-bi-Lipschtz}, we find
\begin{align}
    \label{eq:bound-delta_1}
    \|u\circ \phi_\delta^{-1}\|^2_{\Hnuw} \leq 
    C\big(1+\| \det D\phi_\delta \|_{L^\infty(\R^d)}\big)^2\|u\|^2_{\Hnuw}
    &\leq C \|u\|^2_{\Hnuw},  
\end{align} 
where $C>0$ is independent of $\delta$ by Eq. \eqref{eq:det-bound}.  Using Eq. \eqref{eq:bound-delta_1} in Eq. \eqref{eq:u_delta_intermediate}, we finally have 
\begin{align}
    \label{eq:delta-uniform1}
    \begin{split}
    \|u_\delta\|^2_{\Hnuw}
    &\leq 4A^2\Big(1 + A^2 \int_{\R^d} 1\land|h|^2\widetilde{\nu}(h)\d h\Big)\|u\circ \phi_\delta^{-1}\|^2_{\Hnuw}\\
    &\leq C\|u\|^2_{\Hnuw}, 
    \end{split}
\end{align}
where $C>0$ is independent of $\delta$.

\medskip 

\textbf{Step 3. -- Strong convergence $\nabla v_{\delta} \to \nabla v$ in $L^{2}_{\loc}(\R^d)$.}\\
The goal is to apply the compactness Theorem \ref{thm:local-compactness}. Therefore, it is sufficient to establish the boundedness of $\nabla(v_\delta-v)$ in $\Hnu= \Hnup$, see Remark \ref{rem:symbol}. Under Condition \eqref{eq:lower-bound-cond} and by proceeding as in the proof of Theorem \ref{thm:interpolating}, we obtain

\begin{align}
    \label{eq:difference-grad-v_delta}
    \begin{split}
    \| \nabla v_{\delta} - \nabla v \|^2_{L^{2}(\R^d)} &= \|u_{\delta} - u\|^2_{\Hnusd} \\
    &\leq c_\nu^{-1} \int_{|\xi|>1} |\widehat{u_{\delta}} (\xi) - \widehat{u} (\xi)|^2\widetilde{\psi}(\xi) \d \xi 
    + c_\nu^{-1} \int_{|\xi|\leq 1} |\widehat{u_{\delta}} (\xi) - \widehat{u} (\xi)|^2 \, \psi^{-1}(\xi) \d \xi \\
    &\leq c_\nu^{-1} \int_{|\xi|>1} |\widehat{u_{\delta}} (\xi) - \widehat{u} (\xi)|^2\widetilde{\psi}(\xi) \d \xi 
    + c_\nu^{-1} \int_{|\xi|\leq 1} |\widehat{u_{\delta}} (\xi) - \widehat{u} (\xi)|^2 \, |\xi|^{-2} \d \xi \\
    &\leq  c_\nu^{-1} \left( \|u_{\delta}-u\|^2_{\Hnuwd} + \|u_{\delta}-u\|^2_{\dot{H}^{-1}(\R^d)} \right).
    \end{split}
\end{align} 
Next, let us show that $\|u_{\delta}-u\|^2_{\dot{H}^{-1}(\R^d)} \leq C  \|u_{\delta}-u\|^2_{L^2(\R^d)}$ for some $C>0$ not depending on $\delta>0$, which then implies 
$$
    \| \nabla v_{\delta} - \nabla v \|^2_{L^{2}(\R^d)} \lesssim \| u_{\delta} - u \|^2_{H^{\tilde \psi}(\R^d)}.
$$
Let us recall that the Riesz kernel
$$
    K_{1}(x) := 
    \left\{
    \begin{array}{ll}
        \frac{1}{2\pi}\log|x|, & d = 2, \\[1em]
         C_{d,-1}|x|^{2-d}, & d\geq 3,
    \end{array}
    \right.
$$  
satisfies $\widehat{K_{1}}(\xi)= |\xi|^{-2}$, see for instance Eq. \eqref{eq:fourier-riesz-potential}-\eqref{eq:riesz-operator}. As the one-dimensional case is particular, we will treat it separately below and focus on $d\geq 2$. Now, we may estimate
\begin{align*}
    \|u_{\delta}-u\|^2_{\dot{H}^{-1}(\R^d)} 
    = \int_{\R^d} |\widehat{u_{\delta}} (\xi)-\widehat{u} (\xi)|^2 \, \widehat{K_1}(\xi) \d \xi=  \int_{B(0,R)} (u_{\delta}-u)(x) K_1*(u_{\delta}-u) (x)\d x.
\end{align*}
In the last equality, we exploited the fact that $\supp (u_{\delta}-u)\subset B(0,R)$ and $R>1$ independent of $\delta>0$. Moreover, we have
\begin{align*}
    \begin{split}
     K_1*(u_{\delta}-u) (x)
    &= \int_{\R^{d}} (u_{\delta}(y) - u(y))  \mathds{1}_{B(0,R)}(y)K_1(x-y) \d y,
    \end{split}    
\end{align*}
whence, for $x\in B(0,R)$, we obtain
\begin{align*}
    | K_1*(u_{\delta}-u)(x)|\leq 
    \int_{\R^{d}} |u_{\delta}(y) - u(y)|  \mathds{1}_{B(0,2R)}(x-y)K_1(x-y) \d y.
\end{align*}

This combined with Young's convolution inequality gives
\begin{align*}
    \|K_1*(u_{\delta} - u)\|_{L^{2}(B(0,R))} \leq \|K_1\|_{L^{1}(B(0,2R))} \|u_{\delta} - u\|_{L^{2}(\R^d)}.
\end{align*}
By Cauchy-Schwartz inequality,
\begin{align}
    \label{eq:u-u_delta_L2-control-dD}
    \begin{split}
    \|u_{\delta}-u\|^2_{\dot{H}^{-1}(\R^d)} 
    &\leq \|u_{\delta} - u\|_{L^{2}(\R^d)} \|(u_{\delta} -u)\ast K_1\|_{L^{2}(B(0,R))}\\
    &\leq \|K_1\|_{L^{1}(B(0,2R))} \|u_{\delta} - u\|^2_{L^{2}(\R^d)}. 
    \end{split}
\end{align}

Now, let us address the one-dimensional case and begin by introducing the anti-derivative of $u-u_\delta$, which is given by
\begin{align}
    \label{eq:x1dim-anti-derivative}
	U(x)
	&= \frac{1}{2} \int_{\R} (u - u_{\delta}) (y)  \operatorname{sgn}(x-y)\d y= \frac{1}{2} \int_{B(0,R)} (u - u_{\delta}) (y)  \operatorname{sgn}(x-y)\d y. 
\end{align}
Here, the notation $\operatorname{sgn}(a)= a / |a|$ denotes the sign of $a\neq0$. Since $\supp(u-u_\delta)\subset B(0,R)$ and $\int_{\R} u-u_\delta\d x=0 $, it follows that $\supp U\subset B(0,R)$ and $\widehat{U}(\xi)= -i\xi^{-1} (\widehat{u}- \widehat{u}_\delta) (\xi),$ $\xi\in \R$.  Moreover, we have 
\begin{align*}
	|U(x)|^2\leq R\|u-u_\delta\|^2_{L^2(\R)}. 
\end{align*}
Hence we get 
\begin{align}
    \label{eq:u-u_delta_L2-control-1D}
    \|u-u_\delta\|^2_{\dot{H}^{-1} (\R) } 
    =  \int_{\R} |\widehat{U}(\xi)|^2\d \xi=  \int_{B(0,R)} |U(x)|^2\d x \leq 2R^2\|u-u_\delta\|^2_{L^2(\R)},
\end{align}
which established the control of $\|u-u_\delta\|^2_{\dot{H}^{-1} (\R) } $ in terms of $ \|u-u_\delta\|^2_{L^2}$.

Substituting Eq. \eqref{eq:u-u_delta_L2-control-dD} (resp. Eq. \eqref{eq:u-u_delta_L2-control-1D} for $d=1$) into Eq. \eqref{eq:difference-grad-v_delta}, we obtain 
\begin{align*}
    \begin{split}
    \| \nabla v_{\delta} - \nabla v \|^2_{L^{2}(\R^d)} 
    &\leq  c_\nu^{-1}
    \left(\|u_{\delta}-u\|^2_{\Hnuwd} + \|u_{\delta}-u\|^2_{\dot{H}^{-1}(\R^d)} 
    \right)
    \leq C\|u_{\delta}-u\|^2_{\Hnuw}.
    \end{split}
\end{align*} 
In particular, Estimate \eqref{eq:delta-uniform1} implies
\begin{align}
    \label{eq:delta-unform2}
    \| \nabla v_{\delta} - \nabla v \|_{L^{2}(\R^d)}
    \leq C\|u\|_{\Hnuw}.
\end{align}  
Since  $|\widehat{\nabla v}_{\delta} (\xi)|= |\xi|\psi^{-1}(\xi) |\widehat{u}_\delta(\xi)|$, it is readily seen that 
\begin{align*}
    \|\nabla v_{\delta}- \nabla v\|^2_{\Hnupd} 
    = \|u_{\delta} -u\|^2_{\Hnuwd}.
\end{align*}
By the boundedness of $(u_\delta)_\delta$ in $\Hnuw$ from \textbf{Step 2.}, we find that 
\begin{align*}
    \|\nabla v_{\delta}- \nabla v\|_{\Hnu} 
    = \|\nabla v_{\delta}- \nabla v\|_{\Hnup}\leq C \|u_{\delta} -u\|_{\Hnuw}\leq C \|u\|_{\Hnuw}.
\end{align*}
Taking into account the pointwise convergence, $|\xi| \widehat{v}_{\delta}(-\xi)\to |\xi|\widehat{v}(-\xi)$ as $\delta \to 0$, the compactness Theorem \ref{thm:local-compactness} implies that $\nabla v_{\delta} \to \nabla v$ in $L^{2}_{\loc}(\R^d)$.  

\medskip

\textbf{Step 4. -- Weak convergence of $\frac{1}{\delta}  |\xi|^{-1}\big( \widehat{u}_{\delta}(\xi) - \widehat{u}(\xi) \big)$.}\\
This step is dedicated to showing
$$
    \frac{1}{\delta}  |\xi|^{-1}\big( \widehat{u}_{\delta}(\xi) - \widehat{u}(\xi) \big) \rightharpoonup -i |\xi|^{-1} \xi \cdot \widehat{\eta u}(\xi),
$$
weakly in $L^{2}(\R^d)$,  as $\delta \to 0 $. To prove this claim, note that, by the dominated convergence theorem, we get 
\begin{align}
    \label{eq:DOM}
    \begin{split}
    \lim_{\delta \to 0} \frac{1}{\delta}\big(  \widehat{u}_{\delta}(\xi) -  \widehat{u}(\xi) \big) 
    &=\lim_{\delta \to 0} \int_{B(0,R)} \frac{1}{\delta} \big(  \exp(- i \xi \cdot  \delta \eta(x) \big) - 1\big) \exp(- i \xi \cdot x) u(x) \, \d x \\
    &=- i \xi \cdot \int_{\R^d}  \exp(- i \xi \cdot x ) \eta(x) u(x) \, \d x\\
    &= -i\xi\cdot\widehat{(\eta u)}(\xi).  
    \end{split}    
\end{align}
In other words, for all $\xi \in \R^d$, letting $\delta \to 0 $ gives 
$$
    \frac{1}{\delta}  |\xi|^{-1}( \widehat{u}_{\delta}(\xi) - \widehat{u}(\xi) ) \to -i |\xi|^{-1} \xi \cdot \widehat{\eta u}(\xi). 
$$
 Now, for fixed $\xi \in \R^{d}$, let us define the function $g_{\xi} : [0, \infty)\to \R$,  $g_{\xi}: \delta \mapsto g_{\xi}(\delta) := \widehat{u}_{\delta}(\xi)$ and ${(\eta u)}_{\delta} := {\phi_{\delta}}_{\#} (\eta u)$. Once again, by the dominated convergence theorem, we obtain 
\begin{align*}
\begin{split}
    g_{\xi}'(\delta) &= \lim_{h \to 0} \frac{1}{h}\big( \widehat{u}_{\delta+h}(\xi) -  \widehat{u}_{\delta}(\xi) \big)\\
    &= \lim_{h \to 0} \int_{B(0,R)}  \frac{1}{h} \big(  \exp\big(- i \xi \cdot  h \eta(x) \big) - 1\big) \exp\big(- i \xi \cdot (x + \delta \eta(x)) \big) u(x) \, \d x \\
    &= - i \xi \cdot \int_{\R^d}  \exp\big(- i \xi \cdot (x + \delta \eta(x)) \big) \eta(x) u(x) \, \d x\\
    &= - i \xi \cdot \widehat{\big(\eta u\big)}_{\delta} (\xi). 
\end{split}    
\end{align*}
 Moreover, note that also by dominated convergence theorem, $g_{\xi}'$ is continuous, and therefore $g_{\xi}$ is of class $C^{1}$. Therefore by the intermediate value theorem, there is some $\delta_{\xi} \in (0, \delta)$ such that 
\begin{align*}
    \frac{1}{\delta}\big( \widehat{u}_{\delta}(\xi)- \widehat{u}(\xi) \big) = g_{\xi}'(\delta_{\xi}). 
\end{align*}
 Furthermore, for $\delta\in (0,\delta_0)$, 
\begin{align*}
\| \widehat{(\eta u)}_{\delta} \|_{L^{2}(\R^d)} = \| ({\eta u})_{\delta} \|_{L^{2}(\R^d)} \leq 2 \| {\eta u}\|_{L^{2}(\R^d)}.
\end{align*}
\noindent Therefore  $ \frac{1}{\delta}  |\xi|^{-1}( \widehat{u}_{\delta} - \widehat{u} )$ is bounded in $L^{2}(\R^d) $ uniformly in $\delta$, \emph{i.e.},
\begin{align}
    \label{eq:uniform-H-1-control}
    \frac1\delta\| |\xi|^{-1} (\widehat u_\delta - \widehat u)\|_{L^2(\R^d)}= \frac{1}{\delta} \| u_\delta - u \|_{\dot H^{-1}(\R^d)}  \leq C,
\end{align}
for some constant $C>0$ independent of $\delta >0$. In combination with the pointwise convergence, we have that, as $\delta \to 0 $, $ \frac{1}{\delta}  |\xi|^{-1}\big( \widehat{u}_{\delta}(\xi) - \widehat{u}(\xi) \big)$ converges to $-i |\xi|^{-1} \xi \cdot \widehat{\eta u}(\xi)$ weakly in $L^{2} (\R^d)$.    

\medskip 

 \textbf{Step 5. -- Variations of $\mathcal F$.} \\
 We claim that, for any probability density $u \in \Hnuw \cap \Hnuid$,
\begin{align}
    \label{eq:limit}
   \lim_{\delta \to 0} \frac{1}{\delta}
   \big[\mathcal{F}(u_{\delta}) - \mathcal{F}(u) \big] 
   =\int_{\R^{d}} \nabla{v}(x) \cdot {\eta u}(x) \, \d x.  
\end{align}
This step is divided into 2 parts, where the first part is dedicated to smooth densities and the second part less regular densities, respectively.\\

\underline{\textit{Substep 5.1} -- Smooth densities.} First, we prove Eq. \eqref{eq:limit} for a probability density $u \in C^{1}(\R^d) \cap L^{\infty}(\R^d)$. By Plancherel's Theorem and Eq. \eqref{eq:factorization}, we have, for any $\rho>0$  
\begin{align*}
\frac{1}{\delta}\big( \mathcal{F}(u_{\delta}) - \mathcal{F}(u) \big) 
&=\frac{1}{2} \int_{\R^{d}} -i\xi \big( \widehat{v}_{\delta}(-\xi) + \widehat{v}(-\xi) \big) \cdot \frac{1}{\delta} i\xi|\xi|^{-2}\big( \widehat{u}_{\delta}(\xi) - \widehat{u}(\xi) \big) \, \d \xi\\
   & = \frac{1}{2} \int_{\R^{d}}  \nabla \big( v + v_{\delta} \big)(x) \cdot \frac{1}{\delta}  \nabla  K_{1} *\big( u - u_{\delta} \big)(x)\, \d x\\
   &=\frac{1}{2}\int_{\R^{d} \setminus B(0,\rho)}\nabla ( v + v_{\delta})(x) \cdot \frac{1}{\delta}  \nabla  K_{1} * (u - u_{\delta})(x) \, \d x \\
   &\quad + \frac{1}{2}\int_{B(0,\rho)} \nabla ( v + v_{\delta})(x) \cdot \frac{1}{\delta}  \nabla  K_{1} * ( u - u_{\delta})(x) \, \d x. 
\end{align*} 
Since $\supp\eta\subset B(0,R)$, \textbf{Step 3.} and \textbf{Step 4.} imply by weak-strong lemma that for any $\rho >R$,
\begin{align*}
    \lim_{\delta \to 0} \frac{1}{2}\int_{B(0,\rho)} \nabla \big( v + v_{\delta} \big) (x) \cdot \frac{1}{\delta}  \nabla K_1*\big( u - u_{\delta} \big)(x) \, \d x 
    &= \int_{B(0,\rho)} \nabla{v}(x) \cdot {\eta u}(x) \, \d x\\
    &=  \int_{\R^d} \nabla{v}(x) \cdot {\eta u}(x) \, \d x.  
\end{align*}

Therefore, it is sufficient to prove that the term over $\R^d\setminus B(0,\rho)$ goes to 0 uniformly in $\delta$ as $\rho\to\infty$. Since $u \in C^{1}(\R^d) \cap L^{\infty}(\R^d)$, for some constant  $C$ depending on $\|u\|_{C^{1}(\R^d)}$ and $\|\eta\|_{C^{1}(\R^d)}$, we get 
\begin{align*}
   \left|  u(x) - u_{\delta}(x) \right| &\leq \left|  u(x) - u(\phi_{\delta}(x)) \right| + \left| u(\phi_{\delta}(x)) \left( \det \nabla \phi_{\delta} -    \mathrm{I}\right)\right|\leq C \delta,
\end{align*}
Since $\supp(u - u_{\delta})\subset B(0,R)$, for $|x|>\rho$ with $\rho>2R$ we get
\begin{align*}
    \nabla  K_{1} *\big( u - u_{\delta} \big)(x) =
    c_d \int_{B(0,R)}  \frac{(u - u_{\delta}) (y)(x-y)}{|x-y|^{d}} \d y.
\end{align*}
Thus, since $|x|\leq 2|x-y|$ for $y\in B(0,R)$ we have 
\begin{align*}
\big|\nabla K_{1} *\big( u - u_{\delta} \big)(x) \big|
\leq 
C\delta\int_{B(0,R)}  \frac{\d y}{|x|^{d-1}}= C\delta |x|^{1-d}. 
\end{align*}
and therefore, for $d \geq 2$, we obtain
\begin{align*}
    \sup_{\delta\in (0,\delta_0)} \Bigg| \int_{\R^{d} \setminus B(0,\rho)} &\nabla ( v + v_{\delta})(x)  \cdot \frac{1}{\delta}  \nabla K_1*( u - u_{\delta})(x) \, \d x \Bigg| \\
    &\leq  \sup_{\delta\in (0,\delta_0)} 
    \left\{
    \|\nabla ( v + v_{\delta})\|_{L^{2}(\R^{d})} 
    \ \left(\int_{\R^{d} \setminus B(0,\rho)} \big| \frac{1}{\delta}  \nabla K_1* \big( u - u_{\delta} \big) (x)\big|^{2}\d x \right)^{1/2}
    \right\} \\
    &\leq C \int_{\R^{d} \setminus B(0,\rho)} \frac{\d x}{|x|^{2(d-1)}} \xrightarrow{\rho\to\infty} 0.
\end{align*}
In the one-dimensional case $d=1$, observing that $i\xi|\xi|^{-2}= i\xi^{-1},$ $\xi\in \R$, we can legitimately identify 
$\nabla  K_{1} *\big( u - u_{\delta} \big)(x) =U(x)$ where  $U(x)$ is given in Eq. \eqref{eq:x1dim-anti-derivative},  that is,  
\begin{align*}
    \nabla  K_{1} *\big( u - u_{\delta} \big)(x) 
	&= \frac{1}{2} \int_{\R} (u - u_{\delta}) (y)  \operatorname{sgn}(x-y)\d y= \frac{1}{2} \int_{B(0,R)}  \frac{(u - u_{\delta}) (y)(x-y)}{|x-y|}  \d y.
\end{align*}

 Note that $|x|>R$ and $y\in B(0,R)$ we have $  \operatorname{sgn}(x-y) =  \operatorname{sgn}(x) $. So that 
\begin{align*}
	\nabla  K_{1} *\big( u - u_{\delta} \big)(x) 
	=  \operatorname{sgn}(x) \int_{B(0,R)} (u - u_{\delta}) (y)\d y=0. 
\end{align*}

\underline{\textit{Substep 5.2} -- General case.} Let probability density $u \in \Hnuw \cap \Hnuid$ be given and let $u^{\varepsilon} \in C^{1}(\R^d) \cap L^{\infty}(\R^d)$ be such that $\lim_{\varepsilon \to 0} \|u - u^{\varepsilon}\|_{\Hnuw }  + \|u - u^{\varepsilon}\|_{\Hnuid}= 0$, see for instance \cite[Chapter 3]{guy-thesis}. We define the push forward $u_{\delta}^{\varepsilon} := {\phi_{\delta}}_\# u^{\varepsilon}$, and the associated pressure $v_{\delta}^{\varepsilon} $ by $\widehat{v}_{\delta}^{\varepsilon}  := L^{-1} \widehat{u}_{\delta}^{\varepsilon} $. Using Cauchy-Schwartz and Eq. \eqref{eq:factorization}, we may write
\begin{align*}
    \begin{split}
            \frac{1}{\delta}
            &\left| \big(\mathcal{F}(u_{\delta}) - \mathcal{F}(u) \big) - \big(\mathcal{F}(u_{\delta}^{\varepsilon})- \mathcal{F}(u^{\varepsilon}) \big) \right|\\ 
            &\qquad = \frac{1}{2\delta} \Bigg|\int_{\R^{d}}\big( (\widehat{v}_{\delta}-\widehat{v}_{\delta}^{\varepsilon})(-\xi) + (\widehat{v}-\widehat{v}^{\varepsilon})(-\xi) \big) \, \big(\widehat{u}_{\delta}(\xi) - \widehat{u}(\xi) \big) \, \d \xi \\
            &\qquad \qquad \quad - \int_{\R^{d}}\big( \widehat{v}_{\delta}^{\varepsilon}(-\xi) + \widehat{v}^{\varepsilon}(-\xi) \big) \frac{1}{\delta} \big( (\widehat{u}^{\varepsilon}_{\delta}-\widehat{u}_{\delta})(\xi) - (\widehat{u}^{\varepsilon}-\widehat{u})(\xi) \big) \, \d \xi \Bigg|\\
            &\qquad \leq C \Big(\| u_{\delta}-u_{\delta}^{\varepsilon} \|_{\Hnusd}  + \| u-u^{\varepsilon} \|_{\Hnusd}  \Big) \frac{1}{\delta}\|u_{\delta} - u\|_{ \dot H^{-1}(\R^d)}\\
            &\qquad \qquad + C \| u^{\epsilon}+u_{\delta}^{\varepsilon} \|_{\Hnusd}  \frac{1}{\delta}\Big\|(u^{\varepsilon}_{\delta}-u_{\delta})-(u^{\varepsilon}-u )\Big\|_{\dot H^{-1}(\R^d)}.
    \end{split}
\end{align*}

Proceeding as in \textbf{Step 3},  Eq. \eqref{eq:difference-grad-v_delta}, we have that 
\begin{align*}
      \| u_{\delta}-u_{\delta}^{\varepsilon} \|_{\Hnusd}  &\leq C \left( \| u_{\delta}-u_{\delta}^{\varepsilon} \|_{\Hnuw} + \| u_{\delta}-u_{\delta}^{\varepsilon} \|_{\dot H^{-1}(\R^{d})} \right)\\
      &\leq C \left( \| u-u^{\varepsilon} \|_{\Hnuw} + \| u-u^{\varepsilon} \|_{\dot H^{-1}(\R^{d})} \right),  
\end{align*}
where $C>0$ is independent of $\delta>0$. Let us stress that the first term vanishes, \emph{i.e.}, $\| u-u^{\varepsilon} \|_{\Hnuw}\to 0$, as $\varepsilon \to 0$, by definition. 
Next, by the continuous embedding of Theorem \ref{teo:eq} $(ii)$, we have $\|u\|_{\dot H^{-1}} \leq C \|u\|_{\dot H^{\psi^{-1}}}$, which is bounded by assumption, whence we conclude that $u \in \dot H^{-1}(\R^{d})$.
Moreover, we have that $\| u-u^{\varepsilon} \|_{H^{-1}(\R^{d})} \to 0$, as $\varepsilon \to 0$. By Eq. \eqref{eq:uniform-H-1-control} in \textbf{Step 3},  we have that $ \frac{1}{\delta}\|u_{\delta} - u\|_{H^{-1}(\R^d)}$ is bounded uniformly in $\delta$. 
Furthermore, proceeding as in \textbf{Step 3}, we obtain a bound on $\| u_{\delta}+u_{\delta}^{\varepsilon} \|_{\Hnusd}$, uniform in $\delta$ and $\varepsilon$, see Eq. \eqref{eq:delta-unform2} and, additionally, we obtain that $\frac{1}{\delta}\|(u - u^{\varepsilon})-(u_{\delta} - u^{\varepsilon}_{\delta})\|_{\dot H^{-1}(\R^d)} \leq C \left\|u - u^{\varepsilon}\right\|_{L^{2}(\R^d)} $, see Eq. \eqref{eq:delta-uniform1}. 

Finally, since $\left\|u - u^{\varepsilon}\right\|_{L^{2}(\R^d)} + \left\|\nabla v - \nabla v^{\varepsilon}\right\|_{L^{2}(\R^d)} \to 0$, as $\varepsilon \to 0$, we have that
\begin{align*}
    \int_{\R^{d}} \nabla{v}^{\varepsilon}(x) \cdot {\eta u}^{\varepsilon}(x) \, \d x \to \int_{\R^{d}} \nabla{v}(x) \cdot {\eta u}(x) \, \d x .
\end{align*}
Putting everything together, we have that, for every $\varepsilon>0$,
\begin{align}
\label{eq:longerror}
    \begin{split}
        &\lim_{\delta \to 0} 
        \Big| \frac{1}{\delta} \Big(\mathcal{F}(u_{\delta}) - \mathcal{F}(u) \Big) - \int_{\R^{d}} \nabla{v}(x) \cdot {\eta u}(x) \, \d x \Big| \\
        &\qquad \qquad \leq  \lim_{\delta \to 0} \Big|  \frac{1}{\delta}\Big(\mathcal{F}(u_{\delta}) - \mathcal{F}(u) \Big) - \frac{1}{\delta}\Big(\mathcal{F}(u^{\varepsilon}_{\delta}) - \mathcal{F}(u^{\varepsilon}) \Big) \Big|\\
        &\qquad \qquad \qquad+\lim_{\delta \to 0} \Big| - \frac{1}{\delta}\Big(\mathcal{F}(u^{\varepsilon}_{\delta}) - \mathcal{F}(u^{\varepsilon}) \Big)  - \int_{\R^{d}} \nabla{v}^{\varepsilon}(x) \cdot {\eta u}^{\varepsilon}(x) \, \d x \Big|\\
        &\qquad \qquad  \qquad +\Big| \int_{\R^{d}} \nabla{v}^{\varepsilon}(x) \cdot {\eta u}^{\varepsilon}(x) \, \d x - \int_{\R^{d}} \nabla{v}(x) \cdot {\eta u}(x) \, \d x \Big|\\
        &\qquad \qquad \leq C \Big( \| u-u^{\varepsilon} \|_{\Hnuw} + \| u-u^{\varepsilon} \|_{H^{-1}(\R^{d})} + \| u-u^{\varepsilon} \|_{L^{2}(\R^{d})} \Big)\\
        &\qquad \qquad \qquad + \left| \int_{\R^{d}} \nabla{v}^{\varepsilon}(x) \cdot {\eta u}^{\varepsilon}(x) \, \d x - \int_{\R^{d}} \nabla{v}(x) \cdot {\eta u}(x) \, \d x \right|.
    \end{split}
\end{align}
Since \eqref{eq:longerror} is valid for any $\varepsilon>0$, we have shown claim of the second part.
\medskip 

\textbf{Step 6. --  Conclusion.}\\ Combining all pieces of the jigsaw, we obtain
\begin{align}
    \label{eq:inequality}
    0 \leq \int_{\R^{d}} \nabla{v}(x) \cdot {\eta u}(x) \, \d x - \frac{1}{\tau}  \int_{\R^{d}}\big(T_{u^{k}_{\tau}}^{u^{k-1}_{\tau}} - \mathrm{I}\big) \cdot \eta u^{k}_{\tau} \, \d x.
\end{align}
Replacing $\eta$ by $-\eta$ in Eq. \eqref{eq:inequality}, we have that 
\begin{align}
    \label{eq:inequality2}
    0 = \int_{\R^{d}} \nabla{v}(x) \cdot {\eta u}(x) \, \d x - \frac{1}{\tau}  \int_{\R^{d}}\big(T_{u^{k}_{\tau}}^{u^{k-1}_{\tau}} - \mathrm{I}\big) \cdot \eta u^{k}_{\tau} \, \d x.
\end{align}
We have thus proven Eq. \eqref{eq:firsteq}. Taking a sequence $\eta$ converging to $\nabla v $, and using the fact that 
\begin{align*}
     \int_{\R^{d}} \big|T_{u^{k}_{\tau}}^{u^{k-1}_{\tau}} - \mathrm{I}\big|^{2} u^{k}_{\tau} \, \d x =  W^{2}(u, u^{k-1}_{\tau}),
\end{align*}
we obtain Eq. \eqref{eq:secondeq}. 
\end{proof}

\medskip 

Finally, we show that the limiting curve $u$ is a weak solution, \emph{i.e.}, $u$ satisfies Eq. \eqref{eq:main_eqn}, hence giving the proof of Theorem \ref{thm:main-theorem} $(iv)$.
\begin{theorem}[Weak solution]\label{thm:weak-solution}
Assume that the conditions \eqref{eq:compactness-ass} and \eqref{eq:lower-bound-cond} are satisfied. Moreover, assume the symbol $\widetilde{\psi}$ is associated with a unimodal L{\'e}vy kernel $\widetilde{\nu}$ satisfying the following condition 
\begin{align}\label{eq:double-cond-origin-bis2}
\text{For any $0<\lambda<1$ there is $c_\lambda>0$ s.t. $\widetilde{\nu}(\lambda h)\leq c_\lambda\widetilde{\nu} (h)$ whenever $|h|\leq 1$}.
\end{align} 
Let $u_{0} \in \Hnuid \cap \mathcal{P}_2(\R^d)$. Let $v=L^{-1}u$ where $u$ is the limiting curve defined by Theorem \ref{thm:mini-mvt-scheme} and $v$ its associated pressure. Then, $u$ is a weak solution to Eq. \eqref{eq:main_eqn} in the following sense
\begin{align}
    \label{eq:weak-solo}
    \int_{0}^{\infty} \int_{\R^{d}} (\partial_{t} \varphi - \nabla \varphi \cdot \nabla v) u \, \d x \, \d t =0, \, \quad \forall \varphi \in C^{\infty}_{c} (\R^{d}\times (0, \infty)).
\end{align}
\end{theorem}

\begin{proof}
Given $\varphi \in C^{\infty}_{c} (\R^{d}\times (0, \infty))$, we use $\eta = \nabla_{x} \varphi$ in the Euler-Lagrange equation, Eq. \eqref{eq:firsteq}, in Theorem \ref{thm:euler-lagrange-discrete}, and integrate in time to get
\begin{align}
    \label{eq:soldisc}
    \sum_{k\in \N}\int_{(k-1)\tau}^{k \tau} \left\{\int_{\R^{d}} \nabla {v}^{k}_{\tau} \cdot \eta u^{k}_{\tau}  - \frac{1}{\tau}   \big(T_{u^{k}_{\tau}}^{u^{k-1}_{\tau}} - \mathrm{I}\big) \cdot \eta u^{k}_{\tau} \, \d x\right\} \d t = 0.
\end{align}
Using the definition of the piecewise constant interpolation, Theorem \ref{thm:convergence-discrete} allows us to pass of the limit in the first term of Eq. \eqref{eq:soldisc}, which converges to     
\begin{align*}
    \int_{0}^{+ \infty} \int_{\R^{d}} \nabla {v} \cdot \nabla_{x} \varphi u\, \d x  \, \d t .
\end{align*}
The convergence of the second term associated to the time derivative is classical, and its limit is
\begin{align*}
\int_{0}^{+ \infty} \int_{\R^{d}} \partial_{t} \varphi u \, \d x  \, \d t,
\end{align*}
see, for example, \cite[Theorem 11.1.6]{AGS05}, which concludes the proof.
\end{proof}

Finally, let us address the remaining outstanding item, Theorem \ref{thm:main-theorem}, (iv), that is, the energy dissipation inequality. To this end, we follow the usual argument relying on the De Giorgi interpolation.
\begin{definition}[De Giorgi variational interpolation]
We define the \emph{De Giorgi interpolant} $\widetilde{u}_{\tau} \in AC^{2} ([0, T), (\mathcal{P}_2(\R^d), W)) $ as $\widetilde{u}_{\tau}(k\tau) : = {u}_{\tau}(k\tau)$ for $k=1,2,...$, and 
\begin{align}
    \label{eq:variational}
    \widetilde{u}_{\tau} := \argmin_{u} \Big\{ \frac{1}{2( t - (k-1)\tau )} W^{2}(u, u^{k-1}_{\tau}) + \mathcal{F}(u)\Big\},
\end{align}
for $t \in ((k-1)\tau, k \tau)$ and $k \in \N$. We observe that Eq. \eqref{eq:variational} has a unique solution $\widetilde{u}\in \Hnuw$ which can be established as in the proof of Theorem \ref{thm:main-theorem} $(i)$ and $(ii)$. Moreover, the De Giorgi interpolation coincides with the discrete minimizers at multiples of $\tau$.
\end{definition}

We will now prove a proposition,  analog to Proposition 6.3 in \cite{LMS2018}, regarding De Giorgi variational interpolation, which we will use to prove the energy dissipation inequality, Theorem \ref{thm:main-theorem},  $(iv)$. 

\begin{proposition}
Let ${v}_{\tau} $ be given by $\widehat{\widetilde{v}_{\tau}} = L^{-1} \widehat{\widetilde{u}_{\tau}} $. Then, for all $N \in \mathbb{N}$, and $\tau>0$, 
\begin{align}
    \label{eq:fromprop}
    \frac{1}{2} \int_{0}^{N \tau} \int_{\R^{d}} \big| \nabla v_{\tau} \big|^{2} u_{\tau} \, \d x \, \d t + \frac{1}{2} \int_{0}^{N \tau} \int_{\R^{d}} \big| \nabla \widetilde{v}_{\tau} \big|^{2} \widetilde{u}_{\tau} \, \d x \, \d t + \mathcal{F}\big(u_{\tau} (N \tau ) \big) \leq  \mathcal{F}(u_{0}). 
\end{align}

Moreover, for any $t \in [0,T]$, there holds
\begin{align}
    \label{eq:secondeq2}
    W^{2}\big( \widetilde{u}_{\tau}(t), u_{\tau}(t) \big) \leq 8 \tau \mathcal{F}(u_{0}).
\end{align}
\end{proposition}

\begin{proof}
The proof of this energy dissipation inequality is classical and we give it here for completeness. First, proceeding as in the proof of Theorem \ref{thm:euler-lagrange-discrete}, we show that
\begin{align}
\label{eq:secondeq3}
      \int_{\R^{d}} \big| \nabla \widetilde{v}^{k}_{\tau} \big|^{2} \widetilde{u}_{\tau} \, \d x = \frac{1}{(t - (k-1)\tau)^{2}} W^{2}\big(\widetilde{u}_{\tau}, u^{k-1}_{\tau} \big).
\end{align}
From \cite[Lemma 3.2.2]{AGS05}, we infer the energy estimate
\begin{align}
    \label{eq:tobesummed}
    \frac{1}{2 \tau} W^{2}(u^{k}_{\tau}, u^{k-1}_{\tau}) + \frac{1}{2} \int_{(k-1)\tau}^{k \tau} \frac{1}{(t - (k-1)\tau)^{2}} W^{2}\big(\widetilde{u}_{\tau}, u^{k}_{\tau} \big) \, \d t + \mathcal{F}(u_{\tau}^{k}) \leq \mathcal{F}(u_{\tau}^{k-1}).
\end{align}
Summing over $k\in{1,\ldots, N}$, we obtain
\begin{align}
    \label{eq:tobesummed-intermediate}
    \frac{1}{2 \tau} \sum_{k=1}^N W^{2}(u^{k}_{\tau}, u^{k-1}_{\tau}) + \frac{1}{2} \sum_{k=1}^N \int_{(k-1)\tau}^{k \tau} \frac{1}{(t - (k-1)\tau)^{2}} W^{2}\big(\widetilde{u}_{\tau}, u^{k}_{\tau} \big) \, \d t + \mathcal{F}(u_{\tau}^{N}) \leq \mathcal{F}(u_{\tau}^{0}).
\end{align}
Substituting Eqs. (\ref{eq:secondeq}, \ref{eq:secondeq3}) into Eq. \eqref{eq:tobesummed-intermediate}, we obtain 
\begin{align*}
    \frac{\tau}{2} \sum_{k=1}^N\int_{\R^{d}} \big| \nabla {v}^{k}_{\tau} \big|^{2} u^{k}_{\tau} \, \d x + \frac{1}{2} \sum_{k=1}^N \int_{(k-1)\tau}^{k \tau} \int_{\R^{d}} \big| \nabla \widetilde{v}^{k}_{\tau} \big|^{2} \widetilde{u}_{\tau} \, \d x \, \d t + \mathcal{F}(u_{\tau}^{N}) \leq \mathcal{F}(u_{\tau}^{0}),
\end{align*}
which yields the first statement, Eq. \eqref{eq:fromprop}, after replacing the sums over $k$ by integrals. The second statement, Eq. \eqref{eq:secondeq2},  follows from the triangle inequality and Proposition \ref{prop:discrete} $(i)$.
\end{proof}

\smallskip

We now give the brief proof of the energy dissipation inequality, Theorem \ref{thm:main-theorem} $(v)$.

\begin{theorem}\label{thm:energy-disspation-est}
Let $u$ be a weak solution of Problem \eqref{eq:main_eqn}, obtained as the limit of the minimizing movement scheme.  Then, there holds the following energy estimate
\begin{align}
    \label{eq:energydis}
    \cF(u(T)) + \int_{0}^{T} \int_{\R^{d}} u(t)  \big| \nabla v (t) \big|^{2} \, \d x \, \d t \leq \cF(u_{0}). 
\end{align}
\end{theorem}
\begin{proof}
With the regularity of $(u_\tau)_\tau$ and $(\nabla v_\tau)_\tau$ established above, we may use the lower semicontinuity result \cite[Theorem 5.4.4]{AGS05} to the limit in Eq. \eqref{eq:fromprop}, as $\tau$ tends to $0$. This way we obtain the energy dissipation inequality, Eq. \eqref{eq:energydis}. 
\end{proof}

\section{Examples of kernels}
\label{sec:examples}
Here, we provide various examples of kernels that satisfy the conditions of our main result. Let us recall the notations 
\begin{align*}
    \psi^{-1}(\xi)=\frac{1}{\psi(\xi)},
    \qquad \widetilde{\psi}(\xi)=|\xi|^2\psi^{-1}(\xi),\qquad \text{and} \qquad \psi^{*}(\xi)=|\xi|^2\psi^{-2}(\xi)=\widetilde{\psi}(\xi)\psi^{-1}(\xi), 
\end{align*}
where ${\psi}(\xi)$ is the symbol of a L\'{e}vy operator $L$, see Theorem \ref{thm:fourier-symbol}.
We will also denote the L\'{e}vy kernel corresponding to the symbol $\widetilde{\psi}$ by $\widetilde{\nu}$, whenever the latter exists. 
\subsection{Concrete examples}
\begin{example}[Standard example] \label{ex:fractional-laplace} For $s\in [0,1)$ consider  $\psi(\xi)=|\xi|^{2s}$. As highlighted in the introduction the corresponding L\'{e}vy kernel $\nu(h) =\frac{C_{d,s}}{2}|h|^{-d-2s}$ is the standard interaction  kernel of the fraction Laplacian $L=(-\Delta)^s$ when $s\in (0,1)$. In the two extreme cases, $s=1$, $\psi(\xi)=|\xi|^{2}$ corresponds to the symbol of usual Laplacian, and $s=0$, $\psi(\xi)=1$ corresponds to the symbol of the identity operator $L=I$. In either case, we have the following. 
\begin{align*}
    \psi(\xi) =|\xi|^{2s}, \quad \psi^{-1}(\xi)=|\xi|^{-2s},  \quad  \widetilde{\psi}(\xi)=|\xi|^{2(1-s)}, \quad \text{and} \quad \psi^*(\xi)= |\xi|^{2(1-2s)}.
\end{align*}
The nonlocal spaces associated to $\psi$ are $\Hnup= H^s(\R^d), \quad \Hnupd=\dot{H}^{s}(\R^d)$, and
\begin{align*}
      \Hnuid= \dot H^{-s}(\R^d), \quad \Hnuw = H^{1-s}(\R^d), \quad  \text{and} \quad  \Hnus = H^{1-2s}(\R^d).
\end{align*}
Moreover, we readily see that $\widetilde{\psi}$ is the symbol associated with the operator $(-\Delta)^{1-s}$ whose kernel is given by $\widetilde{\nu} (h)= \frac{C_{d,1-s}}{2} |h|^{-d-2(1-s)}$. 
\end{example}
\begin{remark}[The case $s=1$]
Let us briefly comment on the particular case $s=1$. Indeed, when $s=1$, we find $\Hnuw= L^2(\R^d)$, thus, the compactness derived from the flow-interchange argument reduces to a uniform $L^2(0,T;L^2(\R^d))$-bound on $u$, which, of course is insufficient for strong compactness. However, the weak compactness provided by the Banach-Alaoglu theorem can be complemented by strong compactness of the velocity field. In some sense, the role of weak and strong compactness when passing to the limit in the product $u \nabla v$, can be swapped. To see this, we observe, that in the particular case $s=1$, the velocity field has additional regularity, and formally
$$
 \|\nabla (-\nabla v)\|_{L^2(0,T;L^2(\R^d))}^2 =\|-D^2 L^{-1}u\|_{L^2(0,T;L^2(\R^d))}^2 = \|- \Delta L^{-1}u\|_{L^2(0,T;L^2(\R^d))}^2 = \|u\|_{L^2(0,T;L^2(\R^d))}^2,
$$
which is bounded by the dissipation term produced by the flow-interchange. Additionally, 
$$
    \|v\|_{L^\infty(0,T;L^2(\R^d))}^2=\|\nabla L^{-1}u\|_{L^\infty(0,T;L^2(\R^d))}^2 \leq  \int_{\R^d} u_0 L^{-1} u_0 \d x,
$$
i.e., the $L^\infty(0,T;L^2(\R^d))$-norm of the velocity field is controlled by the initial energy. Thus, we have a uniform control of $-\nabla L^{-1}u$ in $L^2(0,T;H^{1}(\R^d))$. Furthermore, we get some weak time regularity by observing
\begin{align*}
    \int_{\R^d} \partial_t \nabla L^{-1} u \cdot \varphi \d x 
    &= - \int_{\R^d} \partial_t u \nabla \cdot(L^{-1}\varphi) \d x\\
    &= - \int_{\R^d} \nabla \cdot(u \nabla L^{-1} u ) \nabla \cdot( L^{-1}\varphi) \d x\\
    &= \int_{\R^d} u \nabla L^{-1} u \nabla (\nabla \cdot L^{-1} \varphi) \d x.
\end{align*}
Thus 
\begin{align*}
    \|\partial_t (-\nabla L^{-1}u)\|_{H^{-k}(\R^d)} \leq \|u\|_{L^2(\R^d)} \|\nabla L^{-1} u\|_{L^\infty(0,T;L^2(\R^d))} \|\varphi\|_{L^\infty(0,T;H^k(\R^d))},
\end{align*}
where $k\in \N$ is large enough such that $\|\nabla (\nabla \cdot L^{-1}\varphi)\|_{L^\infty(\R^d)}\leq \|\varphi\|_{H^k(\R^d)}$.
An Aubin-Lions-type argument (cf.\cite{simon1986compact}) gives the local strong compactness of the velocity field which is enough to pass to the limit.
\end{remark}
\begin{example}
Consider the L\'{e}vy operator associated with the kernel 
\begin{align*}
\nu(h)= \frac12 \int_0^\infty e^{-t} G_t(h)\d t= \frac12 \int_0^\infty \frac{ e^{-t}}{(4\pi t)^{d/2}} e^{-\frac{|h|^2}{4t}}\d t. 
\end{align*}
The corresponding L\'{e}vy symbol is given by 
\begin{align*}
\psi(\xi)= 2\int_{\R^d}(1-\cos(\xi\cdot h))\nu(h)\d h= 1- \int_0^\infty e^{-t} \widehat{G}_t(\xi) \d t= \frac{|\xi|^2}{1+|\xi|^2}.
\end{align*}
The symbol $\psi$ clearly satisfies 
$$\frac{1}{2}(1\land|\xi|^2)\leq \psi(\xi)\leq(1\land|\xi|^2),$$ 
and  we also  have 
\begin{align*}
    \text{$\widetilde{\psi}(\xi)= 1+|\xi|^2$,\quad  
    $\psi^{-1}(\xi) = 1+|\xi|^{-2}$, \quad and \quad $\psi^{*}(\xi) = 2+|\xi|^{-2}+|\xi|^2$.}
\end{align*}
Therefore, we get the following functions spaces 
\begin{align*}
    &\text{$\dot{H}^{1}(\R^d)+ L^2(\R^d)\subset \Hnupd$,\quad   $\Hnup = L^2(\R^d)$, \quad \text{and} \quad  $\Hnuid=\dot{H}^{-1}(\R^d)\cap L^2(\R^d)$.}
\end{align*} 
as well as
\begin{align*}
    &\text{$\Hnuw= \Hnuwd= H^{1}(\R^d)$, \quad \text{and} \quad  $\Hnus= \Hnusd= \dot{H}^{-1}(\R^d)\cap H^{1}(\R^d).$}
\end{align*}
\end{example}

\begin{example}
Consider $s\in (0,1]$ and $\psi(\xi) = (|\xi|^2+1)^s-1$. 
This is the symbol of  the L\'{e}vy type denoted $L= (-\Delta +I)^s- I$, usually referred to as \emph{fractional Schr{\"o}dinger operator} or the \emph{fractional relativistic operator}. According to \cite{FF14}, see also \cite{JKS23}, it can be shown that 
\begin{align*}
    Lu(x)= (-\Delta +I)^su(x)- u(x)= S_{d,s}\int_{\R^d} \frac{(u(x)-u(y))}{|x-y|^{\frac{d+2s}{2}}} K_{\frac{d+2s}{2}}(|x-y|)\d y,
\end{align*}
where $K_\beta$ is the modified Bessel function of the second kind\footnote{Modified Bessel function of the second kind are also called modified Bessel function of the third kind as for instance in\cite{ArSm61}.} of order $\beta\in \R$ \cite{ArSm61,EMOT81,Wat95} given as 
\begin{align*}
K_\beta(r)= 2^{-\beta-1} r^\beta\int_0^\infty e^{-t} e^{-\frac{r^2}{4t}} t^{-\beta-1}\d t. 
\end{align*}
In the definition of $L$, the constant 
$$
    S_{d,s}= \frac{2^{-\frac{d-2s}{2}+1} }{\pi^{d/2} |\Gamma(-s)|}= \frac{2^{-\frac{d+2s}{2}+1}C_{d,s}}{\Gamma\big(\frac{d+2s}{2}\big) },
$$  
is a normalizing constant such that  $\widehat{L u}(\xi)= \psi(\xi)\widehat{u}(\xi)$, for all $u\in C_c^\infty(\R^d)$, see, for instance, \cite{FF14, JKS23,ArSm61}. One readily checks that 
\begin{align*}
    \lim_{|\xi|\to \infty}\frac{\psi(\xi)}{|\xi|^{2s}}
    = \lim_{|\xi|\to \infty}\frac{(|\xi|^2+1)^s-1}{|\xi|^{2s}}
    = 1,
    \quad \text{and}\quad 
    \lim_{|\xi|\to 0}\frac{\psi(\xi)}{|\xi|^{2}}
    =\lim_{|\xi|\to0}\frac{(|\xi|^2+1)^s-1}{|\xi|^{2}}
    =s.
\end{align*}
That is $\psi(\xi)\sim |\xi|^{2s}$, as $|\xi|\to \infty$ and $\psi(\xi)\sim |\xi|^{2}$, as $|\xi|\to 0$.  Therefore, there exists some $c>0$ such that 
\begin{align*}
    c^{-1} \min(|\xi|^2, |\xi|^{2s})\leq \psi(\xi)\leq c\min(|\xi|^2, |\xi|^{2s}). 
\end{align*}
It turns out that,  
\begin{align*}
    &\text{$\psi(\xi)\asymp \min(|\xi|^2, |\xi|^{2s})$,\quad  and \quad $\psi^{-1}(\xi)\asymp \max(|\xi|^{-2}, |\xi|^{-2s})$,}
\end{align*} 
as well as
\begin{align*}
    &\text{$\widetilde{\psi}(\xi)\asymp \max(1, |\xi|^{2(1-s)})$,\quad and \quad $\psi^{*}(\xi) \asymp \max(|\xi|^{-2}, |\xi|^{2-4s})$.  }
\end{align*}
Moreover, we have 
\begin{align*}
    \text{$\psi(\xi)\geq c_s (1\land|\xi|^2)$,\quad  with\quad  $c_s=\min(2^s-1, s2^{s-1})$.}
\end{align*}
Indeed, for $|\xi|^2\geq1$ we have $\psi(\xi)\geq (2^s-1)\geq (2^s-1)(1\land|\xi|^2)$,  whereas for $|\xi|^2\leq 1$ we get 
\begin{align*}
    &\psi(\xi)= s\int_0^{|\xi|^2} (1+t)^{s-1}\d t \geq s2^{s-1}|\xi|^2\geq s2^{s-1}(1\land|\xi|^2). 
\end{align*}
The corresponding functions spaces are given by
\begin{align*}
    &\text{$\Hnupd= \dot{H}^{1}(\R^d)+ \dot{H}^{s}(\R^d)$,  \quad  $\Hnup = H^{s}(\R^d)$, \quad  and \quad $\Hnuid=\dot{H}^{-1}(\R^d)\cap \dot{H}^{-s}(\R^d)$.}
\end{align*}
as well as
\begin{align*}
    &\text{$\Hnuw= \Hnuwd= H^{1-s}(\R^d)$, \quad and \quad 
    $\Hnusd= \dot{H}^{-1}(\R^d)\cap \dot{H}^{1-2s}(\R^d).$}
\end{align*}
\end{example}

\begin{example}
Consider $\psi(\xi) = \log (|\xi|^2+1)$. 
This symbol of L\'{e}vy type is associated to $L=\log(-\Delta +I)$, usually referred to as the \emph{logarithmic  Schr\"ondiger operator}, see for instance \cite{Feu23}. More precisely, it can be shown that 
\begin{align*}
    Lu(x)=\log(-\Delta +I) u(x)= S_{d,0} 
    \int_{\R^d} \frac{(u(x)-u(y))}{|x-y|^{\frac{d}{2}}} K_{\frac{d}{2}}(|x-y|)\d y,
\end{align*}
where the constant $S_{d,0}= 2^{-\frac{d}{2}+1} \pi^{-d/2}$  is a normalizing constant such that $\widehat{L u}(\xi)= \psi(\xi)\widehat{u}(\xi)$, for all $u\in C_c^\infty(\R^d)$. Observe that 
\begin{align*}
    \lim_{|\xi|\to \infty}\frac{\psi(\xi)}{\log |\xi|^2}
    = \lim_{|\xi|\to \infty}\frac{\log (|\xi|^2+1)}{\log |\xi|^2}
    = 1,
    \quad \text{and}\quad 
    \lim_{|\xi|\to 0}\frac{\psi(\xi)}{|\xi|^{2}}
    =\lim_{|\xi|\to0}\frac{\log (|\xi|^2+1)}{|\xi|^{2}}
    =1.
\end{align*}
That is $\psi(\xi)\sim \log |\xi|^2$, as $|\xi|\to \infty$ and $\psi(\xi)\sim |\xi|^{2}$, as $|\xi|\to 0$. Therefore, there is $c>0$ such that 
\begin{align*}
c^{-1} \min(|\log |\xi|^2|, |\xi|^{2})\leq \psi(\xi)\leq c\min(|\log |\xi|^2|, |\xi|^{2}). 
\end{align*}
As in the previous example, we have 
\begin{align*}
    \text{$ \psi(\xi)\geq c_0 (1\land|\xi|^2)$ \quad with\quad  $c_0=\min(\log2, 2^{-1}) $}, 
\end{align*}
Moreover, for any $0\leq \eta<1$, there exists $c= c(\eta)>0$, such that
\begin{align*}
    &\text{ $\psi(\xi)\asymp \min(|\log |\xi|^2|, |\xi|^{2})\leq c\min(|\xi|^{2}, |\xi|^{2(1-\eta)})$, \quad and \quad  $\psi^{-1}(\xi)\asymp \max(\log^{-1} |\xi|^2, |\xi|^{-2})$,}
\end{align*}
as well as
\begin{align*}
    &\text{$\widetilde{\psi}(\xi)\asymp \max(1, |\xi|^{2}\log^{-1} |\xi|^2) \geq c^{-1}(|\xi|^{2\eta}+1)$,\quad and \quad $\psi^{*}(\xi) \asymp \max(|\xi|^{-2}, |\xi|^{2}\log^{-2} |\xi|^2)$.}
\end{align*}
Thus, for any $0\leq \eta<1$, we get 
\begin{align*}
    \text{$H^{1-\eta} \subset \Hnup$,\quad and \quad $\Hnuw= \Hnuwd\subset H^{\eta}(\R^d)$. }
\end{align*}

\end{example}

\subsection{Miscellaneous examples}\label{sec:Misc-example}
 The notion of symmetric nonlocal  L\'evy operators  goes  beyond the class considered in this paper. The above examples give us the opportunity to  consider more versatile symmetric nonlocal  L\'evy operators that may apply or not  to our  results. 
 
 (i) In view of the Example \ref{ex:fractional-laplace} one can consider general kernels satisfying  the bounds condition
\begin{align*}
	C^{-1}|h|^{-d-2s_1}\leq \nu(h)\leq C|h|^{-d-2s_2} \quad \quad  \text{for all $|h|\leq 1$},
\end{align*}
for some constant $C>0$ and $s_1,s_2\in(0,1)$. This class includes a particular example  like $$\nu(h)= C_{d,s_1}|h|^{-d-2s_1}+C_{d,s_2}|h|^{-d-2s_2},$$ with $0<s_1\leq s_2<1$, yielding the symbol $\psi(\xi)=|\xi|^{2s_1}+|\xi|^{2s_2}$ and hence the operator $$Lu= (-\Delta)^{s_1}u+(-\Delta)^{s_2} u.$$ 
For this particular case we have $\Hnup= H^{s_2}(\R^d)$ whereas 
$\Hnuw= H^{1-s_1}(\R^d)$.

(ii) Another interesting example in connection to the  Example \ref{ex:fractional-laplace} with $s=0$, is the zero other operator of the form 
\begin{align*}
	Lu(x)= u(x)-\nu*u(x)= \int_{\R^d} (u(x)-u(x+h))\nu(h)\d h 
\end{align*}
where $\nu\in L^1(\R^d)$ is radial with $\|\nu\|_{L^1(\R^d)}=1$. Here we have $\Hnup= L^2(\R^d)$ whereas $\Hnuw= H^1(\R^d)$. Thus, this example  is similar to  Example \ref{ex:fractional-laplace} in the case  $s=0$. 

(iii) Our results also apply  to more general nonlocal operators with singular L\'evy measure like  the anisotropic fractional Laplacian given by 
\begin{align*}
	Lu = (-\partial^2_{x_1x_1})^s u+  \cdots+ (-\partial^2_{x_dx_d})^su&&s\in(0,1),
\end{align*} 
which is a pseudo-differential operator with symbol  
\begin{align*}
	\psi(\xi)= |\xi_1|^{2s}+\cdots+|\xi_d|^{2s}.  
\end{align*}
This symbol satisfies the estimates $d^{-s} |\xi|^{2s}\leq \psi(\xi) \leq d^{1-s} |\xi|^{2s}. $
Therefore, the functions spaces here are the same as those of Example \ref{ex:fractional-laplace}. Interestingly, it can be shown that the operator $Lu = (-\partial^2_{x_1x_1})^s u+  \cdots+ (-\partial^2_{x_dx_d})^su,$ see  for instance \cite[Section 2.6]{guy-thesis}, has the representation 
\begin{align*}
	Lu(x) =(-\partial^2_{x_1x_1})^s u(x)+  \cdots+ (-\partial^2_{x_dx_d})^su(x)= 2\pv \int_{\R^d}(u(x)-u(x+h))\d\nu(h)
\end{align*}
where the L\'evy measure $\d\nu$ is given by 
	\begin{align}\label{eq:mixed-levy-measure}
		\d \nu(h) = \sum_{j=1}^{d}C_{1,s} |h|^{-1-2s}\d h\prod_{i\neq j}\delta_{0_i}(\d h)= \sum_{j=1}^{d} C_{1,s}|h|^{-1-2s}\d h\otimes \delta_{\widetilde{0}_i}(\d \widetilde{h}_i). 
	\end{align}
	where $\widetilde{h}_i= (h_1,\cdots, h_{i-1}, h_{i+1},\cdots, h_{d})$ and $\delta_{\widetilde{0}_i}$  represents the Dirac measure at $\widetilde{0}_i$.

(iv) Although one is tempted to think that results extend to, in general, singular measures, the following counter-example of the discrete operator
\begin{align*}
    Lu(x) = 2u(x) - u(x + a) -u(x -a)= 	 2\pv \int_{\R^d}(u(x)-u(x+h))\d\nu_a(h)
\end{align*}
with $a\in \R^d$ and $\nu_a= \frac12(\delta_a+\delta_{-a})$ shows that this is not the case. It is easy to obtain that $\psi(\xi)= 2(1-\cos(\xi\cdot a))$. 	Here the symbol  $\psi(\xi)= 2(1-\cos(\xi\cdot a))$ is degenerate and fails to satisfy the lower bound condition $\psi(\xi)\geq c(1\land|\xi|^2)$.  Thus, this simple case does not enter the scope of our study.

\vspace{2mm}
\noindent \textbf{Data Availability Statement (DAS)}: Data sharing not applicable, no datasets were generated or analyzed during the current study.

\bibliographystyle{alpha}
\newcommand{\etalchar}[1]{$^{#1}$}

\end{document}